\documentclass[11pt]{article}
\usepackage{hyperref}
\usepackage{amsmath,amssymb,amsthm}
\usepackage[margin=1in]{geometry}
\usepackage{graphicx}
\usepackage{bbm}
\usepackage{booktabs}
\usepackage{breakcites}
\usepackage{bm}
\usepackage[numbers,sort&compress]{natbib}
\usepackage{thm-restate}
\usepackage{color-edits}
\usepackage{thmtools}
\usepackage{mathtools}

\expandafter\let\csname enumerate*\endcsname\relax
\expandafter\let\csname endenumerate*\endcsname\relax
\usepackage[inline, shortlabels]{enumitem}
\usepackage{dsfont}
\usepackage{xspace}
\usepackage[algo2e,linesnumbered,ruled,vlined,algosection]{algorithm2e}
\usepackage{algorithmic}
\makeatletter
\let\c@subfigure\relax 
\let\c@subtable\relax 
\makeatother
\usepackage{subcaption}
\usepackage{array}
\usepackage[dvipsnames]{xcolor}
\usepackage{cleveref}
\usepackage{makecell}

\usepackage{colortbl}

\theoremstyle{plain}
\declaretheorem[name=Theorem,numberwithin=section]{theorem}
\newtheorem{lemma}[theorem]{Lemma}
\newtheorem{proposition}[theorem]{Proposition}

\newtheorem{corollary}[theorem]{Corollary}

\newtheorem{fact}[theorem]{Fact}

\theoremstyle{definition}
\newtheorem{definition}[theorem]{Definition}

\theoremstyle{remark}

\newcommand{\paren}[1]{\left(#1 \right )}

\newcommand{\ceil}[1]{\lceil #1 \rceil}

\newcommand{\nnz}[1]{\mathrm{nnz}(#1)}
\newcommand{\nnzo}{\mathrm{nnz}}

\newcommand{\defeq}{\coloneqq}
\newcommand{\eqdef}{\eqqcolon}

\newcommand{\normInline}[1]{\lVert#1\rVert}

\newcommand{\R}{\mathbb R}

\newcommand{\cD}{\mathcal D}

\newcommand{\cM}{\mathcal M}

\newcommand{\cO}{\mathcal O}
\newcommand{\cP}{\mathcal P}

\newcommand{\cU}{\mathcal U}

\newcommand{\cX}{\mathcal X}

\DeclareMathOperator*{\argmin}{argmin} 
\DeclareMathOperator*{\argmax}{argmax} 
\newcommand{\poly}{{\sf poly}}

\allowdisplaybreaks
\definecolor{violet}{RGB}{148, 0, 211}
\hypersetup{
    colorlinks=true,
    linkcolor=blue!70!black, %
    citecolor=violet, %
    filecolor=blue!70!black, %
    urlcolor=magenta
}

\newcommand{\otilde}{\tilde{O}}
\newcommand{\Otilde}{\otilde}

\SetKwInput{KwInput}{Input}
\SetKwInput{KwOutput}{Output}
\SetKwInput{KwParameter}{Parameter}
\SetKwProg{Function}{function}{}{}
\SetKwIF{IfNot}{ElseIfNot}{}{if not}{then}{else if not}{}{} %

\SetCommentSty{mycommfont}
\SetKwFor{Repeat}{repeat}{}{}
\SetKwRepeat{Do}{do}{while} %

\newcommand{\ellTwoEllOne}{\ell_2\text{-}\ell_1}

\newcommand{\minimize}{\mathrm{minimize}}
\newcommand{\maximize}{\mathrm{maximize}}
\newcommand{\tr}{\mathrm{tr}}

\newcommand{\B}{\mathbb{B}}
\newcommand{\KL}{\mathrm{KL}}

\newcommand{\xset}{\cX}

\newcommand{\simplex}{\Delta}
\usepackage{etoolbox}
\usepackage{xparse}
\DeclarePairedDelimiterXPP{\inangle}[1]{}{\langle}{\rangle}{}{#1}
\DeclarePairedDelimiterXPP{\inbraces}[1]{}{\{}{\}}{}{#1}
\DeclarePairedDelimiterXPP{\inparen}[1]{}{(}{)}{}{#1} %
\DeclarePairedDelimiterXPP{\insquare}[1]{}{[}{]}{}{#1}
\DeclarePairedDelimiter{\innorm}{\|}{\|}
\DeclarePairedDelimiter{\inabs}{\lvert}{\rvert}

\NewDocumentCommand\breg{s m m }{ %
  \IfBooleanTF{#1} %
    { V^{r}_{#2} \left( #3 \right) }
    { V^{r}_{#2} ( #3 ) }
}
\NewDocumentCommand\bregwr{s O{} m m }{ %
  \IfBooleanTF{#1} %
    { V^{#2}_{#3} \left( #4 \right) }
    { V^{#2}_{#3} ( #4 ) }
}

\newcommand{\rx}{r_\mathsf{x}}
\newcommand{\ry}{r_\mathsf{y}}

\NewDocumentCommand\xbreg{s m m }{
  \IfBooleanTF{#1} %
    { \bregwr*[\rx]{#2}{#3} }
    { \bregwr[\rx]{#2}{#3} }
}
\NewDocumentCommand\ybreg{s m m }{
  \IfBooleanTF{#1} %
    { \bregwr*[\ry]{#2}{#3} }
    { \bregwr[\ry]{#2}{#3} }
}

\newcommand{\x}{_\mathsf{x}}
\newcommand{\y}{_\mathsf{y}}

\newcommand{\grad}{\nabla}

\DeclarePairedDelimiterXPP{\dualnorm}[1]{}{\|}{\|}{_{*}}{#1}

\newcommand{\overle}[1]{\overset{#1}{\le}}

\newcommand{\diag}{\mathrm{diag}}

\newcommand{\ball}{\B}

\newcommand{\hess}{\nabla^2}

\newcommand{\inKL}[2]{\KL(#2||#1)} %

\newcommand{\xtilde}{\tilde{x}}

\newcommand{\epsprim}{\epsilon'}

\newcommand{\psdcone}{\mathbb{S}_{+}}
\newcommand{\pdcone}{\mathbb{S}_{++}}
\newcommand{\symm}{\mathbb{S}}

\usepackage{tikz}

\newcommand{\smax}{\mathrm{smax}}

\newcommand{\xhat}{\hat{x}}

\newcommand{\Omegatilde}{\tilde{\Omega}}

\newcommand{\thirdd}{\nabla^3}

\newcommand{\xbar}{\bar{x}}

\newcommand{\ghat}{\hat{g}}

\newcommand{\lambdaopt}{\lambda^\star}

\newcommand{\Oball}{\mathcal{O}_{\mathrm{ball}}}

\newcommand{\xopt}{x^\star}

\newcommand{\Lagrange}{\mathcal{L}}

\newcommand{\qtilde}{\tilde{q}}

\newcommand{\projt}{\mathrm{proj}}

\newcommand{\distt}{\mathrm{dist}}

\newcommand{\oraclet}{\mathcal{O}}

\newcommand{\lambdatilde}{\tilde{\lambda}}

\newcommand{\seeddist}{\mathcal{P}_{\mathrm{seed}}}

 \usepackage{microtype}
  \usepackage{nth}
  \usepackage{intcalc}

  \newcommand{\cSTOC}[1]{\nth{\intcalcSub{#1}{1968}}\ Annual\ ACM\ Symposium\ on\ Theory\ of\ Computing\ (STOC)}

  \newcommand{\cFOCS}[1]{\nth{\intcalcSub{#1}{1959}}\ Annual\ IEEE\ Symposium\ on\ Foundations\ of\ Computer\ Science\ (FOCS)}

  \newcommand{\cCOLT}[1]{\nth{\intcalcSub{#1}{1987}}\ Annual\ Conference\ on\ Computational\ Learning\ Theory\ (COLT)}

  \newcommand{\cSODA}[1]{\nth{\intcalcSub{#1}{1989}}\ Annual\ ACM-SIAM\ Symposium\ on\ Discrete\ Algorithms\ (SODA)}
  \newcommand{\cNIPS}[1]{Advances\ in\ Neural\ Information\ Processing\ Systems\ \intcalcSub{#1}{1987} (NeurIPS)}

  \newcommand{\cAAAI}[1]{AAAI\ Conference\ on\ Artificial (AAAI)}

\title{Fast, Parallel, Query-Efficient Binary Classification}

\author{%
    Ishani Karmarkar\thanks{Stanford University, \texttt{\string{ishanik,ocarroll,sidford\string}@stanford.edu}} 
    \and
    Liam O'Carroll\footnotemark[1] 
    \and
    Aaron Sidford\footnotemark[1] 
}

\begin{document}

\pagenumbering{gobble}

\maketitle

\begin{abstract}
We study the fundamental classification problem of computing a separating hyperplane for a binary-labeled dataset of size $n$ with normalized $d$-dimensional features. Letting $\Phi \in \mathbb{R}^{n \times d}$ denote the feature matrix and $\gamma$ the margin of the maximum-margin separating hyperplane, we present a randomized algorithm that solves this problem in $\tilde{O}(\gamma^{-2/3}\, \operatorname{nnz}(\Phi) + \gamma^{-2(\omega+1)/3})$-sequential running time (work), $\tilde{O}(\gamma^{-2/3})$-parallel (computational) depth, and accesses $\Phi$ only through $\tilde{O}(\gamma^{-2/3})$-matrix-vector queries (matvecs). We also present a second, faster randomized algorithm with a $\tilde{O}(\gamma^{-2/3}\, \operatorname{nnz}(\Phi) + \gamma^{-2})$-sequential running time that uses $\tilde{O}(\gamma^{-2/3})$-matvecs to $\Phi$, but achieves only $\tilde{O}(\gamma^{-4/3})$-parallel depth. Both algorithms match the near-optimal deterministic matvec complexity recently established by \citep{kornowski2024oracle, karmarkar2026solving} and achieve improved sequential runtime and parallel depth, albeit at the expense of using randomness.
\end{abstract}

\setcounter{tocdepth}{1}
\tableofcontents \clearpage

\pagenumbering{arabic}

\section{Introduction}\label{sec:intro}

In this paper, we study the foundational binary data classification problem of finding a linear separator for two sets of points. Concretely, we consider the \emph{separating hyperplane}, or \emph{hard-margin support vector machine (SVM) problem}, in which we are given a dataset $\cD = \{(\phi_i, l_i)\}_{i \in [n]}$ where $\phi_i \in \ball^d \defeq \{ x \in \R^d : \innorm{x}_2 \le 1 \}$ is a normalized feature vector and $l_i \in \{+1, -1\}$ is a binary label for each $i \in [n]$. We let $\gamma_\cD$ denote the maximum margin of a separating hyperplane for $\cD$, i.e., 
\begin{align}\label{eq:linear-separability}
    \gamma_\cD \defeq \max_{w \in \ball^d} \min_{i \in [n]} l_i \langle w, \phi_i \rangle, ~~ \text{and hence,} ~~ \exists w^\star \in \ball^d \text{ such that } l_i \langle w^\star, \phi_i \rangle \geq \gamma_{\cD} ~ \text{for all $i \in [n]$}. 
\end{align}
We consider the problem where, given $\rho > 0$, we must find a hyperplane $\hat{w} \in \ball^d$ which achieves a margin of $\gamma_\cD - \rho$.\footnote{For simplicity, as in prior work \citep{kornowski2024oracle, karmarkar2026solving, karmarkar2025solvingzerosumgames, carmon2020coordinate, carmon2024whole, carmon2019variance}, in this paper we focus on the setting of \eqref{eq:linear-separability}, which assumes linear separability through the origin with maximum margin $\gamma_\cD$. However, using standard reductions, \eqref{eq:linear-separability} and correspondingly Definition~\ref{def:separating-hyper-plane-problem} can be relaxed to assume linear separability under arbitrary affine hyperplanes.}
We formalize this problem in the following \Cref{def:separating-hyper-plane-problem}.

\begin{definition}[Maximum Margin Separating Hyperplane Problem]\label{def:separating-hyper-plane-problem} In the \textit{$\rho$-(maximum margin) separating hyperplane problem}, we are given a (binary-labeled) \emph{dataset} $\cD = \{(\phi_i \in \ball^d, l_i \in \{-1, +1\})\}_{i \in [n]}$ and $\rho > 0$, and must output $\hat{w} \in \ball^d$ such that $l_i \langle \hat{w}, \phi_i \rangle \geq \gamma_\cD - \rho$ for all $i \in [n]$. We say that such a $\hat{w}$ \emph{induces a $\rho$-separating} hyperplane for $\cD$. 
\end{definition}
 
This is an incredibly well-studied problem in machine learning, dating back to \citep{rosenblatt1958perceptron, mcculloch1943logical}, with numerous applications (especially as it can be extended to non-linear settings \cite{murty2016kernel}). Our focus in this work is to introduce new, faster randomized algorithms for this foundational learning problem. 

\paragraph{Our results.} For notational convenience, for a dataset $\cD = \{ (\phi_i \in \ball^d, l_i \in \{+1, -1\} )\}_{i \in [n]}$ we use $\Phi_\cD \in \R^{n \times d}$ to denote the covariate matrix whose $i$-th row is given by $l_i \phi_i^\top$. In addition, for $A \in \R^{n \times d}$,  $\nnz A \defeq |\{(i,j) : A_{ij} \neq 0\}| + n + d$ denotes an augmented nonzero count.

Our main result, given in Theorems~\ref{thm:first-result} and~\ref{thm:second-result} below, is two parallel, query-efficient algorithms for the maximum-margin separating hyperplane problem (Definition~\ref{def:separating-hyper-plane-problem}) which offer different trade-offs between complexity metrics. The complexity metrics we study are \textit{depth}, \textit{work} (or sequential runtime), and \textit{matvec} complexity. 
\begin{itemize}
    \item \emph{Depth and work}: Following \cite{jambulapati2024closing}, we say an algorithm has \emph{(computational) depth} $D$ if the number of sequential rounds of computation is $D$.  In particular, we assume element-wise vector operations (e.g., adding/scaling vectors) in $\R^k$ incur $O(1)$-depth and $O(k)$-work and that dot products and matrix-vector multiplications require $O(\log k)$ depth. We let $\omega < 2.3714$ denote the fast matrix multiplication (FMM) constant \citep{alman2025more}; namely, two $k \times k$ matrices can be multiplied in $O(k^\omega)$-work, as well as $\Otilde(1)$-depth \cite{pan1985efficientparallel,pan1987parallelmatrix}.
    \item \emph{Matvec complexity}: Formally, we say that an algorithm can be implemented using $T$ matvecs to a matrix $A \in \R^{n \times d}$ if it can be implemented using $T$ queries to a \emph{matvec oracle for $A$}, which for any query $(x, y) \in \R^{d} \times \R^{n}$ outputs $(A^\top y, A x)$.
\end{itemize}

In the following Theorems~\ref{thm:first-result} and~\ref{thm:second-result} (and throughout the paper) we use $\tilde{O}(\cdot)$ to suppress polylogarithmic factors in $n$, $d$, $\rho^{-1}$, and $\epsilon^{-1}$ (which appears later). \Cref{thm:first-result} corresponds to our most parallelizable algorithm, notably achieving $\tilde{O}(\rho^{-2/3})$-depth.

\begin{restatable}[Fast, parallel binary data classification]{theorem}{fastparallelbinaryclass}\label{thm:first-result} There is a randomized algorithm which, with probability at least $2/3$, solves the $\rho$-separating hyperplane problem and runs in
\[
\tilde{O}(\rho^{-2/3}\, \nnz {\Phi_\cD} + \rho^{-2(\omega +1)/3})\text{-work and }
\tilde{O}(\rho^{-2/3})
\text{-depth}\,.
\] 
Moreover, the algorithm can be implemented with $\tilde{O}(\rho^{-2/3})$-matvecs to $\Phi_\cD$.
\end{restatable}

Interestingly, our next result shows that the $\poly(\rho^{-1})$ dependence in Theorem~\ref{thm:first-result} can be further improved (unless $\omega = 2$), albeit at the cost of increased depth.

\begin{restatable}[Faster binary data classification]{theorem}{fasterbinaryclass}\label{thm:second-result} There is a randomized algorithm which, with probability at least $2/3$, solves the $\rho$-separating hyperplane problem and runs in
\[
\tilde{O}(\rho^{-2/3}\, \nnz {\Phi_\cD} + \rho^{-2})\text{-work and }
\tilde{O}(\rho^{-4/3})
\text{-depth}\,.
\] Moreover, the algorithm can be implemented with $\tilde{O}(\rho^{-2/3})$-matvecs to $\Phi_\cD$.
\end{restatable}

Next, we compare Theorems~\ref{thm:first-result} and \ref{thm:second-result} to the prior art with respect to each of the three complexity metrics, work, depth, and matvecs, that we study.

\paragraph{Total work.} Table~\ref{table:complexities} summarizes our results in this metric compared to the prior art. Recalling $\nnzo(\Phi_{\cD}) = \Omega(n + d)$, we have the following comparisons: The work achieved in \Cref{thm:second-result} improves the $(n + d) \rho^{-2}$ complexity of the second row in the moderate-to-high precision or sparse regime $\rho \ll (\frac{n + d}{\nnzo(\Phi_{\cD})})^{3/4}$. \Cref{thm:second-result} improves upon the $\nnz {\Phi_\cD}+ \sqrt{ \nnz {\Phi_\cD} \cdot (n + d) } \rho^{-1}$ work of the third row (and therefore also the first row) in the sparse, moderate precision regime
\begin{align*}
    \frac{1}{\sqrt{ \nnzo(\Phi_{\cD}) (n + d) }} \ll \rho \ll \inparen*{\frac{n + d}{\nnzo(\Phi_\cD)}}^{3/2} \, .
\end{align*}
Assuming $d \gg 1$, \Cref{thm:second-result} improves upon the $nd + nd^{2/3}\rho^{-2/3} + d\rho^{-2}$ work of the fourth row in the sufficiently sparse regime $\nnzo(\Phi_\cD) \ll n d^{2/3}$ or the moderate-to-high precision regime $\rho \ll (\frac{d}{\nnzo(\Phi_\cD)})^{3/4}$. Importantly, we improve upon \emph{all} prior art in the sparse, low-to-moderate precision regime where $\nnzo(\Phi_\cD) \approx n + d$, $\frac{1}{n + d} \ll \rho \ll 1$, and $d \gg 1$. This regime includes $\rho = 1 / \sqrt{n}$, which is standard in empirical risk minimization problems where statistical noise limits meaningful precision to $1 / \sqrt{n}$.

We note that additional total work complexities may be obtained using interior point methods (IPMs), e.g., \citep{cohen2021solving, van2021minimum}, which may improve on Theorems~\ref{thm:first-result} and~\ref{thm:second-result} in certain high-accuracy regimes (where $\rho$ is very small). However, current IPMs have at least a \textit{quadratic} dependence on $\min \{ n, d \}$, unlike our methods and most of the remaining methods in Table~\ref{table:complexities} (in sufficiently sparse regimes). As we do not focus on this high-accuracy regime, for simplicity, we do not compare in detail against IPMs and defer the reader to, e.g., \citep{carmon2024whole} for more details.

\begin{table}[h]
\renewcommand{\arraystretch}{1.25}
   \centering
    \begin{tabular}{
      @{}
      >{\centering\arraybackslash}m{8.7cm}
      >{\centering\arraybackslash}m{5.9cm}
      @{}
    }

   \toprule
    Method  & Total work  \\ \midrule
   (Exact) gradient methods  \citep{nem04, Nesterov2007dualextrapolation} & $\nnz {\Phi_\cD}\rho^{-1}$ 
   
   \\
    Row-col randomized method \citep{grigoriadis1995sublinear, palaniappan2016stochastic, clarkson2012sublinear} & $(n + d) \rho^{-2}$
    \\
    Row-col randomized method with variance-reduction \citep{palaniappan2016stochastic, carmon2019variance} & $\nnz {\Phi_\cD}+ \sqrt{ \nnz {\Phi_\cD} \cdot (n + d) } \rho^{-1}$
    \\
   Primal ball-accelerated stochastic methods  \citep{carmon2024whole} & $nd + nd^{2/3}\rho^{-2/3} + d\rho^{-2}$
    \\
    \rowcolor[HTML]{EFEFEF}
   Theorem~\ref{thm:first-result} & $\nnz {\Phi_\cD}\rho^{-2/3} + \rho^{-2(\omega + 1)/3}$ 
    \\
   \rowcolor[HTML]{EFEFEF}
   Theorem~\ref{thm:second-result} & $\nnz {\Phi_\cD}\rho^{-2/3} + \rho^{-2}$ \\
   \bottomrule
   \end{tabular}
   \caption{\label{table:complexities} Comparison to prior work. The table shows the total work complexities of the prior art, up to polylogarithmic factors in $n, d$ and $\rho^{-1}$. The complexities in the abstract follow from $\rho =\gamma_\cD/2$.}
\end{table}

\paragraph{Depth.} Mirror descent \cite{nemirovskij1983problem,beck2003mirrordescent} achieves $\Otilde(\rho^{-2})$-depth for this problem, and a variety of deterministic algorithms (accelerated gradient descent, mirror prox, dual extrapolation) achieve an improved $\Otilde(\rho^{-1})$-depth \cite{nesterov2005smooth,nem04,Nesterov2007dualextrapolation}. We are not aware of any prior deterministic or randomized algorithms which improve upon $\Otilde(\rho^{-1})$-depth in the nearly dimension-free parallel regime we study.

Note that although \citep{karmarkar2025solvingzerosumgames, karmarkar2026solving} achieve a $\tilde{o}(\rho^{-1})$-matvec complexity for the problem, they do not immediately yield improved parallel depths for the problem. The algorithms of \citep{karmarkar2025solvingzerosumgames, karmarkar2026solving} perform $\tilde{O}(\rho^{-2/3})$ \textit{outer loop iterations}; however, each outer loop iteration in turn requires solving multiple constrained,
convex optimization sub-problems to high accuracy. Consequently, the work and depth complexities
of their inner loop itself scale polynomially in $n, d, \rho^{-1}$. Due to the complex nature of
these sub-problems, it is perhaps not straightforward to bound this polynomial; \citep{karmarkar2025solvingzerosumgames, karmarkar2026solving} do not bound the work or depth complexity of
their method and leave this as a direction for future work.

We also note that interior point methods (IPMs) or other existing parallel algorithms \cite{nemirovski1994parallel,duchi2012randomized,bubeck2019highlyparallel,carmon2023resqueing,jambulapati2024closingcomputationalquerygap} have \emph{dimension-dependent} parallel complexities which may achieve improvements in high-accuracy regimes, which are not the focus of this paper.

\paragraph{Matvec complexity.} Importantly, our algorithms apply even in the setting where $\Phi_{\cD}$ is unknown, but can be accessed with a matvec oracle. Until recently, the state-of-the-art matvec complexity for the separating hyperplane problem---among both deterministic and randomized algorithms---was $\tilde{O}(\rho^{-1})$ \citep{Nesterov2007dualextrapolation, nem04, Rakhlin2013online}. However, very recently, a line of work \citep{kornowski2024oracle, karmarkar2025solvingzerosumgames, karmarkar2026solving} studied the matvec complexity of a broader class of problems, known as $\ell_2$-$\ell_1$ and $\ell_1$-$\ell_1$ \textit{matrix games}, the former of which encompasses the separating hyperplane problem (as we discuss further in Section~\ref{sec:approach}). These works show that $\tilde{\Theta}(\rho^{-2/3})$ matvecs to $\Phi_\cD$ is necessary and sufficient for deterministic algorithms. 

While these works settle the deterministic matvec complexity of the separating hyperplane problem (up to polylogarithmic factors), \citet{karmarkar2025solvingzerosumgames, karmarkar2026solving} do not analyze other notions of computational complexity. In other words, obtaining faster algorithms or algorithms with improved depth were left as directions for future research. Hence, it is perhaps natural to wonder whether these recent information theoretic improvements come at the cost of greater work. 

Our results indicate that this is not necessarily the case (at least for randomized algorithms). Indeed, Theorems~\ref{thm:first-result} and~\ref{thm:second-result} demonstrate that, if one is willing to use randomness, then $\Otilde(\rho^{-2/3})$ matvec complexities can be obtained at the cost of only $\poly(\rho^{-1})$-overhead in work.

\paragraph{Additional related work on matrix games.} As mentioned above, the separating hyperplane problem can be reduced to the problem of computing an $\epsilon$-solution of an {$\ell_2$-$\ell_1$ matrix game}, which we define formally in Section~\ref{sec:approach}. Our work builds upon a rich line of work on matrix games more broadly \citep{carmon2019variance, carmon2020coordinate, carmon2024whole, karmarkar2025solvingzerosumgames, karmarkar2026solving, kornowski2024oracle, kornowski2024oracleupdated}. As our focus is on the separating hyperplane problem, Table~\ref{table:complexities} compares against works which make the natural normalizing assumption that each feature vector has Euclidean norm at most 1 (i.e., $\normInline{\Phi_\cD}_{2 \to \infty} = \max_{i \in [n]} \normInline{\phi_i}_2 \le 1$). However, \citep{carmon2020coordinate, carmon2019variance} obtain alternative total work complexities under alternative normalization assumptions on the underlying data matrix $\Phi_\cD$. Additionally, alternative total work complexities can be obtained when the rows and columns of $\Phi_{\cD}$ are uniformly sparse, and we defer to \citep{carmon2020coordinate, clarkson2012sublinear} for further discussion. 

\paragraph{Paper organization.} We provide notation in Section~\ref{sec:prelim} and an overview of our approach in \Cref{sec:approach}. In \Cref{sec:reducing-ell_2-ell_1-to-linear-solving}, we leverage ball acceleration techniques \citep{carmon2020acceleration,carmon2021thinking, carmon2023resqueing, carmon2024whole} to show that the separating hyperplane problem can be reduced to solving a sequence of regularized linear systems to high accuracy. In \Cref{sec:linear-system-reductions}, we discuss how to solve these linear systems efficiently, leveraging subspace embeddings and preconditioned iterative solvers \citep{nelson2013osnap,cohen2016nearly,chenakkod2024optimal,chenakkod2024optimal2,derezinski2025faster,derezinski2026approaching}. In \Cref{sec:sample_reuse}, we show how to leverage the sample reuse framework of \citep{jin2025reusingsamplesvariancereduction} to further improve efficiency. We conclude in \Cref{sec:conclusion}.

\section{Preliminaries}\label{sec:prelim}

\paragraph{General notation.} For $x \in \R^d$, $\normInline{x}_p$ is its $\ell_p$ norm and $x_i$ or $[x]_i$ its $i$-th entry (we may use the latter in particular when there are multiple subscripts). We let $\ball^d_r(\xbar) \defeq \inbraces{x \in \R^d : \innorm{x - \xbar}_2 \le r}$ denote the Euclidean ball of radius $r > 0$ centered at $\xbar \in \R^d$. We may drop $r$ if $r = 1$ or $\xbar$ if $\xbar = 0$ for brevity. Given $S \subseteq \R^d$, $f: \R^d \to \R$, and $\epsilon > 0$, we say that $x \in S$ is $\epsilon$-optimal or an $\epsilon$-minimizer of $f$ over $S$ if $f(x) - f(x') \le \epsilon$ for all $x' \in S$. For vectors $a, b \in \R^n$ and $r \geq 1$, we use $a \approx_r b$ to denote that $r^{-1} [b]_i \leq [a]_i \leq r [b]_i$ for all $i \in [n]$.

\paragraph{Matrix notation.} For $A \in \R^{n \times d}$ of rank $r$, we let $\sigma_1(A) \geq \sigma_2(A) \geq \cdots \geq \sigma_r(A) \geq \sigma_{r+1}(A) = \cdots = \sigma_d(A) = 0$ denote its singular values in descending order. $0_d$ and $0_{m \times n}$ denote the $d$-dimensional vector of zeros and $m \times n$ matrix of zeros respectively and $I_{d}$ denotes the $d \times d$ identity matrix; we may drop the subscript $d$ when it is clear from context. For $v \in \R^d$, $\diag(v) \in \R^{d \times d}$ denotes the diagonal matrix whose $(i,i)$-th entry is $[v]_i$. 

We use $\symm^{d}  \defeq \{A \in \R^{d \times d}: A = A^\top\}$ to denote the set of $d \times d$ symmetric matrices. For $A \in \symm^d$, we say that $A$ is positive semi-definite (PSD), denoted $A  \succeq 0_{d \times d}$, if for all $x \in \R^d$, $x^\top A x \geq 0$; if, moreover, $x^\top A x > 0$ for all $x \in \R^d_{\neq 0}$, we say that $A$ is positive definite (PD), denoted $A \succ 0_{d \times d}$. Correspondingly, we define $\psdcone^{d}  \defeq \{A \in \symm^{d}  : A \succeq 0_{d \times d}\}$ and $\pdcone^{d} \defeq \{A \in \symm^{d}  : A \succ 0_{d \times d}\}$. 

For $M \in \pdcone^{n}$ and $x \in \R^n$, we let $\normInline{x}_M \defeq \sqrt{x^\top M x}$ denote the $M$-norm of $x$. For $A, B \in \symm^{d} $, $A \succeq B$ (respectively $A \succ B$) denotes that $A - B \succeq 0_{d \times d}$ (respectively $A - B \succ 0_{d}$) and for $c \geq 1$, $A \approx_c B$ denotes that $c^{-1} A \preceq B \preceq c A$. For any $A \in \psdcone^n$, we use $\lambda_{\min}(A)$ and $\lambda_{\max}(A)$ to denote its smallest and largest eigenvalues, respectively and 
$\kappa(A) \defeq \frac{\lambda_{\max}(A)}{\lambda_{\min}(A)}$ to denote its \textit{condition number}. 

We use the following standard definition of a linear system solver for a PD matrix. 

\begin{definition}[Linear system solutions and solvers]\label{def:linearsystemsolver}  For  $M \in \pdcone^{d} $, $\epsilon \geq 0$, $\hat{x} \in \R^d$ is an \emph{$\epsilon$-approximate solution} to the linear system $Mx = b$ if $\normInline{\hat{x} - x}_{M} \leq \epsilon \normInline{x}_M$. Moreover, a (randomized) oracle $\cO$ is an \emph{$(\epsilon, \delta)$-linear system solver} for $M$ if for any $b \in \R^d$, $\hat{x} \gets \cO(b)$ is, with probability at least $1-\delta$, an $\epsilon$-approximate solution to $Mx = b$. 
\end{definition} %
\section{Technical overview}\label{sec:approach}

In this section, we describe our approach and motivate our techniques based on prior work.

\subsection{Reduction to $\ell_2$-$\ell_1$ matrix games}\label{subsec:reduce_to_l1l2} The first step to obtain our results is to leverage that the $\rho$-separating hyperplane problem (Definition~\ref{def:separating-hyper-plane-problem}) can be reduced to a class of problems known as $\ell_2$-$\ell_1$ \textit{matrix games }(henceforth referred to simply as \textit{games}, for brevity) \citep{carmon2019variance, carmon2020coordinate, carmon2024whole, kornowski2024oracle, karmarkar2026solving}.

\begin{definition}[($\ell_2$-$\ell_1$ matrix) game]\label{def:l2l1_game} In an \emph{$\epsilon$-($\ell_2$-$\ell_1$ matrix) game}, we are given a matrix $A \in \R^{n \times d}$ such that $\normInline{A}_{2 \to \infty} \defeq \max_{i \in [n]} \normInline{A_{i, :}}_2 \leq 1$ and $\epsilon > 0$ and must compute $\hat{x} \in \ball^d$ such that
\begin{align*}
    \max_{y \in \Delta^n} y^\top A \hat{x} \leq \min_{x \in \mathbb{B}^d} \max_{y \in \Delta^n} y^\top A x + \epsilon. 
\end{align*}
We call such an $\hat{x}$ a \textit{(primal) solution} of the $\epsilon$-game of $A$.\footnote{For simplicity, as our focus is on the separating hyperplane problem (Definition~\ref{def:separating-hyper-plane-problem}), which only requires finding a hyperplane $\hat{w}$, we focus on finding a near-optimal primal variable $\hat{x}$ for the minimax problem $\min_{x \in \mathbb{B}^d} \max_{y \in \Delta^n} y^\top A x$. However, using the techniques of \citep{carmon2024extracting}, our method could be extended to extract near-optimal dual variables $\hat{y}$.}
\end{definition}

The following \Cref{lemma:reduction} shows that the separating hyperplane problem (Definition~\ref{def:separating-hyper-plane-problem}) can be reduced to solving an appropriate $\rho$-game (Definition~\ref{def:l2l1_game}); this result is well-known, however, we include a proof in Appendix~\ref{sec:extra} for completeness.

\begin{restatable}[Reduction to $\ell_2$-$\ell_1$ matrix game]{lemma}{reduction}\label{lemma:reduction}

Let $\cD = \{ ( \phi_i \in \ball^d, l_i \in \{+1, -1\} ) \}_{i \in [n]}$. Any solution $\hat{w}$ to the $\rho$-game (Definition~\ref{def:l2l1_game}) of $-\Phi_\cD$ induces a $\rho$-separating hyperplane (Definition~\ref{def:separating-hyper-plane-problem}) for  $\cD$.
\end{restatable}

The algorithmic techniques we develop in this paper apply to $\ell_2$-$\ell_1$ matrix games more broadly. 
Consequently, in the remainder of the paper we focus on designing efficient algorithms for the more general problem of solving an $\epsilon$-game.
In the rest of this section, we fix $\epsilon > 0, A \in \R^{n \times d}$ and discuss our approach for solving the corresponding $\epsilon$-game (Definition~\ref{def:l2l1_game}). 

\subsection{Motivation}\label{subsec:motivate} Our work is motivated in part by recent algorithmic advancements of \cite{karmarkar2025solvingzerosumgames,karmarkar2026solving}, which culminated in a near-optimal \citep{kornowski2024oracle} $\tilde{O}(\epsilon^{-2/3})$-matvec complexity algorithm for solving an $\epsilon$-game (Definition~\ref{def:l2l1_game}). At a high level, \cite{karmarkar2025solvingzerosumgames,karmarkar2026solving} develop a primal-dual analog of accelerated ball-constrained optimization \citep{carmon2020acceleration, carmon2024whole, carmon2023resqueing, hilal2021stochasticbiasreduced} to reduce solving an $\epsilon$-game to a sequence of $\tilde{O}(\epsilon^{-2/3})$ constrained, regularized sub-problems of the form
\begin{align}\label{eq:hessian-stable-region}
\min_{x \in \mathbb{B}^d} \max_{y \in \Delta^n} y^\top A x + \alpha^{(t)} \sum_{x' \in \cU\x} \frac{1}{2} \innorm{x - x'}_2^2 - \alpha^{(t)} \sum_{y' \in \cU\y} \inKL{y'}{y}
\end{align}
to high accuracy, 
for carefully chosen finite (multi)sets of \textit{centers} $\cU\x \subset \ball^d, \cU\y \subset \simplex^n$. Here, $\inKL{y'}{y} \defeq \sum_{i \in [n]} [y]_i \log \frac{[y]_i}{[y']_i}$ denotes the KL divergence.

To solve \eqref{eq:hessian-stable-region} deterministically and query-efficiently (i.e., with few matvecs to $A$), \citep{karmarkar2025solvingzerosumgames,karmarkar2026solving} show that each subproblem of the form \eqref{eq:hessian-stable-region} can be further reduced to $\tilde{O}(1)$ carefully constructed \textit{constrained subproblems}, each corresponding to a constrained version of \eqref{eq:hessian-stable-region}. They show that within the constrained region, the primal-dual objective is approximated by a quadratic up to a multiplicative constant. Their algorithm leverages this property to either build a low-rank approximation to $A$ or make optimization progress on \eqref{eq:hessian-stable-region}. The former enables the latter to be achieved with fewer matvecs, and ultimately the algorithm of \cite{karmarkar2026solving} uses (amortized) $\Otilde(1)$ matvecs to $A$ per iteration across $\Otilde(\epsilon^{-2/3})$ iterations.

More concretely, their methods carefully build and maintain an explicit, deterministic $\tilde{O}(\epsilon^{-2/3})$-rank \emph{model} $M_c$ for the matrix $Y_c^{1/2} A$, where $Y_c$ is a carefully chosen diagonal matrix, so that $Y_c^{1/2}A - M_c$ has a smaller operator norm than $Y_c^{1/2} A$, enabling more efficient optimization progress when combined with composite proximal methods. By using matvec queries to $A$ to maintain this model as the center $y_c$ changes, their method converges to a high-accuracy solution of \eqref{eq:hessian-stable-region} using an amortized $\tilde{O}(1)$ matvecs to $A$ per subproblem of the form \eqref{eq:hessian-stable-region}.

Importantly, although the model-maintenance procedure in \citep{karmarkar2026solving} achieves the optimal matvec complexity in $A$ and avoids the use of randomness, they require additional computation. Hence, while their method enables near-optimal information theoretic complexity, it does not immediately imply improved work or depth. Moreover, the models $M$ maintained in their method may be \emph{dense} even if the original $A$ is sparse, and consequently, even reading the model $M$ may naively require $\tilde{O}((n+d)\epsilon^{-2/3})$-work. Thus it is perhaps unclear how to directly implement their method and obtain total work complexities scaling with $\nnz{A}$. 

Our key insight is that if one is willing to use randomness, there is a natural alternative to the \citep{karmarkar2026solving, karmarkar2025solvingzerosumgames}'s low-rank model-building procedure described above. Indeed, \citep{karmarkar2025solvingzerosumgames} allude to a version of this alternative approach in their introduction, albeit in a different context and towards a worse complexity bound. To motivate this idea, recall from above that in the method of \citep{karmarkar2026solving, karmarkar2025solvingzerosumgames}, the objective of each constrained sub-problem behaves essentially like a quadratic function over the constrained region. While \citep{karmarkar2025solvingzerosumgames,karmarkar2026solving} leverage this property to build a deterministic algorithm, an alternative approach would be to leverage randomized linear-algebraic techniques such as oblivious subspace embeddings \citep{sarlos2006improved, cohen2015optimal, clarkson2017low, nelson2013osnap} and preconditioned iterative linear system solvers \citep{golub2007chebyshev, golub1988convergence} to efficiently solve these constrained sub-problems \eqref{eq:hessian-stable-region}. Moreover, as we discuss in Section~\ref{subsec:sample-reuse} (and Section~\ref{sec:sample_reuse}), this approach is amenable to the \emph{sample reuse framework} of \citep{jin2025reusingsamplesvariancereduction}, which allows us to reuse the same subspace embeddings across all constrained sub-problems. Below, we formalize how this intuition motivates a straightforward \textit{primal-only} approach to the problem that matches the matvec complexity of \citep{karmarkar2025solvingzerosumgames, karmarkar2026solving} while improving over the total work complexities of the prior art (recall Table~\ref{table:complexities}), albeit using randomness.

In the following sections, we discuss this approach in greater detail, and provide a sketch of the analysis which enables Theorems~\ref{thm:first-result} and~\ref{thm:second-result}. As our method requires several linear-algebraic reductions, to anchor the discussion we include Figure~\ref{fig:approach} which summarizes the end-to-end implementation of our approach and analysis and the various reductions which enable it. 

\begin{figure}[t]
    \centering
    \includegraphics[width=.95\linewidth]{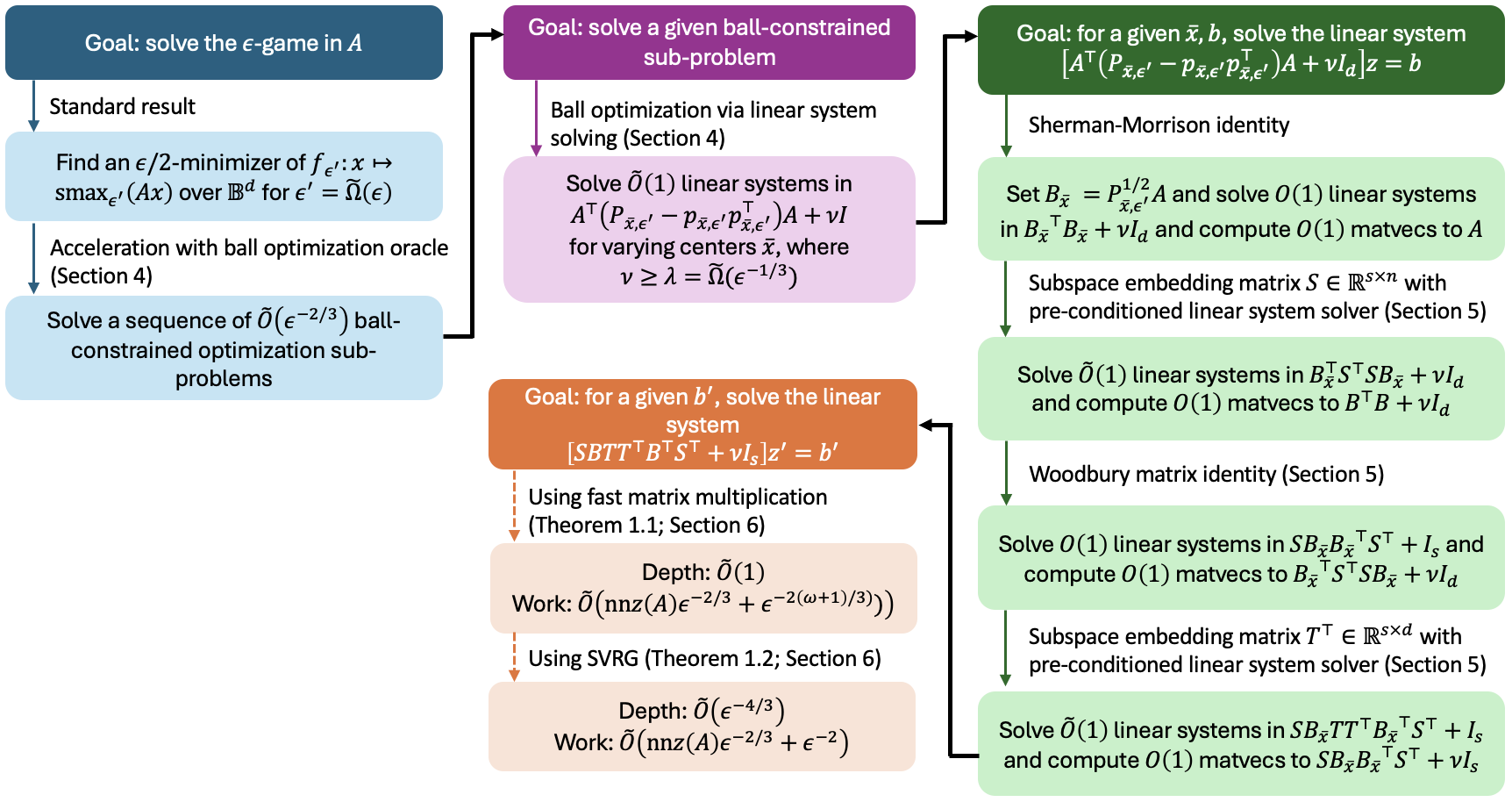}
    \caption{\textbf{Overview of approach.} Dark-colored boxes highlight the main problems and subproblems in the method. Solid colored arrows indicate reduction steps within a single problem, while solid black arrows denote transitions between nested problems. Dotted arrows represent two alternative algorithmic approaches for solving the resulting linear systems. Here, $S \in \R^{s \times n}$ and $T^\top \in \R^{s \times d}$ are oblivious subspace embeddings of size $s = \tilde{O}(\epsilon^{-2/3})$. Using a naive analysis, a fresh $S, T$ would need to be sampled for each $\xbar$, however, using the sample reuse framework of \citep{jin2025reusingsamplesvariancereduction}, we show that the same $T$ can be used across all iterations (Section~\ref{subsec:sample-reuse}).}
    \label{fig:approach}
\end{figure}

\subsection{Reducing the $\epsilon$-game to linear-system solving}\label{subsec:reduce_to_linear_system}

Recall from \Cref{def:l2l1_game} that our goal is to compute an $\epsilon$-minimizer of $f$, defined as $f(x) \defeq \max_{y \in \simplex^n} y^\top A x$, over $\ball^d$. In \Cref{sec:reducing-ell_2-ell_1-to-linear-solving} we show how to reduce this problem to a sequence of approximate linear-system solves, culminating in \Cref{lem:ell_2-ell_1-reduction-to-linear-solves} below. While care is needed to handle approximations and mild extensions not handled in prior work, our overarching approach in this reduction  is similar to that in prior works related to ball acceleration, e.g.,  \cite{carmon2020acceleration,carmon2021thinking,arun2022improvediterationcomplexities}.

To define these systems, for $x \in \ball^d$ and $\eta > 0$, we let $p_{x, \eta} \in \simplex^n$ be defined via
\begin{align}
    \label{eq:overview-p-and-P-def}
    [p_{x, \eta}]_i \defeq \frac{\exp([A{x}]_i/\eta)}{\sum_{j \in [n]} \exp([A{x}]_j/\eta)} \text{ for all $i \in [n]$, and further define }  P_{x, \eta} \defeq \diag(p_{x, \eta}) \, .
\end{align}
Then informally, \Cref{lem:ell_2-ell_1-reduction-to-linear-solves} says that to achieve the work, depth, and matvec complexities of \Cref{thm:first-result} and \Cref{thm:second-result}, it suffices to solve $\Otilde(\epsilon^{-2/3})$ linear systems as needed in \eqref{eq:overview-game-reduced-to-this-linear-system} efficiently in these complexity metrics. We note that \Cref{lem:ell_2-ell_1-reduction-to-linear-solves} exposes the role of the random seed $\chi$ in order to enable application of the sample reuse framework of \citep{jin2025reusingsamplesvariancereduction}, as we discuss further in Section~\ref{subsec:sample-reuse}.

\begin{restatable}{lemma}{reduceGameToLinear}\label{lem:ell_2-ell_1-reduction-to-linear-solves}
For $\epsilon > 0$, there is an algorithm that, with probability at least $3/4$ over the draw of a random seed $\chi \sim \seeddist$, returns $\xtilde \in \ball^d$ such that $f(\xtilde) - \min_{x \in \ball^d} f(x) \le \epsilon$. The computational cost of the method is dominated by the cost of $\Otilde(\epsilon^{-2/3})$-iterations of performing the following:
\begin{itemize}
        \item computing $x' \in \R^d$ which is a $(\poly(\epsilon^{-1}, n))^{-1}$-approximate solution to
    \begin{align}
            \label{eq:overview-game-reduced-to-this-linear-system}
            (H+ \nu I) y = \ghat, \text{ where }
            H = \frac{1}{\epsprim} A^\top ( P_{\xbar, \epsprim} - p_{\xbar, \epsprim} p_{\xbar, \epsprim}^\top) A
        \end{align}
        for some $\xbar \in \ball^d$, $\nu = \Omegatilde(\epsilon^{-1/3})$, and $\ghat \in \R^d$ (which may vary each time) with $\epsprim \defeq \frac{\epsilon}{2 \log n}$, 

    \item $\Otilde(1)$-additional matvecs to $A$,

    \item $\Otilde(n + d)$-additional work, and $\Otilde(1)$-additional depth.
\end{itemize}
Moreover, $\seeddist$ is independent of $A$, and $\chi \sim \seeddist$ can be sampled in $\Otilde(\epsilon^{-2/3})$-work and 
$\Otilde(1)$-depth. 
\end{restatable}

To illustrate how \Cref{lem:ell_2-ell_1-reduction-to-linear-solves} reduces our problem to approximately solving linear systems, recalling that $\nnzo(A) = \Omega(n + d)$, to prove \Cref{thm:first-result} using \Cref{lem:ell_2-ell_1-reduction-to-linear-solves} it suffices to implement \eqref{eq:overview-game-reduced-to-this-linear-system} in $\Otilde(\nnzo(A) + \epsilon^{-2 \omega / 3})$-work and $\Otilde(1)$-depth, using $\Otilde(1)$ matvecs to $A$. Likewise, to prove \Cref{thm:second-result} using \Cref{lem:ell_2-ell_1-reduction-to-linear-solves}, it suffices to implement \eqref{eq:overview-game-reduced-to-this-linear-system} in $\Otilde(\nnzo(A) + \epsilon^{-4/3})$-work, $\Otilde(\epsilon^{-2/3})$-depth, and using $\Otilde(1)$ matvecs to $A$.

Let us now discuss how we prove \Cref{lem:ell_2-ell_1-reduction-to-linear-solves} in \Cref{sec:reducing-ell_2-ell_1-to-linear-solving}. Our first step is to leverage a standard smoothing technique \citep{nesterov2005smooth,carmon2020acceleration, carmon2021thinking, asi2021stochastic} where the nonsmooth objective function $f$ is replaced with a smooth proxy objective. Concretely for $\alpha > 0$, we define the \emph{softmax} function $\smax_\alpha : \R^n \to \R$ via 
\begin{align}\label{eq:smax}
    \smax_\alpha (z) \defeq \max_{y' \in \simplex^n} z^\top y' - \alpha \cdot e(y') = \alpha \log \inparen*{
        \sum_{i \in [n]} \exp([z]_i / \alpha)
    } \, ,
\end{align}
where $e(y) \defeq \sum_{i \in [n]} y_i \log y_i$ is the negative entropy function. Then, defining $f_{\alpha}(x) \defeq \smax_{\alpha}(Ax)$, it is a standard result (see \Cref{lem:properties-of-linear-softmax}) that $f_\alpha$ is an additive $\alpha \log n$ approximation of $f$, i.e., $f(x) \le f_\alpha(x) \le f(x) + \alpha \log n$ for all $x \in \R^d$, and furthermore $f_\alpha$ is $1 / \alpha$-smooth in $\innorm{\cdot}_2$. Thus, defining $\epsilon' \defeq \epsilon/(2\log(n))$, any $\epsilon / 2$-minimizer of $f_{\epsprim}$ over $\ball^d$ is an $\epsilon$-minimizer of $f$ over $\ball^d$, and therefore we focus on obtaining the former.

To (approximately) minimize $f_{\epsprim}$, we leverage the \emph{ball-oracle acceleration} framework developed in a series of prior works \cite{carmon2020acceleration,carmon2021thinking,hilal2021stochasticbiasreduced,carmon2022optimalandadaptivemonteirosvaiter,carmon2022distributionallyrobustoptimizationball,carmon2023resqueing,carmon2024whole}. For $r \in (0, 1)$, this framework reduces minimizing a convex and Lipschitz function $h : \R^d \to \R$ to minimizing a regularized version of $h$ within a ball of radius $r$ at most $\Otilde(r^{-2/3})$ times (see \Cref{prop:ball-accel-guarantee-from-prev-papers} for the formal statement). By choosing $r \gets \tilde{\Theta}(\epsilon)$ and using the fact that $f_\epsprim$ is 1-Lipschitz, we are able to reduce obtaining an $\epsilon / 2$-minimizer of $f_\epsprim$ to $\Otilde(\epsilon^{-2/3})$ computations, for $g_{\nu, \xbar}(x) \defeq f_{\epsprim}(x) + \frac{\nu}{2} \innorm{x - \xbar}_2^2$, of $x' \in \ball^d_r(\xbar) \cap \ball^d$ such that
    \begin{align}
        \label{eq:overview-ball-oracle}
        g_{\nu, \xbar}(x') - \min_{x \in \ball^d_r(\xbar) \cap \ball^d} g_{\nu, \xbar}(x) \le O(\nu \epsilon^{8/3})
    \end{align}
    for some $\xbar \in \ball^d$ and $\nu = \Omegatilde(\epsilon^{-1/3})$ (which may vary each iteration); see \Cref{subsec:ell_2-ell_1-via-linear} for details. Furthermore, each of the $\Otilde(\epsilon^{-2/3})$ iterations requires no additional access to $A$, $\Otilde(d)$-additional work, and $\Otilde(1)$ additional depth, and therefore to prove \Cref{lem:ell_2-ell_1-reduction-to-linear-solves} it suffices to implement \eqref{eq:overview-ball-oracle} using $\Otilde(1)$ linear system solves of the form \eqref{eq:overview-game-reduced-to-this-linear-system}, $\Otilde(n + d)$-additional work, $\Otilde(1)$ additional matvecs to $A$, and $\Otilde(1)$-additional depth.

To achieve this, we leverage that $f_{\epsprim}$ is \emph{$\Otilde(\epsilon^{-1})$-quasi-self-concordant} \cite{bach2010self,sun2019generalized,karimireddy2018globallinearconvergencenewtons,carmon2020acceleration,doikov2023minimizingquasiselfconcordant}, which implies that it is \emph{$c$-Hessian stable} \cite{carmon2020acceleration,karimireddy2018globallinearconvergencenewtons} for an absolute constant $c > 1$ in any ball of radius $r$, i.e., $c^{-1} \hess f_{\epsprim}(y) \preceq \hess f_{\epsprim}(x) \preceq c \hess f_{\epsprim}(y)$ for any $x, y \in \R^d$ with $\innorm{x - y}_2 \le r$. As a result, the smoothness and Hessian-stability of $f_{\epsprim}$ imply $g_{\nu, \xbar}$ is also smooth and Hessian-stable. Combining these properties with the strong convexity of $g_{\nu, \xbar}$, we use a ball-constrained Newton algorithm due to \cite[Alg. 3]{carmon2020acceleration} to reduce \eqref{eq:overview-ball-oracle} to solving $\Otilde(1)$ ball-constrained quadratic optimization problems, along with an additional $\Otilde(1)$-matvecs to $A$, $\Otilde(d)$-work, and $\Otilde(1)$-depth.

More formally, the ball-constrained quadratics subproblems take the form 
\begin{align}
    \label{eq:overview-quadratic-subproblem}
    \underset{x \in \ball_r(\xbar) \cap \ball^d}{\minimize} - g^\top x + \frac{1}{2} x^\top (H + \nu I)  x
\end{align}
for $H = \hess f_{\epsprim}(\xbar)$ as in \eqref{eq:overview-game-reduced-to-this-linear-system}. If \eqref{eq:overview-quadratic-subproblem} had a single ball constraint $x \in \ball_r(\xbar)$, it would be the well-studied \emph{trust-region problem} \cite{conn2000trust}. In particular, \cite{carmon2020acceleration} show how to reduce approximately minimizing the version of \eqref{eq:overview-quadratic-subproblem} with a single ball constraint $x \in \ball_r(\xbar)$ to \emph{exact} linear system solves in matrices of the form $H + \lambda I$ for $\lambda \ge \nu$ by binary searching on the (one-dimensional) Lagrange multiplier associated with the constraint $x \in \ball_r(\xbar)$. In \Cref{subsec:ball-constrained-quadratic-via-linear}, we generalize this approach to \eqref{eq:overview-quadratic-subproblem} by performing a cutting plane method in the two-dimensional dual space of Lagrange multipliers, and carefully show that approximate linear system solves of the form \eqref{eq:overview-game-reduced-to-this-linear-system} suffice to solve \eqref{eq:overview-quadratic-subproblem} to high accuracy.

Note that for any $\xbar \in \ball^d$, one matvec to $H$ (recall \eqref{eq:overview-game-reduced-to-this-linear-system}) can be computed using $O(1)$ matvecs to $A$ (recall also \eqref{eq:overview-p-and-P-def}). Thus, it remains to show how to implement an efficient linear system solver for \eqref{eq:overview-game-reduced-to-this-linear-system}. Next, we discuss how we leverage modern numerical linear algebra techniques to solve such linear systems with low work and depth.

\subsection{Implementing linear system solvers using subspace embeddings}\label{subsec:linsolve-impl}  Here, we focus on solving linear systems of the form \eqref{eq:overview-game-reduced-to-this-linear-system}, by adapting techniques of \citep{derezinski2025faster, derezinski2026approaching}. 

\paragraph{Rank-one reduction.} First, 
note that $H = \nabla^2 f_{\epsprim}(\xbar) = \frac{1}{\epsilon'} A^\top (P_{\xbar} - p_{\xbar}p_{\xbar}^\top) A$ is PSD (because $f_{\epsilon'}$ is convex) and is simply a \emph{rank-one update} of $B_{\xbar}^\top B_{\xbar}$ for $B_{\xbar} \defeq \frac{1}{\sqrt{\epsprim}}P_{\xbar, \epsprim}^{1/2}A$. Consequently, using the Sherman-Morison-Woodbury identity, we show that in order to build a linear-system solver for $ H + \nu I$ it suffices to build a linear-system solver for 
\begin{align*}
    \frac{1}{\epsprim} A^\top P_{\xbar, \epsprim} A + \nu I_{d} = \frac{1}{\epsprim}A^\top P_{\xbar, \epsprim}^{1/2} P_{\xbar, \epsprim}^{1/2} A + \nu I_{d} = B_{\xbar}^\top B_{\xbar} + \nu I_{d}.
\end{align*}
Concretely, we show (Lemma~\ref{lemma:rank-one-term-reduction}) that any linear system in $H + \nu I_{d}$ can be solved by instead solving one linear system in $B_{\xbar}^\top B_{\xbar} + \nu I_d$ plus $O(d)$-additional work and using $\tilde{O}(1)$-additional depth. 

Consequently, we turn our attention to building a linear-system solver for $B_{\xbar}^\top B_{\xbar} + \nu I_{d}$. Here, as in prior work \citep{carmon2022distributionallyrobustoptimizationball, carmon2024whole, carmon2020acceleration}, we observe that 
\begin{align}\label{eq:frobenius-norm-bound}
    \normInline{B_{\xbar}}_F^2 = \frac{1}{\epsilon'} \tr\paren{A^\top P_{\xbar, \epsprim} A} = \frac{1}{\epsilon'} \sum_{i \in [n]} [p_{\xbar, \epsprim}]_i  \normInline{A_{i, :}}_2^2 \leq \frac{1}{\epsilon'} \normInline{A}_{2 \to \infty}^2 = \tilde{O}\paren{\epsilon^{-1}}.
\end{align}
As a result, we can leverage a numerical linear algebra tool called \emph{oblivious random subspace embeddings} to build a good \emph{preconditioner} for the matrix $B_{\xbar}^\top B_{\xbar} + \nu I_d$. 

\paragraph{Oblivious random subspace embeddings.} There are distributions of matrices $S \in \R^{s \times n}$  such that for any matrix $C \in \R^{n \times d}$ and $\nu > 0$, $SC \in \R^{s \times d}$ can be computed in nearly linear time and low parallel depth and, with high probability, for $s \leq d$ sufficiently large, $C^\top S^\top S C + \nu I_d \approx_{2} C^\top C + \nu I_d$ \citep{chenakkod2024optimal, chenakkod2024optimal2, cohen2015optimal, cohen2016nearly, nelson2013osnap}; that is, $C^\top S^\top S C + \nu I_d$ is a \emph{low-rank} 2-approximation for $C^\top C + \nu I_d$.  In particular, in this paper we use that there is an (oblivious) distribution of sparse matrices $S \in \R^{s \times n}$ such that $s = \tilde{O}(\normInline{C}_F^2/\nu)$ suffices \citep{derezinski2026approaching, chenakkod2024optimal2} (Corollary~\ref{corrolary:subspace-embedding}). Consequently, in light of \eqref{eq:frobenius-norm-bound}, and that $\nu = \tilde{\Omega}(\epsilon^{- 1/3})$ in \eqref{eq:overview-game-reduced-to-this-linear-system}, there is a distribution of matrices $S \in \R^{s \times n}$ such that $B_{\xbar}^\top B_{\xbar} + \nu I_d \approx_2 B_{\xbar}^\top S^\top {S B_{\xbar}}+ \nu I_d$ for $s = \tilde{O}(\epsilon^{-2/3})$ and moreover, $SB_{\xbar}$ can be computed with just $\tilde{O}(\epsilon^{-2/3})$-matvecs to $A$.

\paragraph{Preconditioned linear system solvers.} Next, we use that for any $M \in \pdcone^{d}$, in order to build a linear system solver for $M$, it suffices to build a linear system solver for some $N \approx_2 M$ and apply preconditioned methods such as preconditioned Richardson iteration (Theorem~\ref{thm:preconditioned-richardson}) \citep{derezinski2026approaching, golub1988convergence}. Preconditioned Richardson iteration runs $\tilde{O}(1)$ iterations, where each iteration makes one call to a linear system solver for $N$, one matvec to $M$, and performs $\tilde{O}(n+d)$-additional work using $\tilde{O}(1)$-additional depth. Consequently, in order to build a linear-system solver for $B_{\xbar}^\top B_{\xbar} + \nu I_{d}$, it suffices to build a linear-system solver for $B_{\xbar}^\top S^\top S B_{\xbar} + \nu I_{d}$, where $S \in \R^{s \times n}$ is sampled from an appropriate distribution of matrices with $s = \tilde{O}(\epsilon^{-2/3})$.

Our goal is now to  efficiently implement a linear system solver for $B_{\xbar}^\top S^\top S B_{\xbar} + \nu I_{d}$. Note that this is a $d \times d$ linear system, involving the $s \times d$ matrix $F_{\xbar} \defeq S B_{\xbar} \in \R^{s \times d}$ (namely, the coefficient matrix is the regularized Gram matrix of $F_{\xbar}$). To get improved complexities, we further reduce to solving a system involving an $s \times s$ matrix (as opposed to an $s \times d$
matrix). To achieve this, we use a Woodbury matrix identity, preconditioning, and a subspace embedding (similar to \citep{derezinski2026approaching, derezinski2025faster}). We describe this approach in greater detail in the following paragraphs. 

\paragraph{Changing the order of operations and applying a second subspace embedding.} Using the Woodbury matrix identity, in order to build a linear system solver for ${F_{\xbar}}^\top {F_{\xbar}}+ \nu I_{d}$, it suffices to build a linear system solver for ${F_{\xbar}}{F_{\xbar}}^\top + \nu I_{s}$ (Lemma~\ref{lemma:change-order}). 
We then further observe that $\normInline{{F_{\xbar}}}_F^2 \leq \tilde{O}(\epsilon^{-2/3})$ and consequently, we can once again sample another oblivious random subspace embedding matrix $T^\top \in \R^{s \times d}$ with $s = \tilde{O}(\epsilon^{-2/3})$, so that  ${F_{\xbar}}T T^\top {F_{\xbar}}^\top + \nu I_{s} \approx_2 {F_{\xbar}}{F_{\xbar}}^\top + \nu I_{s}$. Now, once again using  preconditioned Richardson iteration, in order to build linear system in ${F_{\xbar}}{F_{\xbar}}^\top + \nu I_{s}$, it suffices to build a linear system solver for ${F_{\xbar}}T T^\top {F_{\xbar}}^\top + \nu I_{s}$. 

Note that we can compute ${F_{\xbar}}T \in \R^{s \times s}$ explicitly with $\tilde{O}(\epsilon^{-2/3})$-matvecs to $A$ in $\tilde{O}(\epsilon^{-2/3} \nnz A)$-work and $\tilde{O}(1)$-depth, by leveraging the \emph{sparsity} of the random matrices $S, T$.\footnote{If $A$ is known explicitly, then ${F_{\xbar}} T$ is computable in just $\tilde{O}(\nnz A)$-time. However, in this paper we often assume $A$ is accessible only through a  matvec oracle, as our results apply to this setting with nearly no loss in complexity.} Moreover, ${F_{\xbar}}T$ is a $\tilde{O}(\epsilon^{-2/3})$-dimensional square matrix. Consequently, using FMM (Lemma~\ref{thm:fmm}) we can implement a $\tilde{O}(1)$-depth linear system solver for ${F_{\xbar}}T T^\top {F_{\xbar}}^\top + \nu I_{s}$ which does $\tilde{O}(\epsilon^{-2\omega/3})$-work. This yields a total work of 
\begin{align}\label{eq:fresh-sample-runtime}
    \tilde{O}(\epsilon^{-2/3}) \cdot \paren{ s\,\nnz A + s^\omega } = \tilde{O}(\epsilon^{-4/3} \nnz A + \epsilon^{-2 (\omega + 1)/3})
\end{align}
and a depth of $\tilde{O}(\epsilon^{-2/3})$. 

Alternatively, at the cost of increased depth, there is an even faster approach for implementing a linear system solver for ${F_{\xbar}}T T^\top {F_{\xbar}}^\top + \nu I_{s}$ using stochastic gradient methods, such as stochastic variance reduced gradient descent (SVRG) \citep{frostig2015regularizing, lin2015universal, johnson2013accelerating}. Using SVRG, we can implement a linear system solver for ${F_{\xbar}}T T^\top {F_{\xbar}}^\top + \nu I_{s}$ in $\tilde{O}(\nnz {B_{\xbar}} + s\normInline{B_{\xbar}}_F^2/\nu) = \tilde{O}(\nnz A + s^2)$ work. Hence, recalling \eqref{eq:frobenius-norm-bound}, we can reduce the work from \eqref{eq:fresh-sample-runtime} to 
\begin{align}\label{eq:fresh-sample-runtime2}
    \tilde{O}(\epsilon^{-2/3}) \cdot \paren{ s\,\nnz A + s^{2} } = \tilde{O}(\epsilon^{-4/3} \nnz A + \epsilon^{-2});
\end{align}
however, the depth increases to $\tilde{O}(\epsilon^{-4/3})$ as SVRG requires sequential computation.

The above bounds in \eqref{eq:fresh-sample-runtime} and \eqref{eq:fresh-sample-runtime2} almost match the total work claimed in Theorem~\ref{thm:first-result} and Theorem~\ref{thm:second-result}, except that the leading term $\epsilon^{-4/3} \nnz A$ is too large by an $\epsilon^{-2/3}$ factor. Likewise, \eqref{eq:fresh-sample-runtime} and \eqref{eq:fresh-sample-runtime2} correspond to a matvec complexity of $\tilde{O}(\epsilon^{-4/3})$, which is an $\epsilon^{-2/3}$ factor worse than prior work \citep{karmarkar2026solving}.  
The bottleneck here is that, naively, for each ball-constrained subproblem, we must construct a \emph{new} ${F_{\xbar}} T = {S B_{\xbar}}T \in \R^{s \times s}$ for freshly sampled oblivious random subspace embeddings $S, T$. Correspondingly, each subproblem naively requires $\tilde{O}(s)$-matvecs. Fortunately, we show that it is possible to reduce this overhead in the following section. 

\subsection{Applying the sample reuse framework}\label{subsec:sample-reuse} As mentioned above, naively, for each center $\xbar$ in the outer ball acceleration loop, computing 
\begin{align*}
     {F_{\xbar}}T = S \frac{1}{\sqrt{\epsprim}}P^{1/2}_{\xbar} A T \in \R^{s \times s}. 
\end{align*}
requires $\tilde{O}(\epsilon^{-2/3})$-matvecs to $A$, and then applying the diagonal rescaling $P_{\xbar}$. Because $\xbar$ is chosen adaptively by the outer ball-acceleration method, naive analysis suggests that each $\xbar$ (i.e., each linear system in Lemma~\ref{lem:ell_2-ell_1-reduction-to-linear-solves}) requires sampling a \emph{fresh} oblivious random subspace embeddings $S, T$ and then reconstructing ${F_{\xbar}}T$.

Fortunately, by leveraging recent work on sample reuse \citep{jin2025reusingsamplesvariancereduction}, a more careful analysis indicates that this can be improved. As each linear system of the form \eqref{eq:overview-game-reduced-to-this-linear-system} is solved to \emph{high accuracy} and $S, T$ are sampled \emph{obliviously} (i.e., from a distribution which is independent of the particular $\xbar$), it is in fact possible to reuse the same subspace embeddings $T$ for all ball acceleration centers $\xbar$ (i.e., each linear system in Lemma~\ref{lem:ell_2-ell_1-reduction-to-linear-solves}) at the cost of at most polylogarithmic overheads in work. 

This allows us to by \emph{pre-compute} $AT$ using $\tilde{O}(\epsilon^{-2/3})$-matvecs to $A$ and $\tilde{O}(\nnz A \epsilon^{-2/3})$-work and $\tilde{O}(1)$-depth (Fact~\ref{fact:AT}). Then, for each iteration in Lemma~\ref{lem:ell_2-ell_1-reduction-to-linear-solves}, we can randomly sample a fresh $S$ and build ${F_{\xbar}}T$ in only $\tilde{O}(\nnz{A})$-work and $\tilde{O}(1)$-depth (Lemma~\ref{lemma:constructing SBT}). This lets us reduce the work in \eqref{eq:fresh-sample-runtime} down to the claimed runtime bound in Theorem~\ref{thm:first-result}.  Analogously, using SVRG in place of FMM, we can reduce \eqref{eq:fresh-sample-runtime2} down to the claimed runtime bound in
Theorem~\ref{thm:second-result}. 

We note that the recent sample reuse framework of \citep{jin2025reusingsamplesvariancereduction} was provided quite generally and hence we use it blackbox. Consequently, the main novelty in our work is in the sequence of careful reductions which bring the original binary classification problem into a form where the framework of \citep{jin2025reusingsamplesvariancereduction} immediately applies and yields our improved work bounds.

\section{Reducing $\ellTwoEllOne$-games to linear system solving}\label{sec:reducing-ell_2-ell_1-to-linear-solving}

In this section, we reduce solving an $\epsilon$-game to approximately solving a sequence of linear systems. In \Cref{subsec:ball-constrained-quadratic-via-linear}, we show how to reduce solving a quadratic (constrained to two balls) to approximate linear system solves. Then in \Cref{subsec:ball-opt-oracle-via-linear}, we combine the results of \Cref{subsec:ball-constrained-quadratic-via-linear} with an accelerated ball-constrained Newton method due to \cite{carmon2020acceleration} to reduce optimizing a ball-constrained Hessian-stable function to approximate linear system solves. Finally, in \Cref{subsec:ell_2-ell_1-via-linear}, we reduce solving an $\epsilon$-game to approximately solving a sequence of linear systems via the ball oracle acceleration framework developed in \cite{carmon2020acceleration,carmon2021thinking,hilal2021stochasticbiasreduced,carmon2022optimalandadaptivemonteirosvaiter,carmon2022distributionallyrobustoptimizationball,carmon2023resqueing,carmon2024whole}.

\paragraph{Notation for \Cref{sec:reducing-ell_2-ell_1-to-linear-solving}.} For a closed convex set $S \subseteq \R^d$ and $x \in \R^d$, $\projt_S(x) \defeq \argmin_{y \in S} \innorm{y - x}_2$ denotes the Euclidean projection of $x$ onto $S$. We define the Euclidean distance between $x$ and $S$ via $\distt(x, S) \defeq \innorm{x - \projt_S(x)}_2$.

\subsection{Ball-constrained quadratics via linear system solves}
\label{subsec:ball-constrained-quadratic-via-linear}

In this section, we give an algorithm for finding a feasible point close in distance to the optimal solution of the optimization problem
\begin{equation}
    \label{eq:trust-region-primal}
\begin{aligned}
& \underset{x \in \R^d}{\minimize}
&&  - g^\top x + \frac{1}{2} x^\top H x \\
& \text{subject to}
&& \innorm{x - u_1}_2 \le  r_1 \text{ and } \innorm{x - u_2}_2 \le r_2
\end{aligned}
\end{equation}
for $g, u_1, u_2 \in \R^d$, $r_1, r_2 > 0$, and $H \in \pdcone^{d}$ where $\mu I \preceq H \preceq L I$ for some $0 < \mu \le L$. We let $\kappa \defeq L/\mu$ and $f(x) \defeq - g^\top x + \frac{1}{2} x^\top H x$ denote the objective function and $C_1 \defeq \inbraces{x' \in \R^d : \innorm{x' - u_1}_2 \le r_1}$ and $C_2 \defeq \inbraces{x' \in \R^d : \innorm{x' - u_2}_2 \le r_2}$ denote the constraint balls. Additionally, we assume that $u_2 \in C_1$ (the second center is contained within the first ball). Thus, Slater's condition is satisfied, strong duality holds, and we let $\xopt$ denote the unique minimizer of \eqref{eq:trust-region-primal}. For $\lambda = (\lambda_1, \lambda_2)$, the Lagrangian (after squaring both sides of the constraints) is given by
\begin{align*}
    \Lagrange(x, \lambda) = - g^\top x + \frac{1}{2} x^\top H x + \lambda_1 (\innorm{x - u_1}_2^2 - r_1^2) + \lambda_2 (\innorm{x - u_2}_2^2 - r_2^2) \, ,
\end{align*}
and we can express the dual optimization problem as
\begin{align}
    \label{eq:trust-region-dual-opt-problem}
    \underset{\lambda \ge 0}{\maximize} \inbraces{ Q(\lambda) \defeq \min_{x \in \R^d} \Lagrange(x, \lambda) } \, .
\end{align}
Letting $\xopt(\lambda) \defeq \argmin_{x \in \R^d} \Lagrange(x, \lambda)$ denote the best response $x$ for some fixed choice of $\lambda \ge 0$, note
\begin{align}
    \label{eq:trust-region-xopt(lambda)}
    \xopt(\lambda) = (H + 2(\lambda_1 + \lambda_2) I)^{-1} (g + 2 \lambda_1 u_1 + 2 \lambda_2 u_2) \, .
\end{align}
Furthermore, the uniqueness of $\xopt(\lambda)$ implies by the envelope theorem that for $\lambda \ge 0$,
\begin{align}
    \label{eq:trust-region-grad-Q}
    \grad_\lambda Q(\lambda) = \grad_\lambda \Lagrange (\xopt(\lambda), \lambda) = (\innorm{\xopt(\lambda) - u_1}_2^2 - r_1^2, \, \innorm{\xopt(\lambda) - u_2}_2^2 - r_2^2) \, .
\end{align}

In the rest of this section, we provide and analyze an algorithm that uses dual gradients \eqref{eq:trust-region-grad-Q} to implement a cutting plane method for the dual objective $Q(\lambda)$. Indeed, note that a dual gradient \eqref{eq:trust-region-grad-Q} can be computed via a linear system solve \eqref{eq:trust-region-xopt(lambda)}, and we show that approximate dual gradients can be computed via approximate linear system solves in \Cref{lem:trust-region-approx-grad-implement-relative}. Once we obtain an (approximately) optimal dual solution via our dual cutting plane method, we perform a final additional (approximate) linear system solve \eqref{eq:trust-region-xopt(lambda)} to obtain a point close in distance to $\xopt$. Finally, to ensure that we return an \emph{exactly} feasible point for \eqref{eq:trust-region-primal}, we project onto the feasible region of \eqref{eq:trust-region-primal} and argue that doing so does not increase the distance to $\xopt$ too much.

Our algorithm can be viewed as a natural extension of the method of \cite[Appendix D]{carmon2020acceleration}, which handles only a single quadratic constraint and performs exact linear system solves, to handling two quadratic constraints and performing approximate linear system solves. Indeed, our cutting plane method in our problem's two-dimensional dual space is the natural analog of the binary search they perform in their problem's single-dimensional dual space.

We state our algorithmic guarantee at the end of this section in \Cref{prop:trust-region-final-algo-guarantee-relative}. We build up to this result by first proving a variety of regularity properties which enable our final guarantee. To start, we bound the maximum size of the dual optimal variables in the following lemma.

\begin{lemma}
    \label{lem:trust-region-bound-on-norm-optimal-dual}
Any optimal pair of dual variables $\lambdaopt \in \argmax_{\lambda \ge 0} Q(\lambda)$ satisfies
\begin{align}
    \label{eq:trust-region-def-R_Q}
    \max \inbraces{\lambdaopt_1, \lambdaopt_2} \le  \frac{  L \innorm{\xopt}_2 + \innorm{g}_2}{2 \min \inbraces{r_1, r_2}} \le \frac{  L (\innorm{u_1}_2 + r_1) + \innorm{g}_2}{2 \min \inbraces{r_1, r_2}} \eqdef R_Q \, .
\end{align}
\end{lemma}

\begin{proof}
First, note that
\begin{align}
    \label{eq:trust-region-bound-on-grad-at-opt}
    \innorm{\grad f(\xopt)}_2 = \innorm{H \xopt - g}_2 \le \innorm{H}_2 \innorm{\xopt}_2 + \innorm{g}_2 \le L \innorm{\xopt}_2  + \innorm{g}_2 \, .
\end{align}
Furthermore, the KKT stationary condition for this problem can be expressed as
\begin{align}
    \label{eq:KKT-trust-region}
    - \grad f(\xopt) = 2 \lambdaopt_1 n_1 + 2 \lambdaopt_2 n_2
\end{align}
for $n_1 \defeq \xopt - u_1$ and $n_2 \defeq \xopt - u_2$. We now consider cases based on which of the two constraints are tight at $\xopt$. If neither constraint is tight, then $\lambdaopt_1 = \lambdaopt_2 = 0$ and we are done. If the first constraint in \eqref{eq:trust-region-primal} is tight and the second is loose so that $\innorm{n_1}_2 = r_1$ and $\lambdaopt_2 = 0$, then taking the norm of both sides of \eqref{eq:KKT-trust-region} yields $\lambdaopt_1 = \innorm{\grad f(\xopt)}_2 / (2 r_1)$, in which case we conclude by \eqref{eq:trust-region-bound-on-grad-at-opt}. Similarly, if the first constraint is loose and the second is tight, we have $\lambdaopt_1 = 0$ and $\lambdaopt_2 = \innorm{\grad f(\xopt)}_2 / (2 r_2)$. 

Now suppose that both constraints are tight at $\xopt$, implying $\innorm{n_1}_2 = r_1$ and $\innorm{n_2}_2 = r_2$. Taking the norm of both sides of \eqref{eq:KKT-trust-region} and squaring, we obtain
\begin{align*}
    \innorm{\grad f(\xopt)}_2^2 = (2 \lambdaopt_1 r_1)^2 + (2 \lambdaopt_2 r_2)^2 + 8 \lambdaopt_1 \lambdaopt_2 (n_1^\top n_2) \, .
\end{align*}
Then by \eqref{eq:trust-region-bound-on-grad-at-opt}, it suffices to show $8 \lambdaopt_1 \lambdaopt_2 (n_1^\top n_2) \ge 0$ (as that would imply $\max \inbraces{ (2 \lambdaopt_1 r_1)^2, (2 \lambdaopt_2 r_2)^2 } \le \innorm{\grad f(\xopt)}_2^2$), in which case it suffices to show $n_1^\top n_2 \ge 0$. Consider the triangle formed by the points $\xopt, u_1, u_2$, and letting $\theta$ denote the angle at the vertex $\xopt$, we have $\cos \theta = \frac{n_1^\top n_2}{\innorm{n_1} \innorm{n_2}}$, so it suffices to show $\cos \theta \ge 0$. By the law of cosines,
\begin{align*}
    \cos \theta = \frac{r_1^2 + r_2^2 - \innorm{u_1 - u_2}_2^2}{2 r_1 r_2} \ge 0
\end{align*}
since the fact that $u_2$ is within the first constraint by assumption implies $\innorm{u_1 - u_2}_2^2 \le r_1^2$.
\end{proof}

Next, we show that the best primal response to any dual variable is bounded.

\begin{lemma}
    \label{lem:trust-region-bound-on-xopt(lambda)}
For all $\lambda \ge 0$, we have 
\begin{align}
    \label{eq:trust-region-bound-on-xopt-Upsilon}
    \innorm{\xopt(\lambda)}_2 \le \Upsilon ~~ \text{where} ~~ \Upsilon \defeq \max \inbraces{\innorm{g}_2 / \mu, \innorm{u_1}_2, \innorm{u_2}_2} \, .
\end{align}
\end{lemma}

\begin{proof}
Letting $\Lambda \defeq \lambda_1 + \lambda_2$, note $\innorm{(H + 2 \Lambda I)^{-1}}_2 \le \frac{1}{\mu + 2 \Lambda}$, in which case \eqref{eq:trust-region-xopt(lambda)} gives
\begin{align*}
    \innorm{\xopt(\lambda)}_2 
    \le \innorm{(H + 2 \Lambda I)^{-1}}_2 \cdot \innorm{g + 2 \lambda_1 u_1 + 2 \lambda_2 u_2}_2 
    \le \frac{\innorm{g}_2 + 2 \Lambda \max \inbraces{\innorm{u_1}_2, \innorm{u_2}_2} }{\mu + 2 \Lambda} 
     \le \Upsilon \, .
\end{align*} 
The last inequality uses that for any $a, b \ge 0$ and $c > 0$, the function $h(t) \defeq \frac{a + bt}{c + 2t}$ is monotonic for $t \ge 0$, and thus can be bounded by the values it approaches at the extremes ($t = 0$ and $t \to \infty$).
\end{proof}

In the following lemma, we show that the dual objective $Q$ is Lipschitz.

\begin{lemma}
    \label{lem:trust-region-Q-Lipschitz}
$Q$ is $\beta_Q$-Lipschitz in $\innorm{\cdot}_2$ over $\lambda \ge 0$ where, for $\Upsilon$ defined in \eqref{eq:trust-region-bound-on-xopt-Upsilon},
\begin{align}
    \label{eq:trust-region-Q-beta_Q-def}
    \beta_Q \defeq \sqrt{2} \max_{i \in \inbraces{1, 2}} \inbraces{ (\Upsilon + \innorm{u_i}_2)^2 + r_i^2 } \, .
\end{align}
\end{lemma}

\begin{proof}
It suffices to show $\innorm{\grad Q(\lambda)}_2 \le \beta_Q $ for all $\lambda \ge 0$. Then recalling in general that $\innorm{v}_2 \le \sqrt{\ell} \innorm{v}_\infty$ for $v \in \R^\ell$, it suffices (by applying the inequality with $\ell \gets 2$) to show $\inabs{ [\grad Q(\lambda)]_i } \le (\Upsilon + \innorm{u_i}_2)^2 + r_i^2$ for $i \in \inbraces{1, 2}$. This follows by combining \eqref{eq:trust-region-grad-Q}, \eqref{eq:trust-region-bound-on-xopt-Upsilon}, and a triangle inequality.
\end{proof}

The next lemma shows that an approximate linear system solve suffices to obtain an approximate dual gradient $\qtilde$.

\begin{lemma}
    \label{lem:trust-region-approx-grad-implement-relative}
For any $\lambda \in \R^2_{\ge 0}$, let $B_\lambda \defeq H + 2(\lambda_1+\lambda_2)I$. Suppose that $x$ is a $\gamma$-approximate solution to the linear system (Definition~\ref{def:linearsystemsolver}) 
\begin{align}\label{eq:exact-linear-system}
    (H + 2(\lambda_1+\lambda_2)I) y = (g + 2\lambda_1 u_1 + 2\lambda_2 u_2). 
\end{align}
Then, for $\Upsilon$ as defined in \eqref{eq:trust-region-bound-on-xopt-Upsilon}, let $\bar x \defeq \projt_{\{z : \normInline{z}_2 \leq \Upsilon\}}(x)$. Then,
\begin{align*}
    \qtilde \defeq \left( \|\bar x-u_1\|_2^2-r_1^2,\, \|\bar x-u_2\|_2^2-r_2^2 \right)
    \text{ satisfies }
    \|\qtilde-\nabla Q(\lambda)\|_2
    \le
    4\gamma\Upsilon^2\sqrt{2\kappa}. 
\end{align*}
\end{lemma}

\begin{proof}
Recall from \eqref{eq:trust-region-xopt(lambda)} that $\xopt(\lambda)$ is the exact solution to \eqref{eq:exact-linear-system}. Therefore, since $x$ is a $\gamma$-approximate solution to this linear system (Definition~\ref{def:linearsystemsolver}), letting $B_\lambda \defeq (H + 2(\lambda_1+\lambda_2)I)$, we have
\begin{align*}
    \|x-\xopt(\lambda)\|_{B_\lambda}
    \le
    \gamma \|\xopt(\lambda)\|_{B_\lambda}.
\end{align*}
Since $B_\lambda \succeq \lambda_{\min}(B_\lambda)I$, it follows from Fact~\ref{lemma:condition-number} that
\begin{align*}
    \|x-\xopt(\lambda)\|_2
    &\le
    \gamma
    \sqrt{
        \kappa(B_\lambda)
    }
    \|\xopt(\lambda)\|_2 .
\end{align*}
Since
\begin{align*}
    \lambda_{\max}(B_\lambda)
    \le L + 2(\lambda_1+\lambda_2)
    \qquad \text{ and } \qquad
    \lambda_{\min}(B_\lambda)
    &\ge \mu + 2(\lambda_1+\lambda_2),
\end{align*}
and $\mu \le L$, we have
\begin{align*}
    \frac{\lambda_{\max}(B_\lambda)}
         {\lambda_{\min}(B_\lambda)}
    \le
    \frac{L+2(\lambda_1+\lambda_2)}
         {\mu+2(\lambda_1+\lambda_2)}
    \le
    \frac{L}{\mu} = \kappa. 
\end{align*}
Using $\|\xopt(\lambda)\|_2 \le \Upsilon$ from \eqref{eq:trust-region-bound-on-xopt-Upsilon}, we obtain
\begin{align*}
    \|x-\xopt(\lambda)\|_2
    \le
    \gamma \sqrt{\kappa}\,\Upsilon .
\end{align*}

Next, since $\xopt(\lambda)$ lies in the $\ell_2$-ball of radius $\Upsilon$ and $\bar x$ is the Euclidean projection of $x$ onto this ball, projection can only decrease the distance to $\xopt(\lambda)$ in the $\ell_2$-norm (due to the nonexpansiveness of projections onto nonempty, closed, convex sets). Consequently,
\begin{align*}
    \|\bar x-\xopt(\lambda)\|_2
    \le
    \|x-\xopt(\lambda)\|_2
    \le
    \gamma \sqrt{\kappa}\,\Upsilon 
\end{align*}

Recalling \eqref{eq:trust-region-grad-Q}, for each $i \in \{1,2\}$,
\begin{align*}
    [\qtilde-\nabla Q(\lambda)]_i
    &=
    \left(\|\bar x-u_i\|_2^2-r_i^2\right)
    -
    \left(\|\xopt(\lambda)-u_i\|_2^2-r_i^2\right) \\
    &=
    \|\bar x-u_i\|_2^2
    -
    \|\xopt(\lambda)-u_i\|_2^2 \\
    &=
    \left\langle
        \bar x-\xopt(\lambda),\,
        \bar x+\xopt(\lambda)-2u_i
    \right\rangle .
\end{align*}
By the Cauchy-Schwarz inequality and the triangle inequality,
\begin{align*}
    \left|[\qtilde-\nabla Q(\lambda)]_i\right|
    &\le
    \|\bar x-\xopt(\lambda)\|_2
    \|\bar x+\xopt(\lambda)-2u_i\|_2 \\
    &\le
    \gamma \sqrt{\kappa}\,\Upsilon
    \left(
        \|\bar x\|_2
        +
        \|\xopt(\lambda)\|_2
        +
        2\|u_i\|_2
    \right).
\end{align*}
By definition of $\bar x$, we have $\|\bar x\|_2 \le \Upsilon$. Moreover, by the choice of $\Upsilon$, we have $\|\xopt(\lambda)\|_2 \le \Upsilon$ and $\|u_i\|_2 \le \Upsilon$. Therefore,
\begin{align*}
    \left|[\qtilde-\nabla Q(\lambda)]_i\right|
    &\le
    \gamma \sqrt{\frac{L}{\mu}}\,\Upsilon
    \left(
        \Upsilon+\Upsilon+2\Upsilon
    \right) 
    =
    4\gamma \sqrt{\kappa}\,\Upsilon^2 .
\end{align*}
Thus,
    $\|\qtilde-\nabla Q(\lambda)\|_\infty
    \le
    4\gamma \sqrt{\frac{L}{\mu}}\,\Upsilon^2.$ Since $\qtilde-\nabla Q(\lambda) \in \mathbb{R}^2$,
\begin{align*}
    \|\qtilde-\nabla Q(\lambda)\|_2
    \le
    \sqrt{2}\|\qtilde-\nabla Q(\lambda)\|_\infty
    \le
    4\sqrt{2\kappa}\,\gamma\,\Upsilon^2 .
\end{align*}
\end{proof}

In the next lemma, we show that if we are able to obtain an approximately optimal dual solution $\lambda$, then the primal best response $\xopt(\lambda)$ is close to the true primal optimum $\xopt$.

\begin{lemma}
    \label{lem:trust-region-dual-subopt-to-distance-to-primal-opt}
For all $\lambda \ge 0$, we have
\begin{align*}
    \innorm{\xopt(\lambda) - \xopt}_2 \le 
    \sqrt{\frac{2 ( \max_{\lambda' \ge 0} Q(\lambda') - Q(\lambda) ) }{\mu}}
\end{align*}
\end{lemma}

\begin{proof}
Note that $x \mapsto \Lagrange(x, \lambda)$ is $\mu$-strongly convex and minimized at $\xopt(\lambda)$, and hence for all $x \in \R^d$,
\begin{align*}
    \Lagrange(x, \lambda) \ge \Lagrange(\xopt(\lambda), \lambda) + \frac{\mu}{2} \innorm{x - \xopt(\lambda)}_2^2 \, .
\end{align*}
Plugging in $x \gets \xopt$ and $\Lagrange(\xopt(\lambda), \lambda) = Q(\lambda)$ and rearranging, we obtain the result via
\begin{align*}
   \frac{\mu}{2}  \innorm{\xopt - \xopt(\lambda)}_2^2 \le  \Lagrange(\xopt, \lambda) - Q(\lambda) \le \max_{\lambda' \ge 0} \Lagrange(\xopt, \lambda') - Q(\lambda) = \max_{\lambda' \ge 0} Q(\lambda') - Q(\lambda) 
\end{align*}
by strong duality. 
\end{proof}

Since our goal is to obtain an \emph{exactly} feasible point which is close in distance to the optimal solution $\xopt$ of \eqref{eq:trust-region-primal}, and we only have access to an approximate linear system solver, it is necessary to show that projecting onto the feasible region does not increase the distance to $\xopt$ too much. We prove this in the following lemma.

\begin{lemma}
    \label{lem:trust-region-projecting-does-not-change}
Let $\alpha_1, \alpha_2 \ge 0$ and $x, y \in \R^d$ be such that $\max_{i \in \inbraces{1,2}} \distt(y, C_i) \le \alpha_1$ and $\innorm{x - y}_2 \le \alpha_2$. Then defining $x' \defeq \projt_{C_1}(x)$ and $x'' \defeq \projt_{C_2}(x')$, we have $x'' \in C_1 \cap C_2$ and $\innorm{x'' - y}_2 \le 2 \alpha_1 + \alpha_2$.
\end{lemma}

\begin{proof}
Note $x'' \in C_2$ by definition, and $x'' \in C_1$ since $x''$ is on the line segment connecting $x' \in C_1$ and $u_2 \in C_1$. We obtain the second result, recalling the non-expansiveness of the projection operator, by combining
\begin{align*}
    \innorm{x'' - y}_2 \le \innorm{x'' - \projt_{C_2} (y)}_2 + \innorm{\projt_{C_2} (y) - y}_2 
    \le \innorm{x' - y}_2 + \alpha_1
\end{align*}
with
\begin{align*}
    \innorm{x' - y}_2 \le \innorm{x' - \projt_{C_1} (y)}_2 + \innorm{\projt_{C_1} (y) - y}_2 
     \le \innorm{x - y}_2 + \alpha_1 
     \le \alpha_2 + \alpha_1 \, .
\end{align*}
\end{proof}

In the next lemma, we state a guarantee for a cutting plane method due to \cite{sidford2023quantumspeedups} which allows for approximate gradients.

\begin{proposition}
    \label{prop:approx-grad-cutting-plane-guarantee}
For $R > 0$ and $S \defeq \inbraces{x \in \R^n_{\ge 0} : \innorm{x}_\infty \le R}$, let $h : S \to \R$ be convex, differentiable, and $\beta$-Lipschitz in $\innorm{\cdot}_2$. Suppose we can only access $h$ through an oracle $\oraclet_{h, \delta} : S \to \R^n$, which, for any query point $x \in S$, satisfies $\innorm{\oraclet_{h, \delta}(x) - \grad h(x)}_2 \le \delta$. Then there is an algorithm which obtains $x \in S$ such that $h(x) - \min_{z \in S} h(z) \le \epsilon$, and makes $\Otilde(n)$ queries to $\oraclet_{h, \delta}$ for $\delta \gets \tilde{\Theta}(\frac{\epsilon}{R \sqrt{n}})$ with $\Otilde(\poly(n))$ additional work, where $\Otilde(\cdot)$ hides polylog factors in $R, \beta, n, \epsilon^{-1}$.
\end{proposition}

\begin{proof}
The query bound to $\oraclet_{h, \delta}$ follows from \cite[Sec. 4]{sidford2023quantumspeedups}, which gives a generic analysis under such an oracle, which they define in \cite[Def. 5]{sidford2023quantumspeedups}. (Namely, their analysis holds for a variety of standard cutting plane algorithms.) Whereas they do not state their guarantees in terms of this oracle directly (they primarily study a stochastic setting where $\oraclet_{h, \delta}$ is an intermediate oracle implemented with high probability), it is straightforward to obtain this bound from the proofs of \cite[Prop. 1]{sidford2023quantumspeedups} and \cite[Prop. 2]{sidford2023quantumspeedups}. As for the additional work, this is achievable by a variety of deterministic cutting plane methods, e.g., \cite{vaidya1989new,atkinson1995cutting}. (Note that the more recent cutting plane methods of \cite{yintat2015fastercuttingplane,haotian2020improvedcuttingplane} are randomized.)
\end{proof}

Finally, we state our main guarantee for this section below. As discussed above, we obtain our guarantee by applying a cutting plane method in the two-dimensional dual space, using an approximate linear system solver to obtain approximate dual gradients, and then obtaining a primal point close to $\xopt$ via an additional approximate linear system solve. Finally, we project onto the feasible region of \eqref{eq:trust-region-primal} and argue that doing so does not increase the distance to $\xopt$ too much.

\begin{proposition}
    \label{prop:trust-region-final-algo-guarantee-relative}
For $\Delta > 0$, there is an algorithm which obtains $\xtilde \in \R^d$ which is feasible for \eqref{eq:trust-region-primal} and satisfies $\innorm{\xtilde - \xopt}_2 \le \Delta$. It requires $\Otilde(1)$ computations of $x' \in \R^d$ such that, for $R_Q, \Upsilon$ defined in \eqref{eq:trust-region-def-R_Q} and \eqref{eq:trust-region-bound-on-xopt-Upsilon}, $x'$ is a 
\begin{align*}
     \tilde{\Theta} \inparen*{ \min \inbraces*{\frac{\Delta}{\sqrt{\kappa} \Upsilon}, \frac{\mu \Delta^2}{R_Q \sqrt{\kappa} \Upsilon^2}} }\text{-approximate solution to \eqref{eq:exact-linear-system}}
\end{align*} 
for some $\lambda \ge 0$ (which may vary each time), $\Otilde(d)$ additional work, $\Otilde(1)$ additional depth, and no additional matvecs to $H$, where $\Otilde(\cdot)$ hides polylog factors in $r_1, r_2, \mu, L, \Delta$ and the $\ell_2$-norms of $u_1, u_2, g$.
\end{proposition}

\begin{proof}
We first apply the cutting plane method of \Cref{prop:approx-grad-cutting-plane-guarantee} on the dual optimization problem \eqref{eq:trust-region-dual-opt-problem} with $h \gets - Q$ (since $Q$ is concave), $n \gets 2$, $\beta \gets \beta_Q$ given by \eqref{eq:trust-region-Q-beta_Q-def} (which is a valid choice by \Cref{lem:trust-region-Q-Lipschitz}), and $R \gets R_Q$ (guaranteeing an optimal dual solution is within $S$ by \Cref{lem:trust-region-bound-on-norm-optimal-dual}). With $\delta \gets \tilde{\Theta}(\epsilon / R_Q)$ as in \Cref{prop:approx-grad-cutting-plane-guarantee} (for $\epsilon > 0$ to be chosen momentarily), a query to $\oraclet_{- Q, \delta}(\lambda)$ for $\lambda \ge 0$ can be implemented by computing $x' \in \R^d$ which is an \begin{align}\label{eq:trust-region-queries-to-get-grad}
    \tilde{\Theta} \paren{\frac{\epsilon}{R_Q \sqrt{\kappa} \Upsilon^2}}\text{-approximate solution to \eqref{eq:exact-linear-system}}
\end{align}
and $O(d)$ additional work by Lemma~\ref{lem:trust-region-approx-grad-implement-relative}. Thus, since $n = 2$, the cutting plane method of \Cref{prop:approx-grad-cutting-plane-guarantee} obtains $\lambdatilde \ge 0$ such that $\max_{\lambda' \ge 0} Q(\lambda') - Q(\lambdatilde) \le \epsilon$ having computed at most $\Otilde(1)$ points $x'$ satisfying \eqref{eq:trust-region-queries-to-get-grad} with $\Otilde(d)$ additional work and no additional matvecs to $H$.

Then, we compute $z \in \R^d$ such that $\innorm{z - \xopt(\lambdatilde)}_2 \le \Delta / 4$, which can be achieved by computing a
\begin{align*}
    \inparen*{\frac{\Delta}{4 \sqrt{\kappa} \Upsilon}}\text{-approximate solution to \eqref{eq:exact-linear-system} (with $\lambda \gets \tilde{\lambda}$})
\end{align*}
via an analogous argument to the beginning of the proof of \Cref{lem:trust-region-approx-grad-implement-relative}.
Choosing $\epsilon = \Theta(\mu \Delta^2)$, we obtain $\innorm{\xopt(\lambdatilde) - \xopt}_2 \le \Delta / 4$ by \Cref{lem:trust-region-dual-subopt-to-distance-to-primal-opt}. Since $\xopt \in C_1 \cap C_2$, we have $\max_{i \in \inbraces{1,2}} \distt(\xopt(\lambdatilde), C_i) \le \Delta / 4$, and thus applying \Cref{lem:trust-region-projecting-does-not-change} with $y \gets \xopt(\lambdatilde)$ and $x \gets z$ gives that $z' \defeq \projt_{C_1}(z)$ and $\xtilde \defeq \projt_{C_2}(z')$ satisfy $\innorm{\xtilde - \xopt(\lambdatilde)}_2 \le 3 \Delta / 4$. Furthermore, $\xtilde$ is feasible and can be computed in $O(d)$ work given $z$ and with no additional matvecs. Finally, $\innorm{\xtilde - \xopt}_2 \le \Delta$ by triangle inequality.
\end{proof}

\subsection{Ball optimization oracles via linear system solves}
\label{subsec:ball-opt-oracle-via-linear}

In this section, we combine the results of \Cref{subsec:ball-constrained-quadratic-via-linear} with the accelerated ball-constrained Newton method of \cite{carmon2020acceleration} to reduce optimizing a ball-constrained Hessian-stable function to approximate linear system solves. To unpack this further, \cite{carmon2020acceleration} showed how to implement a \emph{ball optimization oracle} \cite[Def. 1]{carmon2020acceleration} for a Hessian-stable function via \emph{exact} linear system solves. Our goal in this section is to extend their guarantees to approximate linear system solves, culminating in \Cref{thm:ball-opt-oracle-guarantee} at the end of the section.

First, we recall the definition of a ball optimization oracle from \cite{carmon2020acceleration}. We note that due to our application where we seek to obtain a minimizer over $\ball^d$ (\cite{carmon2020acceleration} seek to obtain an unconstrained minimizer), it is necessary to add an additional constraint over $\ball^d$, which \cite[Def. 1]{carmon2020acceleration} does not require. However, otherwise the definition is the same.

\begin{definition}[{\cite[Def. 1]{carmon2020acceleration}}]
    \label{def:ball-opt-oracle}
We call $\Oball$ a \emph{$(\zeta, r)$-ball optimization oracle} for $f : \R^d \to \R$ if for any $\xbar \in \ball^d$, it outputs $y = \Oball(\xbar) \in \ball^d_r(\xbar) \cap \ball^d$ such that $\innorm{y - x_{\xbar, r}}_2 \le \zeta$ for some $x_{\xbar, r} \in \argmin_{x \in \ball^d_r(\xbar) \cap \ball^d} f(x)$.
\end{definition}

Next, we restate, for completeness, the definition of a Hessian-stable function. We state it for general norms as it appears in \cite{carmon2020acceleration}; however, in the following we will only instantiate it for the $\ell_2$-norm.

\begin{definition}[{\cite[Def. 7]{carmon2020acceleration}}]
    \label{def:Hessian-stable}
A twice-differentiable function $f : \R^d \to \R$ is \emph{$(r, c)$-Hessian stable} for $r, c \ge 0$ with respect to a norm $\innorm{\cdot}$ if for all $x, y \in \R^d$ with $\innorm{x - y} \le r$ we have $c^{-1} \hess f(y) \preceq \hess f(x) \preceq c \hess f(y)$.
\end{definition}

We state our main guarantee for this section in the following lemma. This guarantee follows from combining the guarantee of \Cref{prop:trust-region-final-algo-guarantee-relative} from the previous section with the accelerated ball-constrained Newton method due to \cite{carmon2020acceleration}. In particular, this accelerated ball-constrained Newton method reduces implementing a ball optimization oracle to solving quadratics of the form \eqref{eq:trust-region-double-ball} in our specific application, where there are two ball constraints. \cite{carmon2020acceleration} show how to implement a quadratic constrained to a single ball \eqref{eq:trust-region-single-ball} via exact linear system solves by binary searching on the one-dimensional dual Lagrange multiplier. The necessity of the previous section comes from the fact that our application results in two ball constraints, and furthermore we must solve the linear systems approximately.

We note that the use of an accelerated method here is not strictly necessary for our application since the Hessian-stability parameter $c$ is an absolute constant in our application (see the proof of \Cref{lem:ell_2-ell_1-reduction-to-linear-solves} at the end of \Cref{sec:reducing-ell_2-ell_1-to-linear-solving} where we instantiate $c = O(1)$). This is also the case for the applications of \cite{carmon2020acceleration}, as noted in Section 1.1 of that paper. Following \cite{carmon2020acceleration}, we choose to use it for completeness.

\begin{lemma}
    \label{thm:ball-opt-oracle-guarantee}
Let $f : \R^d \to \R$ be $L$-smooth, $\mu$-strongly convex, and $(r, c)$-Hessian stable, all in $\innorm{\cdot}_2$, and let $\xbar \in \ball^d$ be such that $\innorm{\grad f(\xbar)}_2 \le G$. Then there is an algorithm which implements a $(\zeta, r)$-ball optimization oracle for query point $\xbar$. Additionally, letting $\Otilde(\cdot)$ hide polylog factors in $L, \mu, G, c, r, \zeta$, the computational cost of the method is dominated by the cost of $\Otilde(c)$-iterations of performing the following:
\begin{itemize}
    \item computing $x' \in \R^d$ which is a $(\poly(L, \mu^{-1}, G, c, r, \zeta^{-1}))^{-1}$-approximate solution to $(\nabla^2 f(\bar x) + \lambda I) y = \hat{g}$, for some $\lambda \ge 0$ and $\ghat \in \R^d$ (which may vary each time),

    \item performing $\Otilde(1)$ additional matvecs to $\hess f(\xbar)$, $\Otilde(1)$ additional evaluations of $\grad f(\cdot)$, and not accessing $f$ in any other way,

    \item using $\Otilde(d)$ additional work (outside of the above), and $\Otilde(1)$ additional depth.
\end{itemize}
\end{lemma}

\begin{proof}
Note that by strong convexity, we have $\innorm{\xbar - \xhat}_2 \le G / \mu$ for the unconstrained minimizer $\xhat \defeq \argmin_{x \in \R^d} f(x)$. Then, the proof of correctness follows the same steps as the proof of \cite[Theorem 9]{carmon2020acceleration} with $D \gets G / \mu$ and $M \gets I$, except replacing the call to \cite[Algorithm 7]{carmon2020acceleration} in Line 9 of \cite[Algorithm 3]{carmon2020acceleration} with a call to the cutting plane method of \Cref{prop:trust-region-final-algo-guarantee-relative}. \cite[Algorithm 7]{carmon2020acceleration} 
solves the trust-region problem
\begin{align}
    \label{eq:trust-region-single-ball}
    \underset{x \in \ball_r(\xbar)}{\minimize} \,   - g^\top x + \frac{1}{2} x^\top H x
\end{align}
for $H \in \pdcone^{d}$ with $\mu I \preceq H \preceq LI$ (in particular, $H = \hess f(\xbar)$); specifically by returning $\xtilde \in \ball_r(\xbar)$ such that $\innorm{\xtilde - x_{g, H}}_2 \le \Delta$, where $x_{g, H}$ is the minimizer of \eqref{eq:trust-region-single-ball}. Instead, we replace this subroutine with a call to the algorithm of \Cref{prop:trust-region-final-algo-guarantee-relative} to obtain an approximate minimizer of 
\begin{align}
    \label{eq:trust-region-double-ball}
    \underset{x \in \ball_r(\xbar) \cap \ball^d}{\minimize} \,   - g^\top x + \frac{1}{2} x^\top H x \, ,
\end{align}
namely, $\xtilde'$ such that $\innorm{\xtilde' - x_{g, H}'}_2 \le \Delta$ for $x_{g, H}'$ the minimizer of \eqref{eq:trust-region-double-ball}. In particular, we instantiate $u_1 \gets 0$, $r_1 \gets 1$, $u_2 \gets \xbar$, and $r_2 \gets r$.

The proof of correctness of \cite[Algorithm 3]{carmon2020acceleration} treats \cite[Algorithm 7]{carmon2020acceleration} as a black-box; it only requires that the output satisfies $\innorm{\xtilde - x_{g, H}}_2 \le \Delta$ by \cite[Prop. 8]{carmon2020acceleration}. It does not exploit the specific geometry of the domain. Thus, substituting \cite[Prop. 8]{carmon2020acceleration} with \Cref{prop:trust-region-final-algo-guarantee-relative} in the proof of \cite[Theorem 9]{carmon2020acceleration} yields the desired correctness result, up to requiring a polynomial bound on $\innorm{g}_2$ whenever \eqref{eq:trust-region-double-ball} is invoked, which we do next. (Note that when instantiating \Cref{prop:trust-region-final-algo-guarantee-relative}, we choose $\Delta \gets \frac{\mu \zeta^2}{4 L c (5 r + D)}$ as in Line 5 of \cite[Algorithm 3]{carmon2020acceleration}.)

To conclude the proof of correctness, we must obtain a polynomial bound on $\innorm{g}_2$ each time Line 9 of \cite[Algorithm 3]{carmon2020acceleration} solves an instance of \eqref{eq:trust-region-double-ball}, due to the $\innorm{g}_2$-dependence in \Cref{prop:trust-region-final-algo-guarantee-relative}. To obtain this bound, note that it is clear from the pseudocode of \cite[Algorithm 3]{carmon2020acceleration} that $g$ always takes the form $\grad f(y) - H (\alpha y + (1 - \alpha) z)$ for some $\alpha \in (0, 1)$ and $y, z \in \ball^d$. Thus, we can bound
\begin{align*}
    \innorm{ \grad f(y) - H (\alpha y + (1 - \alpha) z) }_2 &\le \innorm{\grad f(y)}_2 + \innorm{ H (\alpha y + (1 - \alpha) z) }_2 \le G + 3L
\end{align*}
since $\innorm{\grad f(y)}_2 = \innorm{\grad f(y) - \grad f(\xbar) + \grad f(\xbar)}_2 \le L \innorm{y - \xbar}_2 + \innorm{\grad f(\xbar)}_2 \le 2L + G$, and $\innorm{ H (\alpha y + (1 - \alpha) z) }_2 \le L$ since $0 \preceq H \preceq LI$ in particular.

Finally, as for the complexity bounds, the number of iterations of \cite[Algorithm 3]{carmon2020acceleration} is $\Otilde(c)$ as shown in the proof of \cite[Thm. 9]{carmon2020acceleration}. Each iteration calls the algorithm of \Cref{prop:trust-region-final-algo-guarantee-relative} (previously \cite[Algorithm 7]{carmon2020acceleration}) a single time and makes a single additional matvec to $\hess f(\xbar)$ and query to $\grad f(\cdot)$. Thus, the bound on the number of approximate linear system solves follows from \Cref{prop:trust-region-final-algo-guarantee-relative}. The additional work bound follows since each iteration of \cite[Algorithm 3]{carmon2020acceleration} uses $O(d)$ additional work.
\end{proof}

\subsection{$\ellTwoEllOne$-games via linear system solves}
\label{subsec:ell_2-ell_1-via-linear}

In this section, we put everything together to reduce solving an $\epsilon$-game to approximate linear system solves, culminating in the proof of \Cref{lem:ell_2-ell_1-reduction-to-linear-solves}. In particular, this is achieved by reducing solving an $\epsilon$-game to implementing a series of ball optimization oracles, in which case we may apply \Cref{thm:ball-opt-oracle-guarantee} from the previous section. This reduction is performed using the ball oracle acceleration framework developed in \cite{carmon2020acceleration,carmon2021thinking,hilal2021stochasticbiasreduced,carmon2022optimalandadaptivemonteirosvaiter,carmon2022distributionallyrobustoptimizationball,carmon2023resqueing,carmon2024whole}.

First, however, we set up notation and give some preliminary technical lemmas. Recall the following definition of a \emph{quasi-self-concordant function}, for which we use the formulation of \cite{carmon2020acceleration}, restated for completeness. See also \cite{bach2010self,sun2019generalized,karimireddy2018globallinearconvergencenewtons,carmon2020acceleration,doikov2023minimizingquasiselfconcordant} for other work on optimizing such functions. 

\begin{definition}[Quasi-self-concordance {\cite[Def.~10]{carmon2020acceleration}}]
    \label{def:quasi-self-concordance}
We say that a thrice-differentiable $h : \R^d \to \R$ is $M$-\emph{quasi-self-concordant (QSC)} with respect to some norm $\innorm{\cdot}$, for $M \ge 0$, if for all $u, w, x \in \R^d$,
\begin{align*}
    \inabs{\grad^3 h(x)[u, u, w]} \le M \innorm{w} \innorm{u}^2_{\hess h(x)} \, ,
\end{align*}
i.e., the restriction of the third-derivative tensor of $h$ to any direction is bounded by a multiple of its Hessian norm.
\end{definition}

Next, For $\alpha > 0$, we define $\smax_\alpha : \R^n \to \R$ via 
\begin{align}
    \label{eq:smax-def-sec-4}
    \smax_\alpha (y) \defeq \max_{y' \in \simplex^n} y^\top y' - \alpha \cdot e(y') = \alpha \log \inparen*{
        \sum_{i \in [n]} \exp(y_i / \alpha)
    }
\end{align}
where $e(y) \defeq \sum_{i \in [n]} y_i \log y_i$ is the negative entropy function. 
We collect some properties of $\smax_\alpha(y)$ in the following lemma. For the latter two properties, we cite \cite{carmon2020acceleration}.

\begin{lemma}[Properties of $\smax_\alpha(y)$]
    \label{lem:properties-of-generic-softmax}
For $\alpha > 0$, letting $h(y) \defeq \smax_\alpha(y)$, we have that $h$ is $1$-Lipschitz, $1 / \alpha$-smooth, and $2 / \alpha$-QSC, all in $\innorm{\cdot}_\infty$.
\end{lemma}

\begin{proof}
The latter two properties are given in \cite[Lemma 14]{carmon2020acceleration}. As for Lipschitzness, note
\begin{align*}
    \frac{\partial h}{\partial y_k}(y) = \frac{\exp(y_k / \alpha)}{\sum_{i \in [n]} \exp(y_i / \alpha)} \, ,
\end{align*}
implying $\innorm{\grad h(y)}_1 \le 1$ for all $y \in \R^n$.
\end{proof}

Next, we collect standard properties of $\smax_\alpha(Ax)$, proven here for completeness.

\begin{lemma}[Properties of $\smax_\alpha(Ax)$]
    \label{lem:properties-of-linear-softmax}
For $\alpha > 0$, $A \in \R^{n \times d}$ with $\innorm{A}_{2 \to \infty} = \max_{i \in [n]} \innorm{A_{i, :}}_2 \le 1$, $g(x) \defeq \smax_\alpha(Ax)$, and $f(x) \defeq \max_{y \in \simplex^n} y^\top A x$, we have that $f(x) \le g(x) \le f(x) + \alpha \log n$ for all $x \in \R^d$. Furthermore, $g$ is $1$-Lipschitz, $1 / \alpha$-smooth, and $2 / \alpha$-QSC, all in $\innorm{\cdot}_2$. Finally, defining $p_x \in \simplex^n$ via $[p_x]_i \propto \exp([A x]_i / \alpha)$ for $i \in [n]$, we have
\begin{align*}
    \grad g(x) = A^\top p_x ~~\text{and}~~ \hess g(x) = \frac{1}{\alpha} A^\top (\diag(p_x) - p_x p_x^\top) A \, .
\end{align*}
\end{lemma}

\begin{proof}
The first claim follows from \eqref{eq:smax-def-sec-4} and because the negative entropy function $e(y)$ takes values in $[- \log n, 0]$. Then, defining $h(y) \defeq \smax_\alpha(y)$, note $\grad g(x) = A^\top \grad h(Ax)$. Lipschitzness follows since $\innorm{\grad g(x)}_2 \le 1$ for all $x \in \R^d$ by triangle inequality along with the fact that $\max_{i \in [n]} \innorm{A_{i, :}}_2 \le 1$ by assumption and $\innorm{\grad h(Ax)}_1 \le 1$ by \Cref{lem:properties-of-generic-softmax}. 

As for smoothness, recall in general that a twice-differentiable convex function $\phi : \R^d \to \R$ is $\beta$-smooth with respect to a norm $\innorm{\cdot}$ if and only if $v^\top \hess \phi(x) v \le \beta  \innorm{v}^2$ for all $x, v \in \R^d$. Then note that for any $v \in \R^d$, we have
\begin{align*}
    v^\top \hess g(x) v = v^\top A^\top \hess h(Ax) A v \overle{(i)} \frac{1}{\alpha} \innorm{Av}_\infty^2 \overle{(ii)} \frac{1}{\alpha} \innorm{v}_2^2
\end{align*}
by $(i)$ \Cref{lem:properties-of-generic-softmax} and $(ii)$ the fact that $\sup_{u \ne 0} \frac{\innorm{Au}_\infty}{\innorm{u}_2} \le 1$ by assumption on $A$, implying $\innorm{Av}_\infty \le \innorm{v}_2$ in particular. 

As for quasi-self concordance, note that for all $u, w, x \in \R^d$,
\begin{align*}
    \inabs{\thirdd g(x)[u, u, w]} = \inabs{\thirdd h (Ax) [Au, Au, Aw]} \overle{(iii)} \frac{2}{\alpha} \innorm{Aw}_\infty \innorm{Au}^2_{\hess h(Ax)} \overle{(iv)} \frac{2}{\alpha} \innorm{w}_2 \innorm{u}^2_{\hess g(x)}
\end{align*}
by $(iii)$ \Cref{lem:properties-of-generic-softmax} and $(iv)$ $\innorm{Aw}_\infty \le \innorm{w}_2$ as well as the fact that
\begin{align*}
    \innorm{Au}^2_{\hess h(Ax)} = u^\top A^\top \hess h(Ax) A u = u^\top \hess g(x) u = \innorm{u}^2_{\hess g(x)} \, .
\end{align*}
The gradient and Hessian follow via straightforward computation.
\end{proof}

Recall in the previous section we introduced the formulation of a ball optimization oracle from \cite{carmon2020acceleration} in \Cref{def:ball-opt-oracle}. We used this formulation there as it was directly adapted to their accelerated ball-constrained Newton method, which we used to prove \Cref{thm:ball-opt-oracle-guarantee}. However, the ball oracle acceleration framework guarantee which we cite in \Cref{prop:ball-accel-guarantee-from-prev-papers} below uses a slightly different formulation of a ball optimization oracle, which involves minimizing a regularized function over a ball of radius $r$. In the following lemma, we show that implementing the former suffices to implement the latter. 

\begin{lemma}
    \label{lem:implement-ball-oracle-from-ball-accel-papers}
Let $h : \R^d \to \R$ be convex, $L$-smooth, and $\beta$-Lipschitz in $\innorm{\cdot}_2$. Then for $\lambda > 0$ and $g_{\lambda, \xbar}(x) \defeq h(x) + \frac{\lambda}{2} \innorm{x - \xbar}_2^2$, the output $x' = \Oball(\xbar)$ of a $(\zeta, r)$-ball optimization oracle $\Oball(\cdot)$ for $g_{\lambda, \xbar}$ (\Cref{def:ball-opt-oracle}) satisfies
\begin{align*}
    g_{\lambda, \xbar}(x') - \min_{x \in \ball_r(\xbar) \cap \ball^d} g_{\lambda, \xbar}(x) \le (\beta + \lambda r) \zeta + (L + \lambda) \zeta^2 / 2 \, .
\end{align*}
\end{lemma}

\begin{proof}
Note that $g_{\lambda, \xbar}$ is $(L + \lambda)$-smooth, in which case, letting $\xopt \defeq \argmin_{x \in \ball_r(\xbar) \cap \ball^d} g_{\lambda, \xbar}(x)$,
\begin{align*}
    g_{\lambda, \xbar}(x') - g_{\lambda, \xbar}(\xopt) &\le \inangle{\grad g_{\lambda, \xbar}(\xopt), x' - \xopt} + \frac{L + \lambda}{2} \innorm{x' - \xopt}_2^2 \\
    &\le \innorm{\grad g_{\lambda, \xbar}(\xopt)}_2 \innorm{x' - \xopt}_2 + (L + \lambda) \zeta^2 / 2 \\
    &\le \zeta \innorm{\grad h(\xopt) + \lambda (\xopt - \xbar)}_2 + (L + \lambda) \zeta^2 / 2 \\
    &\le \zeta ( \innorm{\grad h(\xopt)}_2 + \lambda \innorm{\xopt - \xbar}_2 ) + (L + \lambda) \zeta^2 / 2 \\
    &\le (\beta + \lambda r) \zeta + (L + \lambda) \zeta^2 / 2 \, .
\end{align*}
\end{proof}

Next we state a ball oracle acceleration guarantee due to \cite[Prop. 2]{carmon2023resqueing}, which is itself based on \cite[Prop. 1]{carmon2022distributionallyrobustoptimizationball} with minor modifications. In particular, their framework reducing minimizing a convex and Lipschitz function to a sequence of regularized ball-constrained subproblems of the form \eqref{eq:ball-opt-from-DRO-paper}. Importantly, each regularized subproblem is guaranteed a certain minimal level of regularization, which is critical to obtaining our ultimate guarantees.

\begin{proposition}[{\cite[Prop. 2]{carmon2023resqueing} and \cite[Prop. 1]{carmon2022distributionallyrobustoptimizationball}}]
    \label{prop:ball-accel-guarantee-from-prev-papers}
Let $\epsilon > 0$, $r \in (0, 1)$, and $h : \R^d \to \R$ be convex and $\beta$-Lipschitz in $\innorm{\cdot}_2$. Then there is an algorithm that, with probability at least $3/4$ over the draw of a random seed $\chi \sim \seeddist$, returns $\xtilde \in \ball^d$ such that $h(\xtilde) - \min_{x \in \ball^d} h(x) \le \epsilon$. Additionally, letting $\Otilde(\cdot)$ hide polylog factors in $\beta, \epsilon^{-1}, r^{-1}$, the computational cost of the method is dominated by the cost of $\Otilde(r^{-2/3})$-iterations of performing the following:
\begin{itemize}
    \item computing $x' \in \ball^d_r(\xbar) \cap \ball^d$ such that, for $g_{\lambda, \xbar}(x) \defeq h(x) + \frac{\lambda}{2} \innorm{x - \xbar}_2^2$,
    \begin{align}
        \label{eq:ball-opt-from-DRO-paper}
        g_{\lambda, \xbar}(x') - \min_{x \in \ball^d_r(\xbar) \cap \ball^d} g_{\lambda, \xbar}(x) \le O(\lambda r^{8/3})
    \end{align}
    for some $\xbar \in \ball^d$ and $\Omegatilde(\epsilon r^{-4/3}) \le \lambda \le O(\beta \epsilon^{-1})$ (which may vary each time), and does not access $h$ in any other way,

    \item $\Otilde(d)$ additional work, and $\Otilde(1)$ additional depth.
\end{itemize}
Moreover, $\seeddist$ is independent of $h$, and $\chi \sim \seeddist$ can be sampled in $\Otilde(r^{-2/3})$-work and $\Otilde(1)$-depth. 
\end{proposition}

\begin{proof}
This follows from an instantiation of \cite[Prop. 2]{carmon2023resqueing}, which is itself based on \cite[Prop. 1]{carmon2022distributionallyrobustoptimizationball}. The setup of \cite[Prop. 2]{carmon2023resqueing} is not constrained to the unit ball, but \cite[Prop. 1]{carmon2022distributionallyrobustoptimizationball} allows for an arbitrary convex constraint set $\xset$, which we set to $\ball^d$, so that the constraints of \eqref{eq:ball-opt-from-DRO-paper} match \cite[Def. 1]{carmon2022distributionallyrobustoptimizationball}. The total work bound is clear from the discussion in \cite[Appendix C]{carmon2023resqueing} as well as Algorithms 1 and 2 in \cite{carmon2022distributionallyrobustoptimizationball}. In particular, there are $\Otilde(r^{-2/3})$ iterations, each with $\Otilde(d)$ additional work.
\end{proof}

In the following lemma, we combine \Cref{prop:ball-accel-guarantee-from-prev-papers} and \Cref{thm:ball-opt-oracle-guarantee}  to reduce optimizing a general convex, smooth, Lipschitz, and Hessian-stable function to approximate linear system solves in the regularized Hessian of the function. We note that to enable more direct application of the sample reuse framework of \cite{jin2025reusingsamplesvariancereduction} in later sections, we ``redefine'' a single iteration of the algorithm of \Cref{lem:alg-for-general-Hessian-stable-function} to correspond to a single linear system solve.

\begin{lemma}
    \label{lem:alg-for-general-Hessian-stable-function}
For $\epsilon > 0$ and $r \in (0, 1)$, let $h : \R^d \to \R$ be convex, $L$-smooth, $\beta$-Lipschitz, and $(r, c)$-Hessian stable (\Cref{def:Hessian-stable}), all in $\innorm{\cdot}_2$. Then there is an algorithm that, with probability at least 3/4 over the draw of a random seed $\chi \sim \seeddist$, returns $\xtilde \in \ball^d$ such that $h(\xtilde) - \min_{x \in \ball^d} h(x) \le \epsilon$. Additionally, letting $\Otilde(\cdot)$ hide polylog factors in $L, \beta, \epsilon^{-1}, r^{-1}, c$, the computational cost of the method is dominated by the cost of $\Otilde(c r^{-2/3})$-iterations of performing the following:
\begin{itemize}
    \item computing $x' \in \R^d$ which is a $(\poly(L, \beta, \epsilon^{-1}, r^{-1}, c))^{-1}$-approximate solution to $(\hess h(\xbar) + \lambda I)y = \ghat$
    for some $\xbar \in \ball^d$, $\lambda = \Omegatilde(\epsilon r^{-4/3})$, and $\ghat \in \R^d$ (which may vary each time),

    \item performing $\Otilde(1)$ additional matvecs to $\hess h(\cdot)$, $\Otilde(1)$ additional evaluations of $\grad h(\cdot)$, and not accessing $h$ in any other way,

    \item using $\Otilde(d)$ additional work (outside of the above), and $\Otilde(1)$ additional depth.
\end{itemize}
Moreover, $\seeddist$ is independent of $h$, and $\chi \sim \seeddist$ can be sampled in $\Otilde(r^{-2/3})$-work and $\Otilde(1)$-depth.
\end{lemma}

\begin{proof}
This follows from instantiating \Cref{prop:ball-accel-guarantee-from-prev-papers} where each implementation of \eqref{eq:ball-opt-from-DRO-paper} is performed via the algorithm of \Cref{thm:ball-opt-oracle-guarantee}. Indeed, \Cref{lem:implement-ball-oracle-from-ball-accel-papers} shows that \eqref{eq:ball-opt-from-DRO-paper} can be implemented via a single call to a $(\zeta, r)$-ball optimization oracle for $g_{\lambda, \xbar}$ for an appropriate choice of precision $\zeta$. Furthermore, note 
\begin{align*}
    \grad g_{\lambda, \xbar}(x) = \grad h(x) + \lambda (x - \xbar) ~~\text{and}~~ \hess g_{\lambda, \xbar} (x) = \hess h(x) + \lambda I \, ,
\end{align*}
and therefore a gradient evaluation of $\grad g_{\lambda, \xbar}(\cdot)$ requires a gradient evaluation of $\grad h(\cdot)$ and $O(d)$ additional work. Also, a matvec to $\hess g_{\lambda, \xbar}(\cdot)$ requires a matvec to $\hess h(\cdot)$ and at most $O(d)$ additional work. Finally, it is straightforward to show that $g_{\lambda, \xbar}$ is also $(r, c)$-Hessian stable.
\end{proof}

Finally, we prove the main result of \Cref{sec:reducing-ell_2-ell_1-to-linear-solving} below in \Cref{lem:ell_2-ell_1-reduction-to-linear-solves}, which reduces solving the $\epsilon$-game to approximate linear system solves. The proof is a straightforward application of \Cref{lem:alg-for-general-Hessian-stable-function}, using the properties of the softmax function derived above. For \Cref{lem:ell_2-ell_1-reduction-to-linear-solves}, mirroring the context of its initial statement in \Cref{subsec:reduce_to_linear_system}, we fix an $\epsilon$-game per \Cref{def:l2l1_game}, define $f(x) \defeq \max_{y \in \simplex^n} y^\top A x$, and define $P_{\xbar, \epsprim}, p_{\xbar, \epsprim}$ as in \eqref{eq:overview-p-and-P-def}.

\reduceGameToLinear*

\begin{proof}
We apply \Cref{lem:alg-for-general-Hessian-stable-function} with $h(x) \defeq \smax_\alpha(Ax)$ for $\alpha \defeq \frac{\epsilon}{2 \log n}$ and $r \gets \Theta(\epsilon / \log n)$. $\inabs{f(x) - h(x)} \le \epsilon / 2$ for all $x \in \R^d$ by \Cref{lem:properties-of-linear-softmax} and therefore it suffices to obtain a $\frac{\epsilon}{2}$-minimizer of $h$. By \Cref{lem:properties-of-linear-softmax}, we have that $h$ is 1-Lipschitz, $\frac{2 \log n}{\epsilon}$-smooth, and $\frac{4 \log n}{\epsilon}$-QSC in $\innorm{\cdot}_2$. The latter and \cite[Lemma 11]{carmon2020acceleration} (with $M \gets I$) imply $h$ is $(r, O(1))$-Hessian stable in $\innorm{\cdot}_2$. Finally, note that by the gradient and Hessian expressions in \Cref{lem:properties-of-linear-softmax}, a matvec to $\hess h(\cdot)$ as well as an evaluation of $\grad h(\cdot)$ can both be computed in $O(1)$ matvecs to $A$ and $O(n)$ additional work.
\end{proof}

\section{Linear system solving}\label{sec:linear-system-reductions}

In this section, we show how to implement an efficient $(\epsilon, \delta)$-linear system solver (Definition~\ref{def:linearsystemsolver}) for
\begin{align}\label{eq:target-solver}
    \frac{1}{\epsilon'} A^\top (P_{\xbar, \epsilon'} - p_{\xbar, \epsilon'}p_{\xbar, \epsilon'}^\top) A + \nu I_{d}
\end{align}
for some $\xbar \in \mathbb{B}^d, \epsilon, \epsilon' > 0$, $\nu > \lambda > 0$, $\delta \in (0, 1)$, and $A \in \R^{n \times d}$ with $\normInline{A}_{2 \to \infty} = 1$. Recall from \eqref{eq:overview-p-and-P-def} that for any $x \in \mathbb{B}^d$ and $\eta > 0$, we use the notation $p_{x, \eta}$ to denote the vector in $\simplex^n$ defined via $[p_{x, \eta}]_i \propto \exp([A x]_i /\eta)$ for $i \in [n]$ and correspondingly denote $P_{x, \eta} \defeq \diag(p_{x, \eta})$.

First, in Section~\ref{sec:prelim-linalg} we discuss several linear algebraic preliminaries which will aid in our analysis. Next in Section~\ref{sec:linear-system-reductions-subsection1} and~\ref{sec:linear-system-reductions-subsection2}, we discuss the sequence of reductions discussed in Section~\ref{sec:approach}, which enable our method. Finally, in Section~\ref{sec:putting-together} we combine the results in the preceding section to give our final linear system solver guarantees. Throughout this section, we use $\tilde{O}(\cdot)$ to hide polylogarithmic factors in the following parameters: $n, d, \epsilon^{-1}, \delta^{-1}, \lambda^{-1}, \nu^{-1}$.

\paragraph{Randomized algorithms and random seeds.} In the remainder of the paper, on occasion, it will be helpful to expose the random seed(s) used to seed the randomness in an $(\epsilon, \delta)$-linear system solver $\cO$ for $M \in \pdcone^{d}$ (Definition~\ref{def:linearsystemsolver}). In this case, we may write $\cO_{\xi}$ to make explicit that the solver's randomness is fixed by the seed $\xi$. That is, $\cO_{\xi}: \R^d \to \R^d$ is a \textit{deterministic} mapping corresponding to \emph{conditioning} on a fixed seed $\xi$. Moreover, on any input $b \in \R^d$ one can view $\cO(b) \in \R^d$ as a random variable whose value is given by $\cO_{\xi}(b)$ for a random $\xi \sim \seeddist$ (where $\seeddist$ is the distribution of the random seed used by $\cO$). Note that because $\cO$ is an $(\epsilon, \delta)$-linear system solver, for any $b \in \R^d$, we have that with probability $1-\delta$ over the draw of $\xi \sim \seeddist$, $\cO_{\xi}(b)$ is an $\epsilon$-solution of $Mx =b$. 

\subsection{Linear algebraic preliminaries}\label{sec:prelim-linalg}

Here, we review several useful properties from numerical linear algebra. As our techniques for linear system solving are in part motivated by those discussed in \citep{derezinski2026approaching, derezinski2025faster}, our presentation often follows their notation and definitions. In the paragraphs below, we discuss several matrix identities, formally introduce subspace embeddings, discuss preconditioning for linear system solving, and introduce several standard linear system solvers. 

\paragraph{Matrix identities.} A fact we use repeatedly throughout this paper is the Woodbury matrix identity. 

\begin{fact}[Woodbury matrix identity]\label{fact:woodbury} Let $A \in \R^{n \times n}, B \in \R^{n \times k}, C \in \R^{k \times k}, D \in \R^{k \times n}$ where $C, A,$ and $C^{-1} + DA^{-1}B$ are invertible. Then, 
\begin{align*}
    (A + BCD)^{-1} = A^{-1} - A^{-1} B(C^{-1} + DA^{-1}B)^{-1} DA^{-1}. 
\end{align*}
\end{fact}

The following special case of Fact~\ref{fact:woodbury}, where $k = 1$, is commonly known as the Sherman-Morrison matrix identity.

\begin{fact}[Sherman-Morrison matrix identity]\label{fact:sharman-morrison-woodbury} Let $A \in \R^{n \times n}$ be invertible and $u, v \in \R^n$. Then, $A + uv^\top$ is invertible if and only if $1 + v^\top A^{-1} u \neq 0$ and moreover, 
\begin{align*}
    (A + uv^\top)^{-1} = A^{-1} - \frac{A^{-1} uv^\top A^{-1}}{1 + v^\top A^{-1} u}. 
\end{align*}
\end{fact}

\paragraph{Subspace embeddings.} Here, we formally introduce the notion of a regularized embedding. Intuitively, an $(\epsilon, \nu)$-embedding matrix constructs a low-dimensional $(1+\epsilon)$ spectral approximation  of $A^\top A + \nu I$, in the following sense.

\begin{definition}[Regularized-embedding, Definition 8 of \citep{derezinski2025faster}]\label{def:embedding} For $\epsilon, \nu > 0$, we call $S \in \R^{k \times n}$ an $(\epsilon, \nu)$-(regularized subspace) embedding for $A \in \R^{n \times d}$ if 
\begin{align*}
    A^\top S^\top S A+ \nu I \approx_{1+\epsilon} A^\top A + \nu I.
\end{align*}
\end{definition}

The property of being an $(\epsilon, \nu)$-embedding is monotonic in $\nu$, as formalized in the following fact.

\begin{fact}\label{fact:monotonicity} Suppose that $\epsilon, \nu > 0$ and $S \in \R^{k \times n}$ is a $(\epsilon, \nu)$-embedding for $A \in \R^{n \times d}$. Then for all $\nu' \geq \nu$, $S$ is a $(\epsilon, \nu')$-embedding for $A$.
\end{fact}
\begin{proof} This follows from $(A^\top S^\top S A + \nu I) \approx_{1+\epsilon} (A^\top A + \nu I)$ and that $(\nu'-\nu) I \approx_{1+\epsilon} (\nu' - \nu) I$. 
\end{proof}

More specifically, our results use the following notion of a sparse oblivious subspace embedding. 

\begin{definition}[Sparse oblivious efficient regularized subspace embedding]
\label{def:sparse-embedding}
Fix $\lambda > 0$, $s,b \in \mathbb{N}$, and $\delta \in (0,1)$.
A distribution $\mathcal{P}_{\mathrm{embedding}}$ over matrices
$S \in \mathbb{R}^{s \times n}$ is called a
\emph{sparse oblivious $(s,b,\delta,\lambda)$-efficient subspace embedding distribution}
for a set of matrices $\cM \subset \mathbb{R}^{n \times d}$ if the following hold: every matrix in the support of $\mathcal{P}_{\mathrm{embedding}}$ has exactly $b$ nonzero entries per column;
a sample $S \sim \mathcal{P}_{\mathrm{embedding}}$ can be generated using $O(nb)$ work and $\widetilde{O}(1)$ depth;
and for every $A \in \cM$, with probability at least $1-\delta$ over
$S \sim \mathcal{P}_{\mathrm{embedding}}$, the matrix $S$ is a $(2,\lambda)$-embedding for $A$.
When the latter holds, we call $S$ a
\emph{sparse oblivious $(s,b,\lambda)$-embedding} for $A$.
When $\lambda$ and $\cM$ are clear from context, we may simply say
an \emph{$(s,b,\delta)$-embedding distribution}. Likewise, when $\lambda$ and $A$ are clear from context we may simply say an \emph{$(s,b)$-embedding}. 
\end{definition}

While several prior works \citep{nelson2013osnap, chenakkod2024optimal, chenakkod2024optimal2, cohen2016nearly} study guarantees for sparse oblivious embedding matrices, we use the version stated in \citep{derezinski2026approaching}, which was originally due to \citep{chenakkod2024optimal}. In particular, we leverage the following special case of Lemma~12 of \citep{derezinski2026approaching}, which bounds the sketch size needed to construct a constant approximation for $A \in \R^{n \times d}$.

\begin{corollary}\label{corrolary:subspace-embedding} Let $A \in \R^{n \times d}$, $\lambda > 0, \delta \in (0,1)$. There exists a sparse oblivious $(s,b, \delta, \lambda)$-subspace embedding distribution $\cP_{\mathrm{embedding}}$(Definition~\ref{def:sparse-embedding}) for $\{A \in \R^{n \times d} : \normInline{A}_F^2 \leq \lambda k\}$ with 
$s = {O}((k + \log(1/(\delta))))$ and $b = O\paren{\log^2(k/(\delta)) + \log^3(d/(\delta))}$. 
\end{corollary}
\begin{proof} Take $\epsilon = 1/6$ in Lemma 12 of \cite{derezinski2026approaching}. Then, we have
\begin{align*}
    A^\top S^\top S A + \frac{1}{k} \sum_{i > k} \sigma_i^2(A) I_d \approx_{2} A^\top A + \frac{1}{k} \sum_{i > k} \sigma_i^2(A) I_d. 
\end{align*}
Since $\frac{1}{k}\sum_{i > k} \sigma_i^2(A) \leq \normInline{A}_F^2/k \leq \lambda$, from Fact~\ref{fact:monotonicity}, we have $A^\top S^\top S A + \lambda I_d \approx_{2} A^\top A + \lambda I_d$.
\end{proof}

One important property of sparse oblivious subspace embeddings is that they are $b$-\textit{column sparse}. The following lemma shows that such matrices preserve sparsity, upon application to $A$. This is helpful in obtaining our overall work complexity guarantees, as we discuss in greater detail Section~\ref{sec:putting-together}. 

\begin{lemma}[Sparse embedding matrices preserve sparsity]\label{def:sparsity-preserved}
Let $S \in \R^{s \times n}$ be a matrix with at most $b$ non-zeros per column and let $A \in \R^{n \times d}$. Then, $\nnz{SA} \leq b \, \nnz{A}$.
\end{lemma}
\begin{proof}
Let $B = SA \in \R^{s \times d}$. For any $r \in [s]$ and $j \in [d]$, the $(r,j)$-th entry of $B$ is
\[
B_{rj} = \sum_{i=1}^n S_{ri} A_{ij}.
\]
If $B_{rj} \neq 0$, then there exists an index $i \in [n]$ such that $S_{ri} \neq 0$ and $A_{ij} \neq 0$. Fix a nonzero entry $A_{ij} \neq 0$. Its contribution to the product $SA$ is confined to entries $B_{rj}$ with $S_{ri} \neq 0$. Since column $i$ of $S$ has at most $b$ nonzero entries, the nonzero $A_{ij}$ can contribute to at most $b$ entries of $B$. Summing over all nonzero entries of $A$, each of which contributes to at most $b$ entries of $SA$, we obtain the result.
\end{proof}

\paragraph{Preconditioning.} The next theorem gives a standard guarantee on preconditioned Richardson iteration for solving a linear system in $M \in \pdcone^d$ given a linear system solver (Definition~\ref{def:linearsystemsolver}) for $N \in \pdcone^d$ such that $M \approx_c N$. Such an $N$ is commonly called a \emph{preconditioner} for $M$. We note that the following guarantee is well-known (see, e.g. \citep{golub1988convergence}, Lemma 2.5 of \citep{cohen2018solving} or the discussion in \citep{derezinski2026approaching}).

\begin{theorem}[Preconditioned Richardson]\label{thm:preconditioned-richardson} There is an algorithm \[\textsc{PrecondRichardson}(M, c, \epsilon, \delta, b, \cO)\]
which takes in $M \in \pdcone^{d}$, $c \geq 1$, $\epsilon > 0$, $\delta \in (0, 1)$, $b \in \R^d$, and an oracle $\cO$. Here, $\cO$ is a $(1/(4c^2), \delta/L)$-linear system solver for a matrix $N \in \pdcone^d$ such that $N \approx_{c} M$. Letting $L = \ceil{2c^2\log(2c/\epsilon)}$, the algorithm makes $L$ queries to $\cO$, $L$ matvecs to $M$, performs $O(dL)$-additional work using $\tilde{O}(1)$-additional depth, and with probability $1-\delta$, outputs an $\epsilon$-solution of $Mx = b$.
\end{theorem}

\begin{proof} Fix $b \in \mathbb{R}^d$, and let $x^\star = M^{-1} b$. We define the solver via the preconditioned Richardson iteration with step size $\eta = 1/c$:
initialize $x_0 = 0$, and for $t = 0,1,\ldots,L-1$ perform
\[
r_t = b - Mx_t,
\qquad
y_t = \cO(r_t),
\qquad
x_{t+1} = x_t + \eta y_t, 
\]
and output $x_L$. Note that $M \approx_c N$ implies
\[
\frac{1}{c} I \preceq N^{-1/2} M N^{-1/2} \preceq c I.
\]
In particular, all eigenvalues of $N^{-1} M$ lie in $[1/c,\,c]$. Consider the update matrix $(I - \eta N^{-1} M)$. Since $\eta = 1/c$, the eigenvalues of this matrix lie in the range $[1 - \eta c, \, 1 - \eta(1/c)] = [0, \, 1 - 1/c^2]$. Therefore, for any $v \in \R^d$, the operator norm contracts:
\begin{equation}\label{eq:richardson_contraction_corrected}
\|(I - \eta N^{-1} M)v\|_N \le \left(1 - \frac{1}{c^2}\right)\|v\|_N.
\end{equation}
Also note that $\|N^{-1} M v\|_N \le c \|v\|_N$.

Let $e_t \defeq x_t - x^\star$ denote the error at iteration $t$. We define the oracle error $\xi_t$ such that $y_t = N^{-1}r_t + \xi_t$. The error update becomes:
\begin{align}\label{eq:error-term-corrected}
e_{t+1} &= x_t + \eta(N^{-1}r_t + \xi_t) - x^\star \nonumber \\
&= e_t + \eta N^{-1} M (x^\star - x_t) + \eta \xi_t \nonumber \\
&= (I - \eta N^{-1} M) e_t + \eta \xi_t.
\end{align}
The oracle guarantees that $\|y_t - N^{-1}r_t\|_N \le \frac{1}{4c^2}\|N^{-1}r_t\|_N$. 
By a union bound over $L$ steps, with probability $1-\delta$, for all $t$:
\[
\|\xi_t\|_N 
\le \frac{1}{4c^2} \|N^{-1} M e_t\|_N
\le \frac{1}{4c^2} \cdot c \|e_t\|_N 
= \frac{1}{4c} \|e_t\|_N.
\]
Plugging this into \eqref{eq:error-term-corrected} and using the triangle inequality:
\[
\|e_{t+1}\|_N 
\le \left(1 - \frac{1}{c^2}\right)\|e_t\|_N + \eta \|\xi_t\|_N
\le \left(1 - \frac{1}{c^2}\right)\|e_t\|_N + \frac{1}{c} \left(\frac{1}{4c} \|e_t\|_N\right).
\]
Simplifying the terms, we obtain
\[
\|e_{t+1}\|_N 
\le \left(1 - \frac{1}{c^2} + \frac{1}{4c^2}\right)\|e_t\|_N
= \left(1 - \frac{3}{4c^2}\right)\|e_t\|_N
\le \left(1 - \frac{1}{2c^2}\right)\|e_t\|_N.
\]
Consequently, by induction,
\[
\|e_L\|_N \le \left(1 - \frac{1}{2c^2}\right)^L \|e_0\|_N \le \exp\!\left(-\frac{L}{2c^2}\right)\|x^\star\|_N.
\]
Taking $L = \lceil 2c^2\log(2c/\epsilon) \rceil$, this implies $\|x_L - x^\star\|_N \le \frac{\epsilon}{c} \|x^\star\|_N$.
Using $M \approx_c cN$, we conclude $\|x_L - x^\star\|_M \le \epsilon \|x^\star\|_M$ by Fact~\ref{lemma:basic-properties}. 
\end{proof}

\paragraph{Standard solvers.}
In our algorithms underlying Theorems~\ref{thm:first-result} and~\ref{thm:second-result}, we use preconditioned Richardson iteration (Theorem~\ref{thm:preconditioned-richardson}) and consider two underlying subroutines for implementing a linear system solver for the preconditioner $N$. The first, is to simply apply FMM, which eventually corresponds to our final runtime claimed in Theorem~\ref{thm:first-result}. 

\begin{lemma}[FMM \citep{alman2025more}]\label{thm:fmm} Consider $C \in \R^{d \times d}, \nu$ and $\epsilon > 0$. There is a (deterministic) algorithm $\textsc{FMM}(C, b, \nu, \epsilon)$, which takes in an explicit $C \in \R^{d \times d}, b \in \R^d, \nu, \epsilon > 0$, does $\tilde{O}(d^\omega)$-work in $\tilde{O}(1)$-depth and outputs an $\epsilon$-solution of $(C^\top C + \nu I_d) x = b$ (Definition~\ref{def:linearsystemsolver}). 
\end{lemma}

The second, is to use stochastic variance-reduced gradient descent (SVRG), which will eventually correspond to our final runtime claimed in Theorem~\ref{thm:second-result}.

\begin{lemma}[SVRG \cite{frostig2015regularizing, lin2015universal}]\label{thm:svrg} There is a randomized algorithm $\textsc{SVRG}_{\xi}(C, b, \nu, \epsilon, \delta)$ that takes in an explicit $C \in \R^{d \times d}, b \in \R^d, \nu, \epsilon > 0, \delta \in (0,1)$ and a random seed $\xi \sim \seeddist$ such that, with probability $1-\delta$ over the draw of $\xi \sim \seeddist$, the algorithm outputs an $\epsilon$-solution of $(C^\top C + \nu I_d) x = b$ (Definition~\ref{def:linearsystemsolver}). Moreover, the algorithm can be implemented in $\tilde{O}(\nnz C + d \normInline{C}_F^2/\nu)$-work using $\tilde{O}(\normInline{C}_F^2/\nu)$-depth, and $\cP_{\mathrm{seed}}$ is independent of $C$ and $b$. 
\end{lemma}
\begin{proof} Consider 
\begin{align*}
    G = \begin{pmatrix}
        C \\
        \sqrt{\nu} I_d
    \end{pmatrix} \in \R^{2d \times d}, \qquad b' = \begin{pmatrix}
        0_d \\
        \frac{1}{\sqrt{\nu}} b
    \end{pmatrix} \in \R^{2d \times 1} 
\end{align*}
and $g(x) \defeq \frac{1}{2} \normInline{Gx - b'}_2^2$. This is a least squares problem, whose normal equations are
\begin{align*}
    G^\top G x = (C^\top C  + \nu I_d) x = G^\top b' = b. 
\end{align*}
Consequently, the (unique) minimizer $x^\star \in \R^d$ of $g(x)$ satisfies $(C^\top C + \nu I_d)x^\star = b$. Moreover,
\begin{align*}
    g(x) = \sum_{i \in [2d]} g_i(x), \text{ where } g_i(x) = \frac{1}{2}(G_{i, :}^\top x - [b']_i)^2, 
\end{align*}
and each $g_i$ is convex and has smoothness
\begin{align*}
    \normInline{G_{i, :}}_2^2 = \begin{cases}
    \normInline{C_{i, :}}_2^2, &i \leq d \\
    \nu, & i > d
\end{cases}.
\end{align*}
Consequently, standard SVRG guarantees \citep{frostig2015regularizing, lin2015universal} imply that SVRG converges to an $\epsilon$-solution of $(C^\top C + \nu I_{d})x = b$ using $\tilde{O}(1)$ matvecs to $G$, $\tilde{O}(\normInline{C}_F^2/\nu)$ queries to a row of $G$, and  $\tilde{O}(\normInline{C}_F^2/\nu)$-depth, which completes the proof. Since a matvec to $G$ can be implemented by a matvec to $C$, and a row query to $G$ requires $\tilde{O}(d)$-work, the result follows. 
\end{proof}

\subsection{Simplifying the target matrix via a rank-one update}\label{sec:linear-system-reductions-subsection1}

Here, we follow the proof approach in Section~\ref{sec:approach} to reduce the problem of building a linear system solver for \eqref{eq:target-solver} into the problem of building a linear system solver for $SBTT^\top B^\top S$ where $B = \frac{1}{\sqrt{\epsilon'}} P_{\xbar, \epsilon'}^{1/2} A$ and $S \in \R^{s \times n}, T^\top \in \R^{s \times d}$ are $(s,b)$-embeddings (Definition~\ref{def:sparse-embedding}) of an appropriate sketch size and column sparsity $s$ and $b$. 

First, the following lemma states a general result that given a linear system solver (Definition~\ref{def:linearsystemsolver}) for $M \in \pdcone^d$ and a rank one update $uu^\top$ such that $M - uu^\top \in \pdcone^d$, we can implement a linear system solver for $M - uu^\top$ with low computational overhead and parallel depth. 

As alluded in Section~\ref{sec:approach}, this result reduces building a linear system solver for \eqref{eq:target-solver} to building a linear system solver for
\begin{equation}\label{eq:reduced}
    \frac{1}{\epsilon'} A^\top P_{\bar{x}} A + \nu I.
\end{equation}
This is advantageous because $B = \frac{1}{\sqrt{\epsilon'}} P_{\xbar, \epsilon'}^{1/2} A$ is simply a diagonal rescaling of $A$, and consequently, it is more straightforward to apply the regularized embedding techniques discussed in Section~\ref{sec:prelim-linalg} to build an efficient linear system solver for \eqref{eq:reduced} $= B^\top B + \nu I$.  

\begin{algorithm2e}
\caption{$\textsc{RankOneSolve}(M, u, \lambda, \epsilon, \delta, b, \cO)$}\label{alg:rank-one-solve}
\KwInput{Matrix-vector access to $M \in \R^{d \times d}, u \in \R^d, b \in \R^d$, $\lambda > 0, \epsilon \in (0, 1/2), \delta \in (0, 1)$ such that $M - uu^\top \succeq \lambda I_d$.}
\KwInput{$\cO$ is a $\paren{\frac{\epsilon}{9\kappa_u^2}, \frac{\delta}{2}}$-approximate linear system solver for $M$ with $\kappa_u \defeq 1 + \frac{\|u\|_2^2}{\lambda}$.}
 $\hat{y} \gets \cO(b)$\; 
 $\hat{z} \gets \cO(u)$\; 
$\hat{\alpha} \gets 1-u^\top \hat{z}$\; 
$\hat{x} \gets \hat{y} + \frac{u^\top \hat{y}}{\hat{\alpha}} \hat{z}$\; 
\textbf{Return }{$\hat{x}$}
\end{algorithm2e}

\begin{lemma}[Solving with a symmetric rank-one update]
\label{lemma:rank-one-term-reduction} 
There is an algorithm \[\textsc{RankOneSolve}(M, u, \lambda, \epsilon, \delta, b, \cO)\] (Algorithm~\ref{alg:rank-one-solve}) which takes in $M\succeq \lambda I_d$ and $u\in\mathbb{R}^d$ such that $M-uu^\top \succeq \lambda I_d$ for $\lambda > 0$, $\epsilon\in(0,1/2)$, $\delta\in(0,1)$, $b \in \R^d$, and an oracle $\mathcal{O}$. The algorithm makes two queries to $\mathcal{O}$, performs $O(d)$-additional work using $\tilde{O}(1)$-additional depth. If $\cO$ is a $\paren{\frac{\epsilon}{9\kappa_u^2}, \frac{\delta}{2}}$-approximate linear system solver for $M$ with $\kappa_u \defeq 1 + \frac{\|u\|_2^2}{\lambda}$, then with probability $1-\delta$, the algorithm outputs an $\epsilon$-solution to $(M - uu^\top)x = b$.
\end{lemma}

\begin{proof}
Fix $b\in\mathbb{R}^d$. Let $\tilde{M} \defeq M - uu^\top$. We wish to approximate $x^\star \defeq \tilde{M}^{-1}b$. Define:
\[
y \defeq M^{-1}b,\qquad z \defeq M^{-1}u,\qquad\text{ and }\qquad \alpha \defeq 1-u^\top z.
\]
First, we establish a lower bound on $\alpha$. Since $M \succeq \tilde{M} \succeq \lambda I$, we have $M^{-1} \preceq \tilde{M}^{-1}$. By the Sherman-Morrison identity (Fact~\ref{fact:sharman-morrison-woodbury}), \[\tilde{M}^{-1} = M^{-1} + \frac{1}{\alpha}zz^\top.\] 

Since $\tilde{M}^{-1}$ is PD, $\alpha \in (0, 1)$. Furthermore, $M \succeq \lambda I + uu^\top$ implies $M^{-1} \preceq (\lambda I + uu^\top)^{-1}$. Thus:
\[
u^\top M^{-1} u \le u^\top (\lambda I + uu^\top)^{-1} u = \frac{\|u\|_2^2}{\lambda + \|u\|_2^2}.
\]
Consequently,
\begin{equation}\label{eq:alpha-bound}
\alpha = 1 - u^\top M^{-1} u \ge 1 - \frac{\|u\|_2^2}{\lambda + \|u\|_2^2} = \frac{\lambda}{\lambda + \|u\|_2^2} = \frac{1}{\kappa_u}.
\end{equation}

The exact solution satisfies $x^\star = y + \frac{u^\top y}{\alpha}z$. The algorithm computes estimates using the oracle $\cO$ with precision $\epsilon_{\text{op}} \defeq \frac{\epsilon}{9\kappa_u^2}$:

\[
\hat{y} \gets \cO(b), \qquad \hat{z} \gets \cO(u), \qquad \hat{\alpha} \defeq 1-u^\top \hat{z}, \qquad \hat{x} \defeq \hat{y} + \frac{u^\top \hat{y}}{\hat{\alpha}}\hat{z}
\]
and outputs $\hat{x}$. 
With probability $1 - 2(\delta/2) = 1-\delta$, the oracle guarantees:
\begin{equation}\label{eq:solver-guar}
\|\hat{y}-y\|_M \le \epsilon_{\text{op}}\|y\|_M \qquad
\text{ and } \qquad
 \|\hat{z}-z\|_M \le \epsilon_{\text{op}}\|z\|_M.
\end{equation}
Condition on this event. Note that $\|z\|_M = \sqrt{u^\top M^{-1} M M^{-1} u} = \sqrt{u^\top M^{-1} u} = \sqrt{1-\alpha}$. 

We first bound the error in the scalar $\hat{\alpha}$. By Cauchy-Schwarz and \eqref{eq:solver-guar}:
\[
|\hat{\alpha}-\alpha| = |u^\top(\hat{z}-z)| \le \|u\|_{M^{-1}}\|\hat{z}-z\|_M \le \epsilon_{\text{op}}\|z\|_M^2 = \epsilon_{\text{op}}(1-\alpha).
\]
Recall that $\epsilon_{\text{op}} = \frac{\epsilon}{9\kappa_u^2}$. Since $\alpha \ge 1/\kappa_u$ by \eqref{eq:alpha-bound}, we have $\epsilon_{\text{op}} \le \frac{\epsilon}{9}\alpha^2$.
Substituting this into the error bound:
\[
|\hat{\alpha}-\alpha| \le \frac{\epsilon}{9}\alpha^2(1-\alpha) < \frac{\alpha}{2} \cdot \left(\frac{2\epsilon \alpha(1-\alpha)}{9}\right).
\]
Since $\epsilon < 1/2$ and $\alpha(1-\alpha) \le 1/4$, the term in the parenthesis on the right hand side of the display above, is strictly less than 1. Thus $|\hat{\alpha}-\alpha| \le \alpha/2$, which implies:
\begin{equation}\label{eq:alpha-lower}
\hat{\alpha} \ge \alpha - \frac{\alpha}{2} = \frac{\alpha}{2}, \qquad \text{hence} \qquad \frac{1}{\hat{\alpha}} \le \frac{2}{\alpha}.
\end{equation}

Next, we decompose the error in $\hat{x}$:
\begin{align}
\|\hat{x}-x^\star\|_M &\le \|\hat{y}-y\|_M + \left\| \frac{u^\top \hat{y}}{\hat{\alpha}}\hat{z} - \frac{u^\top y}{\alpha}z \right\|_M \notag \\
&\le \|\hat{y}-y\|_M + \frac{|u^\top \hat{y}|}{\hat{\alpha}}\|\hat{z}-z\|_M + \left|\frac{u^\top \hat{y}}{\hat{\alpha}} - \frac{u^\top y}{\alpha}\right|\|z\|_M. \label{eq:tri-full-corrected}
\end{align}
We now bound the two scalar factors. First, by Cauchy--Schwarz, \eqref{eq:alpha-lower}, and \eqref{eq:solver-guar}:
\[
\frac{|u^\top \hat{y}|}{\hat{\alpha}}
\le \frac{\|u\|_{M^{-1}}\|\hat{y}\|_M}{\hat{\alpha}}
\le \frac{2\|u\|_{M^{-1}}(\|y\|_M+\|\hat{y}-y\|_M)}{\alpha}
\le \frac{4\|u\|_{M^{-1}}\|y\|_M}{\alpha}.
\]
Combining this with $\|\hat{z}-z\|_M \le \epsilon_{\text{op}}\|z\|_M$ and $\|z\|_M=\|u\|_{M^{-1}}$, we get the second term bound:
\begin{equation}\label{eq:term2-full}
\frac{|u^\top \hat{y}|}{\hat{\alpha}}\|\hat{z}-z\|_M
\le
\frac{4\epsilon_{\text{op}}}{\alpha}\|u\|_{M^{-1}}^2\|y\|_M
=
\frac{4\epsilon_{\text{op}}(1-\alpha)}{\alpha}\|y\|_M.
\end{equation}

Next, for the third term, we split the difference:
\[
\left|\frac{u^\top \hat{y}}{\hat{\alpha}}-\frac{u^\top y}{\alpha}\right|
\le
\frac{|u^\top(\hat{y}-y)|}{\hat{\alpha}}
+
|u^\top y|\left|\frac{1}{\hat{\alpha}}-\frac{1}{\alpha}\right|.
\]
For the first part, using $1/\hat{\alpha} \le 2/\alpha$ and \eqref{eq:solver-guar}:
\begin{equation}\label{eq:coef1-full}
\frac{|u^\top(\hat{y}-y)|}{\hat{\alpha}}
\le
\frac{2\epsilon_{\text{op}}\|u\|_{M^{-1}}\|y\|_M}{\alpha}.
\end{equation}
For the second part, using $|\hat{\alpha}-\alpha| \le \epsilon_{\text{op}}(1-\alpha)$ and the bound $1/\hat{\alpha} \le 2/\alpha$ from \eqref{eq:alpha-lower},
\[
\left|\frac{1}{\hat{\alpha}}-\frac{1}{\alpha}\right|
=
\frac{|\hat{\alpha}-\alpha|}{\alpha\hat{\alpha}}
\le
\frac{2\epsilon_{\text{op}}(1-\alpha)}{\alpha^2}.
\]
Multiplying both parts by $\|z\|_M = \|u\|_{M^{-1}} = \sqrt{1-\alpha} \leq 1$ yields:
\begin{equation}\label{eq:term3-full}
\left|\frac{u^\top \hat{y}}{\hat{\alpha}}-\frac{u^\top y}{\alpha}\right|\|z\|_M
\le \paren{
\frac{2\epsilon_{\text{op}}}{\alpha}
+
\frac{2\epsilon_{\text{op}}}{\alpha^2} }\|y\|_M.
\end{equation}

Substituting \eqref{eq:solver-guar}, \eqref{eq:term2-full}, and \eqref{eq:term3-full} into \eqref{eq:tri-full-corrected}, and simplifing using $\alpha \le 1$:
\begin{align}\label{eq:final bound}
\|\hat{x}-x^\star\|_M
\le
\epsilon_{\text{op}}\|y\|_M
+
\frac{4\epsilon_{\text{op}}}{\alpha}\|y\|_M
+
\left(\frac{2\epsilon_{\text{op}}}{\alpha}+\frac{2\epsilon_{\text{op}}}{\alpha^2}\right)\|y\|_M
\le
\frac{9\epsilon_{\text{op}}}{\alpha^2}\,\|y\|_M.
\end{align}

We now convert this to the $\tilde{M}$-norm. Since $\tilde{M} \preceq M$, we have $\|v\|_{\tilde{M}} \le \|v\|_M$. Also, $\|x^\star\|_{\tilde{M}}^2 = b^\top \tilde{M}^{-1} b = b^\top M^{-1} b + \frac{1}{\alpha}(b^\top z)^2 \ge \|y\|_M^2$. Thus:
\[
\|\hat{x}-x^\star\|_{\tilde{M}} \le \|\hat{x}-x^\star\|_M \le \frac{9\epsilon_{\text{op}}}{\alpha^2}\|x^\star\|_{\tilde{M}}.
\]
Substituting $\epsilon_{\text{op}} = \frac{\epsilon \alpha^2}{9}$ (which is satisfied since $\alpha \ge 1/\kappa_u$) yields $\|\hat{x}-x^\star\|_{\tilde{M}} \le \epsilon \|x^\star\|_{\tilde{M}}$.
\end{proof}

\subsection{Leveraging subspace embeddings and Woodbury identity}\label{sec:linear-system-reductions-subsection2}

Next, as discussed in Section~\ref{sec:approach}, using the Woodbury matrix identity (Fact~\ref{fact:woodbury}), we first prove a general result that for any $B \in \R^{n \times d}$, in order to solve a linear system in $B^\top B + \nu I_{d}$ it suffices to solve a linear system in $BB^\top + \nu I_n$ plus moderate additional compute. 

As discussed in Section~\ref{sec:approach}, this is helpful because, as we will see in Lemma~\ref{lemma:ose-twice}, this enable us to leverage two $(s,b)$-embedding matrices $S \in \R^{s \times n}, T^\top \in \R^{s \times d}$ (Definition~\ref{def:sparse-embedding}) to reduce building a linear system solver for the $d \times d$ matrix $B^\top B + \nu I_{d}$ to building a linear system solver for an $s \times s$ matrix $SBTT^\top BS + \nu I_{s}$ for an appropriate sketch dimension $s \ll d$. The following result is similar to Lemma 25 of \citep{derezinski2026approaching}.  

\begin{algorithm2e}
\caption{$\textsc{TransposeSolve}(B, \nu, \epsilon, \delta, b, \cO)$}\label{alg:transpose-solve}
\KwInput{Matrix-vector access to $B \in \R^{n \times d}, \nu > 0, \epsilon >0, \delta \in (0, 1), b \in \R^n$.}
\KwInput{$\cO$ is an $\left(\frac{\epsilon \nu}{\|B\|_2^2}, \delta\right)$-linear system solver for $K \defeq B^\top B + \nu I_d$.}
 $u \gets B^\top b$\; 
 $\hat{z} \gets \cO(u)$\; 
 $\hat{y} \gets \frac{1}{\nu} (b - B \hat{z})$\; 
\textbf{Return }{$\hat{y}$}
\end{algorithm2e}

\begin{lemma}[Solving in the Transpose]
\label{lemma:change-order} 
There is an algorithm \[\textsc{TransposeSolve}(B, \nu, \epsilon, \delta, b, \cO)\] (Algorithm~\ref{alg:transpose-solve}) which takes in $B \in \R^{n \times d}$, $\nu > 0$, $\epsilon > 0$, $\delta \in (0,1)$, $b \in \mathbb{R}^n$, and an oracle $\cO$. If $\cO$ is an $\left(\frac{\epsilon \nu}{\|B\|_2^2}, \delta\right)$-linear system solver for $K \defeq B^\top B + \nu I_d$, then the algorithm outputs an $\epsilon$-approximate solution to $(BB^\top + \nu I_n)y = b$ with probability $1-\delta$. The algorithm makes one query to $\cO$, one matvec to $B$, and performs $O(n+d)$-additional work using $\tilde{O}(1)$-depth. 
\end{lemma}

\begin{proof}
Fix $b \in \R^n$ and let $M \defeq BB^\top + \nu I_{n}$. We wish to approximate $y^\star \defeq M^{-1} b$. 
By the Woodbury identity (Fact~\ref{fact:woodbury}), $M^{-1} = \frac{1}{\nu} (I_n - B(B^\top B + \nu I_d)^{-1}B^\top)$. 
Letting $K \defeq B^\top B + \nu I_{d}$, we define:
\[
z^\star \defeq K^{-1} B^\top b, \qquad y^\star = \frac{1}{\nu} (b - B z^\star).
\]
The algorithm computes $u \gets B^\top b$, queries the oracle $\hat{z} \gets \cO(u)$, and outputs $\hat{y} \defeq \frac{1}{\nu} (b - B \hat{z})$. 
Condition on the event that $\cO$ succeeds (which occurs with probability at least $1-\delta$), providing a solution $\hat{z}$ such that $\|\hat{z} - z^\star\|_K \le \epsilon_{\text{op}} \|z^\star\|_K$ for $\epsilon_{\text{op}} \defeq \frac{\epsilon \nu}{\|B\|_2^2}$.

Let $e_z \defeq \hat{z} - z^\star$. Then $\hat{y} - y^\star = -\frac{1}{\nu} B e_z$. To relate the error in the $M$-norm to the oracle error in the $K$-norm, observe that for any $v \in \R^d$:
\begin{align*}
\|Bv\|_M^2 &= v^\top B^\top M B v \\
&= v^\top B^\top (BB^\top + \nu I_{n}) B v \\
&= v^\top (B^\top B)(B^\top B + \nu I_{d}) v \\
&\le \|B^\top B\|_2 \cdot (v^\top K v) = \|B\|_2^2 \|v\|_K^2.
\end{align*}
Applying this with $v=e_z$ gives:
\begin{equation}\label{eq:transpose-error-step}
\|\hat{y} - y^\star\|_M = \frac{1}{\nu} \|B e_z\|_M \le \frac{\|B\|_2}{\nu} \|e_z\|_K \le \frac{\|B\|_2}{\nu} \epsilon_{\text{op}} \|z^\star\|_K.
\end{equation}
Next, we bound $\|z^\star\|_K$ in terms of $\|y^\star\|_M$. Note that $z^\star = B^\top y^\star$. Using the same logic as above:
\[
\|z^\star\|_K^2 = (y^\star)^\top B (B^\top B + \nu I_d) B^\top y^\star = (y^\star)^\top (BB^\top)(BB^\top + \nu I_n) y^\star \le \|B\|_2^2 \|y^\star\|_M^2.
\]
Substituting $\|z^\star\|_K \le \|B\|_2 \|y^\star\|_M$ into \eqref{eq:transpose-error-step}:
\[
\|\hat{y} - y^\star\|_M \le \frac{\|B\|_2^2}{\nu} \epsilon_{\text{op}} \|y^\star\|_M.
\]
By our choice of $\epsilon_{\text{op}} = \frac{\epsilon \nu}{\|B\|_2^2}$, we conclude $\|\hat{y} - y^\star\|_M \le \epsilon \|y^\star\|_M$.
\end{proof}

By combining this result with sparse oblivious subspace embeddings (Definition~\ref{def:sparse-embedding}), we obtain the following guarantee.

\begin{lemma}\label{lemma:ose-twice} There is a randomized algorithm 
\begin{align*}
    \textsc{SketchAndSolve}_{\xi_1, \xi_2, \xi_3}(B, \nu, k, \epsilon, \delta, b', \cO)
\end{align*}
which takes in $B\in \R^{n \times d}$, $\nu \geq \lambda > 0$, $k \in \mathbb{N}$, $\epsilon, \delta \in (0, 1), b' \in \R^d$, and an oracle $\cO$. Here, letting, $s = \tilde{O}((k + \log(1/\delta)))$, $b = \tilde{O}(\log^2(k/\delta) + \log^3(k/\delta))$, and $L = \ceil{\log(1/\epsilon)}$, $\cO$ is an
\begin{align*}
    \paren{\frac{\epsilon}{\normInline{B}_2^2}, \frac{\delta}{3L^2}}\text{-linear system solver for $B^\top B + \nu I_d$}. 
\end{align*} 
The algorithm is parameterized by three independent random seeds $\xi_1 \sim \cP_1, \xi_2 \sim \cP_2, \xi_3 \sim \cP_3$ such that $\xi_1 $ seeds the randomness used to sample a matrix $S \in \R^{s \times n}$ from an $(s, b, \delta/3, \lambda)$-efficient subspace embedding distribution for $\{C \in \R^{n \times d}: \normInline{C}_F^2 \leq k\lambda\}$ (i.e., conditioned on $\xi_1$, $S$ is fixed); $\xi_2$, seeds the randomness used to sample $T^\top \in \R^{s \times d}$ from an $(s, b, \delta/3, \lambda)$-efficient subspace embedding distribution for $\{C \in \R^{s \times d}: \normInline{C}_F^2 \leq k\lambda\}$  (i.e., conditioned on $\xi_2$, $T$ is fixed); and $\xi_3$ seeds the randomness of $\cO$ 
(i.e., for any $b'' \in \R^d$, $\cO(b'')$ is a random variable equal to $\cO_{\xi_3}(b'')$ for $\xi_3 \sim \cP_3$),\footnote{If $\cO$ is deterministic, $\xi_3$ can be omitted.} 

The algorithm makes $\tilde{O}(1)$ queries to $\cO_{\xi_3}$, matvec queries to $SB B^\top S^\top + \nu I_{s}$, matvec queries to $B^\top S^\top$ and $SB$, and matvec queries to $SBTT^\top B^\top S^\top + \nu I_{s}$, 
and does $\tilde{O}(s + d)$-additional work using $\tilde{O}(1)$-additional depth. Moreover, $\cP_1, \cP_2$ do not depend on $B, b$ and the random seeds $\xi_1, \xi_2$ can be sampled in $\tilde{O}(s(n+d))$-additional work using $\tilde{O}(1)$-additional depth.
\end{lemma}
\begin{proof} By Corollary~\ref{corrolary:subspace-embedding}, with probability $1-\delta/3$ over the draw of $\xi_1 \sim \cP_1$, 
\begin{align}\label{eq:first-sketch}
    B^\top S^\top S B + \nu I_{d} \approx_2 B^\top B  + \nu I_{d}\,.
\end{align}
Conditioning on this event, by Fact~\ref{lemma:basic-properties}, $\normInline{SB}_F^2 \leq 2\normInline{B}_F^2$. Again, by Corollary~\ref{corrolary:subspace-embedding}, we have that with probability $1-\delta/3$ over the draw of $\xi_2 \sim \cP_2$,  
\begin{align*}
    SB T T^\top B^\top S^\top + \nu I_{s} \approx_2 SB B^\top S^\top + \nu I_{s}. 
\end{align*}
Condition on such $\xi_1, \xi_2$ in the remainder of the proof and let $\epsilon_{\text{op}} \defeq \frac{\epsilon}{\normInline{B}_2^2}$ and $\delta_{\text{op}}\defeq \frac{\delta}{3L^2}$. Now, conditioning on $\xi_3$, suppose that the algorithm outputs $\cO^3(b)$, where
\begin{itemize}
    \item $\cO^1: b' \in \R^s \mapsto \textsc{PrecondRichardson}(SB B^\top S^\top + \nu I_s, 2, \epsilon_{\text{op}}, \delta_{\text{op}}, b', \cO_{\xi_3})$; 
    \item $\cO^2: b' \in \R^d \mapsto \textsc{TransposeSolve}( B^\top S^\top S B, \nu, \epsilon_{\text{op}}, \delta/(3L), b', \cO^1)$; and 
    \item $\cO^3: b' \in \R^d \mapsto \textsc{PrecondRichardson}(B^\top B + \nu I_d, 2, \epsilon, \delta/3, b', \cO^2)$ 
\end{itemize}
The following statements hold with probability $1-\delta/3$ over the draw of $\xi_3$.

\begin{itemize}
    \item Note that by Theorem~\ref{thm:preconditioned-richardson}, $\cO^1$ is an $(\epsilon_{\text{op}}, \delta/(3L))$-linear system solver for $SBB^\top S^\top + \nu I_s$. 
    \item Consequently, by Lemma~\ref{lemma:change-order}, $\cO^2$ is an $(\epsilon, L \delta_{\text{op}})$-linear system solver for $B^\top S^\top SB+ \nu I_d$.
    \item Hence, by Theorem~\ref{thm:preconditioned-richardson}, $\cO^3$ is an $(\epsilon, \delta)$-linear system solver for $B^\top B + \nu I_d$.  
\end{itemize}

By a union bound over the failure probabilities induced by $\xi_1, \xi_2, \xi_3$, the correctness guarantee follows. Finally, to bound the complexity, we combine the complexity guarantees of Theorem~\ref{thm:preconditioned-richardson} and Lemma~\ref{lemma:change-order} as well as the sparse oblivious embedding guarantee Corollary~\ref{corrolary:subspace-embedding}. 
\end{proof}

\subsection{Solving the embedded linear systems efficiently}\label{sec:putting-together}

In order to apply Lemma~\ref{lemma:ose-twice}, we must show how to compute an efficient linear system solver matrix $SBTT^\top B^\top S^\top$ arising in Lemma~\ref{lemma:ose-twice} as well as support matvecs to $SB$ and $SBT$ efficiently. Here, we discuss this in greater detail. 

In the remainder of this subsection, we fix $\xbar \in \R^d$, a failure probability $\delta \in (0,1)$, $A \in \R^{n \times d}$, $\lambda > 0$, $\epsilon'
 >0$, $B = \frac{1}{\sqrt{\epsilon'}} P_{\xbar, \epsilon'}^{1/2} A$, and $k \defeq \ceil{\normInline{B}_F^2/\lambda}$. Furthermore, we let $S \in \R^{s \times n}, T^\top \in \R^{s \times d}$ denote $(s,b)$-embeddings (Definition~\ref{def:sparse-embedding}). We will show that we can efficiently implement matvecs to $SB$ and $AT$, as well as construct $SBT$ explicitly. 

 First, we argue that we can efficiently support matvecs to $SB$. 

\begin{lemma}[Supporting matvecs to $SB$]\label{lemma:matvec-support} Given an explicit realization of $S$, one matvec to
$SB \in \R^{s \times d}$ can be implemented in $O(b\,\nnz{A})$-time, using one matvec to $A$, $\tilde{O}(bn)$-additional work, and $\tilde{O}(1)$-additional depth.
\end{lemma}
\begin{proof}
To support one matvec to $SB \in \R^{s \times d}$, consider a query
$(x,y) \in \R^{d} \times \R^{s}$. To compute $SBx$, note that we can first compute $u \gets Bx = \frac{1}{\sqrt{\epsilon'}} P_{\xbar, \epsilon'}^{1/2}Ax \in \R^n$ in
$\nnz{A}$-work and then compute $SBx = Su$ in $O(bn)$-work, leveraging that
each column of $S$ has at most $b$ nonzero entries.

To compute $y^\top SB$, note that we can first compute
$u^\top \gets y^\top S \in \R^{n}$ in $O(bn)$-work, again leveraging that each
column of $S$ has at most $b$ nonzero entries. Then, we can compute
$u^\top B = u^\top \frac{1}{\sqrt{\epsilon'}} P_{\xbar, \epsilon'}^{1/2}A \in \R^{d}$ in $\nnz{A}$-work.
\end{proof}

Second, we argue that we can efficiently and explicitly compute $AT$. 

\begin{fact}[Constructing $AT$]\label{fact:AT} 
Given an explicit realization of $T$, $AT \in \R^{n \times s}$ can be computed explicitly in in $O(s\,\nnz{A})$-work, using  $\tilde{O}(1)$-depth and $\tilde{O}(s)$-matvecs to $A$. 
\end{fact}
\begin{proof} This follows immediately, as $T \in \R^{d \times s}$ has $s$ columns. 
\end{proof}

Lastly, we argue that we can efficiently and explicitly compute $SBT$, which is useful for our downstream linear system solving algorithms (recall Theorems~\ref{thm:fmm} and~\ref{thm:svrg}). Note that the fulling guarantee leverages the \emph{sparsity} of the embeddings to obtain the total work bound. 

\begin{lemma}[Constructing $SBT$]\label{lemma:constructing SBT} Given an explicit $AT \in \R^{n \times s}$, $SBT \in \R^{s \times s}$ can be computed explicitly in $O(b^2\,\nnz{A} + s^2)$-work and $\tilde{O}(1)$-depth. \end{lemma} \begin{proof} Since $AT\in\R^{n \times s}$ is given as input and $\frac{1}{\sqrt{\epsilon'}} P_{\xbar, \epsilon'}^{1/2}$ is diagonal, in $O(\nnz {AT}) \leq O(b \nnz{A})$-time (recall Lemma~\ref{def:sparsity-preserved}) we can compute $BT \in \R^{n \times s}$. Next, to compute $SBT \in \R^{s \times s}$, we apply the following procedure. For each $i\in[n]$, the $i$-th column of $S$ has at most $b$ nonzeros at rows
$h_1(i),\dots,h_b(i)\in[s]$ with corresponding values
$v_1(i),\dots,v_b(i)\in\{-1/\sqrt{b}, 1/\sqrt{b}\}$. We compute $C = SBT \in \R^{s \times s}$ as follows: 
Initialize $C \gets 0 \in \R^{s\times s}$ and for each $i \in [n]$, for each $t\in[b]$ (for which $S_{h_t(i),\,i}=v_t(i)\neq 0$), update
        \[
        C_{h_t(i),:} \;\gets\; C_{h_t(i),:} + v_t(i)\,(BT)_{i,:}.
        \]
For any $r\in[s]$ and
$j\in[s]$,
\[
C_{rj} \;=\; (SBT)_{rj} \;=\; \sum_{i=1}^n S_{ri}(BT)_{ij}.
\]
The above algorithm adds $(BT)_{i,:}$ into exactly those rows $r$ for which
$S_{ri}\neq 0$, and scales by $S_{ri}$, hence the final $C$ equals $SBT$. For the runtime, each nonzero entry $(BT)_{ij}$ participates in at most $b$
updates (one per nonzero in column $i$ of $S$), so the total work is $O(b\,\nnz{BT})$. Using $\nnz{BT}=\nnz{AT}\le b\,\nnz{A}$ (recall Lemma~\ref{def:sparsity-preserved}), this is
$O(b^2\,\nnz{A})$-work. The initialization of the dense output $C\in\R^{s\times s}$
costs $O(s^2)$-work. Thus the overall work is $O(b^2\,\nnz{A} + s^2)$.
\end{proof}

\subsection{Putting it all together}\label{sec:putting-together-2}

Here, we combine the previous results to obtain the main results of this section. In the remainder of this section, we use $\tilde{O}(\cdot)$ notation to hide polylogarithmic factors in $n, d, 1/\epsilon, 1/\epsilon'$ and $1/\delta$. 

\begin{lemma}[FMM solver with fresh sparse oblivious subspace embeddings]\label{lemma:fmm-fresh} Let $A \in \R^{n \times d}$ with $\normInline{A}_{2 \to \infty} \leq 1$, $\xbar \in \mathbb{B}^d$, $\lambda \ge \tilde{\Omega}({\epsilon'}^{-1/3})$, $\epsilon, \epsilon' >0$, $\delta \in (0,1)$. For any $b \in \R^d, \nu > \lambda$, let 
\begin{align*}
    \cO''(b) = \textsc{RankOneSolve}\paren{\frac{1}{\epsilon'} P_{\xbar, \epsilon'} B, \frac{1}{\sqrt{\epsilon'}} A p_{\xbar, \epsilon'}, \lambda, \epsilon, \delta, b, \cO'_{\xi_1, \xi_2, \xi_3}},
\end{align*}
where 
\begin{align*}
    \cO'_{\xi_1, \xi_2, \xi_3}(b') = \textsc{SketchAndSolve}_{\xi_1, \xi_2, \xi_3}\paren{B, \nu, \frac{1}{\lambda \epsilon'}, \frac{\epsilon}{1 + \frac{n}{\epsilon' \lambda}}, \frac{\delta}{2}, b', \cO}. 
\end{align*}
and $\cO$ instantiates FMM (Lemma~\ref{thm:fmm}). Then, $\cO''$ is an $(\epsilon, \delta)$-linear system solver for 
\begin{align*}
    \frac{1}{\epsilon'} A^\top (P_{\xbar, \epsilon'} - p_{\xbar, \epsilon'}p_{\xbar, \epsilon'}^\top) A + \nu I_d, 
\end{align*}
and each query to $\cO''$ does $\tilde{O}({\epsilon'}^{-2/3}\, \nnz A + {\epsilon'}^{-2\omega/3})$-work using $\tilde{O}(1)$-depth and $\tilde{O}({\epsilon'}^{-2/3})$-matvecs to $A$.  
\end{lemma}
\begin{proof} First, let
\begin{align*}
    B = \frac{1}{\sqrt{\epsilon'}}P_{\xbar, \epsilon'}^{1/2} A, 
\end{align*}
and note that $\normInline{B}_F^2 \leq {\epsilon'}^{-1}$ (by the argument as in \eqref{eq:frobenius-norm-bound}) and consequently let us define $k := \normInline{B}_F^2/\lambda$. Moreover, observe that 
\begin{align*}
    \normInline{B}_2^2 \leq \frac{1}{\epsilon'} \normInline{A}_2^2 \leq  {\epsilon'}^{-1} \normInline{A}_F^2 \leq \frac{1}{\epsilon'} \sum_{i \in [n]} \normInline{A_{i, :}}^2_2 \leq \frac{n}{\epsilon'} \normInline{A}_{2\to \infty}^2 = \frac{n}{\epsilon'}. 
\end{align*}
Similarly,
\begin{align*}
    \normInline{Ap_{\xbar, \epsilon'}}_2^2 \leq \normInline{A}_2^2 \normInline{p_{\xbar, \epsilon'}}_{2}^2 \leq \frac{n}{\epsilon'}, 
\end{align*}
since $\normInline{p_{\xbar, \epsilon'}}_2 \leq \normInline{p_{\xbar, \epsilon'}}_1 = 1$. Thus, the correctness follows directly by invoking Lemma~\ref{lemma:ose-twice} and Lemma~\ref{lemma:rank-one-term-reduction}. Correspondingly, the complexity follows directly from Lemma~\ref{lemma:matvec-support}, Fact~\ref{fact:AT}, Lemma~\ref{lemma:constructing SBT},  and Lemma~\ref{thm:fmm}. 
\end{proof}

\begin{lemma}[SVRG solver with fresh sparse oblivious subspace embeddings]\label{lemma:svrg-fresh} Let $A \in \R^{n \times d}$ with $\normInline{A}_{2 \to \infty} \leq 1$, $\xbar \in \mathbb{B}^d$, $\lambda \ge \tilde{\Omega}({\epsilon'}^{-1/3})$, $\epsilon, \epsilon' >0$, $\delta \in (0,1)$. For any $b \in \R^d, \nu > \lambda$, let  
\begin{align*}
    \cO''(b) = \textsc{RankOneSolve}\paren{\frac{1}{\epsilon'} P_{\xbar, \epsilon'} B, \frac{1}{\sqrt{\epsilon'}}A p_{\xbar, \epsilon'}, \lambda, \epsilon, \delta, b, \cO'_{\xi_1, \xi_2, \xi_3}},
\end{align*}
where 
\begin{align*}
    \cO'_{\xi_1, \xi_2, \xi_3}(b') = \textsc{SketchAndSolve}_{\xi_1, \xi_2, \xi_3}\paren{B, \nu, \frac{1}{\lambda \epsilon'}, \frac{\epsilon}{1 + \frac{n}{\epsilon' \lambda}}, \frac{\delta}{2}, b', \cO}. 
\end{align*}
and $\cO$ instantiates SVRG (Lemma~\ref{thm:svrg}). Then, $\cO''$ is an $(\epsilon, \delta)$-linear system solver for 
\begin{align*}
    \frac{1}{\epsilon'} A^\top (P_{\xbar, \epsilon'} - p_{\xbar, \epsilon'}p_{\xbar, \epsilon'}^\top) A + \nu I_d, 
\end{align*}
and each query to $\cO''$ does $\tilde{O}({\epsilon'}^{-2/3}\, \nnz A + {\epsilon'}^{-4/3})$-work using $\tilde{O}({\epsilon'}^{-2/3})$-depth and $\tilde{O}({\epsilon'}^{-2/3})$-matvecs to $A$.  
\end{lemma}
\begin{proof} Correctness follows identically as the proof of Lemma~\ref{lemma:fmm-fresh}.  The complexity follows directly from Lemma~\ref{lemma:matvec-support}, Fact~\ref{fact:AT}, Lemma~\ref{lemma:constructing SBT},  and Lemma~\ref{thm:svrg}. 
\end{proof}

In order to obtain Theorems~\ref{thm:first-result} and Theorems~\ref{thm:second-result}, in the following section we show that it is actually possible to \emph{reuse} the same  $T$ across \emph{every} linear system solve in Lemma~\ref{lem:ell_2-ell_1-reduction-to-linear-solves}. This enables us to pre-compute $AT^\top$ \textit{once} in the very first iteration, using $\tilde{O}(s\, \nnz A)$-work, $\tilde{O}(s)$-matvecs to $A$ and $\tilde{O}(1)$-depth and reuse it in all future iterations. Correspondingly, we reduce the auxiliary work and depth in Lemmas~\ref{lemma:svrg-fresh} and~\ref{lemma:fmm-fresh}.   %
\section{Main results}\label{sec:sample_reuse}

In this section, we discuss how the sample reuse framework of \citep{jin2025reusingsamplesvariancereduction} immediately enables us to reuse the same seed $\xi_2$ across sequential invocations of Lemmas~\ref{lemma:fmm-fresh} and \ref{lemma:svrg-fresh} when used to instantiate the linear system solver for Lemma~\ref{lem:ell_2-ell_1-reduction-to-linear-solves}, and we use this to obtain our main results. Throughout this section, we use $\tilde{O}(\cdot)$ to hide polylogarithmic factors in the following parameters: $n, d, \epsilon^{-1}, \delta^{-1}, \lambda^{-1}, \nu^{-1}$. 

In order to directly apply the sample reuse framework of \citep{jin2025reusingsamplesvariancereduction}, we need to check the three key assumptions that enable their framework. 

First note that $\xi_2 \sim \cP_2$ in Lemmas~\ref{lemma:fmm-fresh} and \ref{lemma:svrg-fresh} is used to sample an $(s,b)$-embedding (Definition~\ref{def:sparse-embedding}) in Lemma~\ref{lemma:ose-twice}, meaning that $\cP_2$ is an oblivious sampling distribution (i.e., $\cP_2$ depends only on the sketch size $s = \Tilde{O}(\epsilon^{-2/3})$ and number of nonzeros $b = \Tilde{O}(1)$.) 

Second, note that the requirements of the outer loop in Lemma~\ref{lem:ell_2-ell_1-reduction-to-linear-solves}, which orchestrates the sequence of linear system solves, is $\ell_\infty$-robust in the sense of Definition 2.2 of \citep{jin2025reusingsamplesvariancereduction}, as illustrated by the following lemma. 

\begin{lemma}[$\ell_\infty$-robustness]\label{lemma:robustness} Suppose that $x' \in \R^d$ satisfies 
    \begin{align*}
        \innorm*{x' - \inparen*{ \epsilon' A^\top ( P_{\xbar, \epsprim} - p_{\xbar, \epsprim} p_{\xbar, \epsprim}^\top) A   + \nu I}^{-1} \ghat}_2 \le \alpha
    \end{align*}
for some $\epsilon', \nu, \alpha > 0, \xbar \in \R^d, A \in \R^{n \times d}$ and suppose that $\zeta \in \R^d$ satisfies $\normInline{\zeta}_\infty < \tau$. Then, 
    \begin{align*}
        \innorm*{(x' + \zeta) - \inparen*{ \epsilon' A^\top ( P_{\xbar, \epsprim} - p_{\xbar, \epsprim} p_{\xbar, \epsprim}^\top) A   + \nu I}^{-1} \ghat}_2 \le \alpha + \tau \sqrt{d}. 
    \end{align*}
\end{lemma}
\begin{proof} Note that for any $x \in \R^d$, $\normInline{x}_\infty \leq \normInline{x}_2 \leq \sqrt{d} \normInline{x}_\infty$. Consequently, we have 
    \begin{align*}
        &\innorm*{(x' + \zeta) - \inparen*{ \epsilon' A^\top ( P_{\xbar, \epsprim} - p_{\xbar, \epsprim} p_{\xbar, \epsprim}^\top) A   + \nu I}^{-1} \ghat}_2 \\
        &\leq   \innorm*{x' - \inparen*{ \epsilon' A^\top ( P_{\xbar, \epsprim} - p_{\xbar, \epsprim} p_{\xbar, \epsprim}^\top) A   + \nu I}^{-1} \ghat}_\infty + \normInline{\zeta}_2 \\
        &\leq \alpha + \tau \sqrt{d}. 
    \end{align*}
\end{proof}
Consequently, the requirements of Lemma~\ref{lem:ell_2-ell_1-reduction-to-linear-solves} are robust to any $\tau$-bounded $\ell_\infty$ perturbation, provided that the accuracy requirement of the linear system solves is correspondingly lowered by a factor of $\tau/\sqrt{d}$. 

Third, note that the complexities of Lemmas~\ref{lemma:fmm-fresh} and~\ref{lemma:svrg-fresh} depend at most \emph{polylogarithmically} on the required solver accuracy $\epsilon$ (i.e., in the language of \citep{jin2025reusingsamplesvariancereduction}, the linear system solvers can solve to \textit{high accuracy}). Consequently, lowering the required accuracy of the linear system solves by a factor of $\tau/\sqrt{d}$ comes at the cost of at most polylogarithmic overhead in the key problem parameters. 

Thus, by a straightforward application of Theorem 2.6 of \citep{jin2025reusingsamplesvariancereduction}, we see that the same random seed $\xi_2$ can be used across all sequential invocations of Lemmas~\ref{lemma:fmm-fresh} and~\ref{lemma:svrg-fresh} when used to instantiate the linear system solver for Lemma~\ref{lem:ell_2-ell_1-reduction-to-linear-solves}. 

Finally, we are prepared to prove our main results. First, using an FMM-based linear system solver, we prove the following guarantee. 

\begin{theorem}[Final FMM-based solver]\label{main-fmm} There is a randomized algorithm which, with probability $1-\delta$, solves the $\epsilon$-game in $A$ (Definition~\ref{def:l2l1_game}) in $\Tilde{O}(\epsilon^{-2/3} \nnz A + \epsilon^{-2(\omega +1)/3})$-work, using $\Tilde{O}(\epsilon^{-2/3})$ matvecs to $A$ and $\Tilde{O}(\epsilon^{-2/3})$-depth.    
\end{theorem}
\begin{proof} By Lemma~\ref{lem:ell_2-ell_1-reduction-to-linear-solves}, we see that it suffices to show that
for any $\xbar \in \R^d$ and $\epsilon, \delta > 0$ we can implement an $(\epsilon, \delta)$-linear system solver for
\begin{align*}
    \frac{1}{\epsilon'} A^\top (P_{\xbar, \epsprim} - p_{\xbar, \epsprim}p_{\xbar, \epsprim}^\top) A + \nu I 
\end{align*}
for $\nu \geq \lambda = \Tilde{\Omega}(\epsilon^{-1/3})$,  which runs in $\Tilde{O}(\nnz A + \epsilon^{-2\omega/3})$-work and $\Tilde{O}(1)$-parallel depth and makes at most 
$\Tilde{O}(1)$-matvec queries to $A$. 

To this end, let $k, s, b$ as in Lemma~\ref{lemma:fmm-fresh} and~\ref{lemma:ose-twice} and let $T_*^\top \in \R^{s \times d}$ be an $(s,b)$-embedding  (Definition~\ref{def:sparse-embedding}) with $s = \Tilde{O}(\epsilon^{-2/3})$ and $b = \Tilde{O}(1)$. Note that by Fact~\ref{fact:AT}, we can compute $AT_* \in \R^{n \times s}$ in $\tilde{O}(s\, \nnz{A})$-time, $\tilde{O}(1)$-parallel-depth, and $\tilde{O}(s)$-matvecs to $A$. 

By combining Lemma~\ref{lemma:robustness} and Theorem 2.2 of \citep{jin2025reusingsamplesvariancereduction}, the guarantees of Theorem~\ref{lemma:fmm-fresh} remain unchanged if the random seed $\xi_2$ is fixed across all invocations of $\Tilde{\cO}$. Consequently, by Lemmas~\ref{lemma:constructing SBT} and~\ref{lemma:matvec-support}, the complexity of Lemma~\ref{lemma:fmm-fresh} can be reduced to $\Tilde{O}(b\, \nnz{A}) = \Tilde{O}(\nnz{A})$-work and only $\Tilde{O}(1)$ matvecs to $A$, by reusing the precomputed $AT_*$ in all invocations. 
\end{proof}

Second, using an SVRG-based linear system solver, we prove the following guarantee. 

\begin{theorem}[Final SVRG-based solver]\label{main-svrg} There is a randomized algorithm which, with probability $1-\delta$, solves the $\epsilon$-game in $A$ (Definition~\ref{def:l2l1_game}) in $\Tilde{O}(\epsilon^{-2/3} \nnz A + \epsilon^{-2})$-work, using $\Tilde{O}(\epsilon^{-2/3})$ matvecs to $A$ and $\Tilde{O}(\epsilon^{-4/3})$-depth.    
\end{theorem}
\begin{proof} The proof follows identically as that of Theorem~\ref{main-fmm}, except that we leverage Lemma~\ref{lemma:svrg-fresh} in place of Lemma~\ref{lemma:fmm-fresh}. 
\end{proof}

Theorems~\ref{thm:first-result} and~\ref{thm:second-result} now follow immediately from Lemma~\ref{lemma:reduction} and Theorems~\ref{main-fmm} and~\ref{main-svrg}, respectively. %
\section{Conclusion}\label{sec:conclusion}

In this work, we obtain faster randomized algorithms for approximately computing maximum-margin separating hyperplanes for binary data classification. Our algorithms match the matvec complexity of prior work and offer improved parallel computational depth. A natural question left open by our work is whether it is possible to achieve the depth of Theorem~\ref{thm:first-result} while matching the total work of Theorem~\ref{thm:second-result}, perhaps using very recent work of \cite{karmarkar2026solving}. Another natural problem is to try to extend our techniques to obtain improved work complexities or depth for other important classes of matrix games, such as zero-sum games. Lastly, it would be interesting to explore whether the work complexities or depth obtained in Theorems~\ref{thm:first-result} or~\ref{thm:second-result} can be obtained deterministically. 

\section*{Acknowledgements}

Thank you to anonymous reviewers for their helpful feedback and thank you to LLMs for various writing advice, latex formatting assistance, and help with proof-reading algebraic derivations.
Aaron Sidford was supported in part
by a Microsoft Research Faculty Fellowship, NSF Grant CCF1955039, and a PayPal research award. Liam was supported by an Amazon AI PhD Fellowship.

\bibliographystyle{plainnat}

\begin{thebibliography}{64}
\providecommand{\natexlab}[1]{#1}
\providecommand{\url}[1]{\texttt{#1}}
\expandafter\ifx\csname urlstyle\endcsname\relax
  \providecommand{\doi}[1]{doi: #1}\else
  \providecommand{\doi}{doi: \begingroup \urlstyle{rm}\Url}\fi

\bibitem[Alman et~al.(2025)Alman, Duan, Williams, Xu, Xu, and
  Zhou]{alman2025more}
Josh Alman, Ran Duan, Virginia~Vassilevska Williams, Yinzhan Xu, Zixuan Xu, and
  Renfei Zhou.
\newblock More asymmetry yields faster matrix multiplication.
\newblock In \emph{Proceedings of the 2025 Annual ACM-SIAM Symposium on
  Discrete Algorithms (SODA)}, pages 2005--2039. SIAM, 2025.

\bibitem[Asi et~al.(2021{\natexlab{a}})Asi, Carmon, Jambulapati, Jin, and
  Sidford]{asi2021stochastic}
Hilal Asi, Yair Carmon, Arun Jambulapati, Yujia Jin, and Aaron Sidford.
\newblock Stochastic bias-reduced gradient methods.
\newblock \emph{Advances in Neural Information Processing Systems},
  34:\penalty0 10810--10822, 2021{\natexlab{a}}.

\bibitem[Asi et~al.(2021{\natexlab{b}})Asi, Carmon, Jambulapati, Jin, and
  Sidford]{hilal2021stochasticbiasreduced}
Hilal Asi, Yair Carmon, Arun Jambulapati, Yujia Jin, and Aaron Sidford.
\newblock Stochastic bias-reduced gradient methods.
\newblock In \emph{Proceedings of the 35th International Conference on Neural
  Information Processing Systems}, NIPS '21, Red Hook, NY, USA,
  2021{\natexlab{b}}. Curran Associates Inc.

\bibitem[Atkinson and Vaidya(1995)]{atkinson1995cutting}
David~S Atkinson and Pravin~M Vaidya.
\newblock A cutting plane algorithm for convex programming that uses analytic
  centers.
\newblock \emph{Mathematical Programming}, 69\penalty0 (1-3):\penalty0 1--43,
  1995.

\bibitem[Bach(2010)]{bach2010self}
Francis Bach.
\newblock Self-concordant analysis for logistic regression.
\newblock \emph{Electronic Journal of Statistics}, 4:\penalty0 384--414, 2010.

\bibitem[Beck and Teboulle(2003)]{beck2003mirrordescent}
Amir Beck and Marc Teboulle.
\newblock Mirror descent and nonlinear projected subgradient methods for convex
  optimization.
\newblock \emph{Operations Research Letters}, 31\penalty0 (3):\penalty0
  167--175, 2003.
\newblock ISSN 0167-6377.

\bibitem[Bubeck et~al.(2019)Bubeck, Jiang, Lee, Li, and
  Sidford]{bubeck2019highlyparallel}
S\'{e}bastien Bubeck, Qijia Jiang, Yin~Tat Lee, Yuanzhi Li, and Aaron Sidford.
\newblock \emph{Complexity of highly parallel non-smooth convex optimization}.
\newblock Curran Associates Inc., Red Hook, NY, USA, 2019.

\bibitem[Carmon and
  Hausler(2022)]{carmon2022distributionallyrobustoptimizationball}
Yair Carmon and Danielle Hausler.
\newblock Distributionally robust optimization via ball oracle acceleration.
\newblock In \emph{arXiv preprint arXiv:2203.13225}, 2022.

\bibitem[Carmon et~al.(2019)Carmon, Jin, Sidford, and Tian]{carmon2019variance}
Yair Carmon, Yujia Jin, Aaron Sidford, and Kevin Tian.
\newblock Variance reduction for matrix games.
\newblock In \emph{\cNIPS{2019}}, 2019.

\bibitem[Carmon et~al.(2020{\natexlab{a}})Carmon, Jambulapati, Jiang, Jin, Lee,
  Sidford, and Tian]{carmon2020acceleration}
Yair Carmon, Arun Jambulapati, Qijia Jiang, Yujia Jin, Yin~Tat Lee, Aaron
  Sidford, and Kevin Tian.
\newblock Acceleration with a ball optimization oracle.
\newblock In \emph{\cNIPS{2020}}, 2020{\natexlab{a}}.

\bibitem[Carmon et~al.(2020{\natexlab{b}})Carmon, Jin, Sidford, and
  Tian]{carmon2020coordinate}
Yair Carmon, Yujia Jin, Aaron Sidford, and Kevin Tian.
\newblock Coordinate methods for matrix games.
\newblock In \emph{\cFOCS{2020}}, 2020{\natexlab{b}}.

\bibitem[Carmon et~al.(2021)Carmon, Jambulapati, Jin, and
  Sidford]{carmon2021thinking}
Yair Carmon, Arun Jambulapati, Yujia Jin, and Aaron Sidford.
\newblock Thinking inside the ball: Near-optimal minimization of the maximal
  loss.
\newblock In \emph{\cCOLT{2021}}, 2021.

\bibitem[Carmon et~al.(2022)Carmon, Hausler, Jambulapati, Jin, and
  Sidford]{carmon2022optimalandadaptivemonteirosvaiter}
Yair Carmon, Danielle Hausler, Arun Jambulapati, Yujia Jin, and Aaron Sidford.
\newblock Optimal and adaptive monteiro-svaiter acceleration.
\newblock In \emph{Proceedings of the 36th International Conference on Neural
  Information Processing Systems}, NIPS '22, Red Hook, NY, USA, 2022. Curran
  Associates Inc.

\bibitem[Carmon et~al.(2023)Carmon, Jambulapati, Jin, Lee, Liu, Sidford, and
  Tian]{carmon2023resqueing}
Yair Carmon, Arun Jambulapati, Yujia Jin, Yin~Tat Lee, Daogao Liu, Aaron
  Sidford, and Kevin Tian.
\newblock Resqueing parallel and private stochastic convex optimization.
\newblock In \emph{\cFOCS{2023}}, pages 2031--2058. IEEE, 2023.

\bibitem[Carmon et~al.(2024{\natexlab{a}})Carmon, Jambulapati, Jin, and
  Sidford]{carmon2024whole}
Yair Carmon, Arun Jambulapati, Yujia Jin, and Aaron Sidford.
\newblock A whole new ball game: A primal accelerated method for matrix games
  and minimizing the maximum of smooth functions.
\newblock In \emph{\cSODA{2024}}, 2024{\natexlab{a}}.

\bibitem[Carmon et~al.(2024{\natexlab{b}})Carmon, Jambulapati, O'Carroll, and
  Sidford]{carmon2024extracting}
Yair Carmon, Arun Jambulapati, Liam O'Carroll, and Aaron Sidford.
\newblock Extracting dual solutions via primal optimizers.
\newblock \emph{arXiv preprint arXiv:2412.02949}, 2024{\natexlab{b}}.

\bibitem[Chenakkod et~al.(2024{\natexlab{a}})Chenakkod, Derezinski, and
  Dong]{chenakkod2024optimal}
Shabarish Chenakkod, Michal Derezinski, and Xiaoyu Dong.
\newblock Optimal oblivious subspace embeddings with near-optimal sparsity.
\newblock \emph{arXiv preprint arXiv:2411.08773}, 2024{\natexlab{a}}.

\bibitem[Chenakkod et~al.(2024{\natexlab{b}})Chenakkod, Derezinski, Dong, and
  Rudelson]{chenakkod2024optimal2}
Shabarish Chenakkod, Michal Derezinski, Xiaoyu Dong, and Mark Rudelson.
\newblock Optimal embedding dimension for sparse subspace embeddings.
\newblock In \emph{Proceedings of the 56th Annual ACM Symposium on Theory of
  Computing}, pages 1106--1117, 2024{\natexlab{b}}.

\bibitem[Clarkson and Woodruff(2017)]{clarkson2017low}
Kenneth~L Clarkson and David~P Woodruff.
\newblock Low-rank approximation and regression in input sparsity time.
\newblock \emph{Journal of the ACM (JACM)}, 63\penalty0 (6):\penalty0 1--45,
  2017.

\bibitem[Clarkson et~al.(2012)Clarkson, Hazan, and
  Woodruff]{clarkson2012sublinear}
Kenneth~L Clarkson, Elad Hazan, and David~P Woodruff.
\newblock Sublinear optimization for machine learning.
\newblock In \emph{Journal of the ACM (JACM)}, 2012.

\bibitem[Cohen(2016)]{cohen2016nearly}
Michael~B Cohen.
\newblock Nearly tight oblivious subspace embeddings by trace inequalities.
\newblock In \emph{Proceedings of the twenty-seventh annual ACM-SIAM symposium
  on Discrete algorithms}, pages 278--287. SIAM, 2016.

\bibitem[Cohen et~al.(2015)Cohen, Nelson, and Woodruff]{cohen2015optimal}
Michael~B Cohen, Jelani Nelson, and David~P Woodruff.
\newblock Optimal approximate matrix product in terms of stable rank.
\newblock \emph{arXiv preprint arXiv:1507.02268}, 2015.

\bibitem[Cohen et~al.(2018)Cohen, Kelner, Kyng, Peebles, Peng, Rao, and
  Sidford]{cohen2018solving}
Michael~B Cohen, Jonathan Kelner, Rasmus Kyng, John Peebles, Richard Peng,
  Anup~B Rao, and Aaron Sidford.
\newblock Solving directed laplacian systems in nearly-linear time through
  sparse lu factorizations.
\newblock In \emph{2018 IEEE 59th annual symposium on foundations of computer
  science (FOCS)}, pages 898--909. IEEE, 2018.

\bibitem[Cohen et~al.(2021)Cohen, Lee, and Song]{cohen2021solving}
Michael~B Cohen, Yin~Tat Lee, and Zhao Song.
\newblock Solving linear programs in the current matrix multiplication time.
\newblock \emph{Journal of the ACM (JACM)}, 68\penalty0 (1):\penalty0 1--39,
  2021.

\bibitem[Conn et~al.(2000)Conn, Gould, and Toint]{conn2000trust}
Andrew~R. Conn, Nicholas I.~M. Gould, and Philippe~L. Toint.
\newblock \emph{Trust Region Methods}.
\newblock MOS-SIAM Series on Optimization. SIAM, Philadelphia, PA, 2000.

\bibitem[Derezinski and Sidford(2026)]{derezinski2026approaching}
Michal Derezinski and Aaron Sidford.
\newblock Approaching optimality for solving dense linear systems with low-rank
  structure.
\newblock In \emph{Proceedings of the 2026 Annual ACM-SIAM Symposium on
  Discrete Algorithms (SODA)}, pages 925--938. SIAM, 2026.

\bibitem[Derezinski et~al.(2025)Derezinski, Musco, and
  Yang]{derezinski2025faster}
Michal Derezinski, Christopher Musco, and Jiaming Yang.
\newblock Faster linear systems and matrix norm approximation via multi-level
  sketched preconditioning.
\newblock In \emph{Proceedings of the 2025 Annual ACM-SIAM Symposium on
  Discrete Algorithms (SODA)}, pages 1972--2004. SIAM, 2025.

\bibitem[Doikov(2023)]{doikov2023minimizingquasiselfconcordant}
Nikita Doikov.
\newblock Minimizing quasi-self-concordant functions by gradient regularization
  of newton method, 2023.

\bibitem[Duchi et~al.(2012)Duchi, Bartlett, and
  Wainwright]{duchi2012randomized}
John~C Duchi, Peter~L Bartlett, and Martin~J Wainwright.
\newblock Randomized smoothing for stochastic optimization.
\newblock \emph{SIAM Journal on Optimization}, 22\penalty0 (2):\penalty0
  674--701, 2012.

\bibitem[Frostig et~al.(2015)Frostig, Ge, Kakade, and
  Sidford]{frostig2015regularizing}
Roy Frostig, Rong Ge, Sham Kakade, and Aaron Sidford.
\newblock Un-regularizing: approximate proximal point and faster stochastic
  algorithms for empirical risk minimization.
\newblock In \emph{International Conference on Machine Learning}, pages
  2540--2548. PMLR, 2015.

\bibitem[Golub and Overton(1988)]{golub1988convergence}
Gene~H Golub and Michael~L Overton.
\newblock The convergence of inexact chebyshev and richardson iterative methods
  for solving linear systems.
\newblock \emph{Numerische Mathematik}, 53\penalty0 (5):\penalty0 571--593,
  1988.

\bibitem[Golub and Varga(2007)]{golub2007chebyshev}
Gene~H Golub and Richard~S Varga.
\newblock Chebyshev semi-iterative methods, successive overrelaxation iterative
  methods, and second order richardson iterative methods.
\newblock \emph{Milestones in Matrix Computation-Selected Works of Gene H.
  Golub, with Commentaries}, pages 45--67, 2007.

\bibitem[Grigoriadis and Khachiyan(1995)]{grigoriadis1995sublinear}
Michael~D Grigoriadis and Leonid~G Khachiyan.
\newblock A sublinear-time randomized approximation algorithm for matrix games.
\newblock In \emph{Operations Research Letters}, 1995.

\bibitem[Jambulapati et~al.(2022)Jambulapati, Liu, and
  Sidford]{arun2022improvediterationcomplexities}
Arun Jambulapati, Yang~P. Liu, and Aaron Sidford.
\newblock Improved iteration complexities for overconstrained p-norm
  regression.
\newblock In \emph{Proceedings of the 54th Annual ACM SIGACT Symposium on
  Theory of Computing}, STOC 2022, pages 529--542, New York, NY, USA, 2022.
  Association for Computing Machinery.
\newblock ISBN 9781450392648.

\bibitem[Jambulapati et~al.(2024{\natexlab{a}})Jambulapati, Sidford, and
  Tian]{jambulapati2024closing}
Arun Jambulapati, Aaron Sidford, and Kevin Tian.
\newblock Closing the computational-query depth gap in parallel stochastic
  convex optimization.
\newblock In \emph{The Thirty Seventh Annual Conference on Learning Theory},
  pages 2608--2643. PMLR, 2024{\natexlab{a}}.

\bibitem[Jambulapati et~al.(2024{\natexlab{b}})Jambulapati, Sidford, and
  Tian]{jambulapati2024closingcomputationalquerygap}
Arun Jambulapati, Aaron Sidford, and Kevin Tian.
\newblock Closing the computational-query depth gap in parallel stochastic
  convex optimization.
\newblock In \emph{Annual Conference Computational Learning Theory},
  2024{\natexlab{b}}.

\bibitem[Jiang et~al.(2020)Jiang, Lee, Song, and
  Wong]{haotian2020improvedcuttingplane}
Haotian Jiang, Yin~Tat Lee, Zhao Song, and Sam Chiu-wai Wong.
\newblock An improved cutting plane method for convex optimization,
  convex-concave games, and its applications.
\newblock In \emph{Proceedings of the 52nd Annual ACM SIGACT Symposium on
  Theory of Computing}, STOC 2020, page 944–953, New York, NY, USA, 2020.
  Association for Computing Machinery.
\newblock ISBN 9781450369794.
\newblock \doi{10.1145/3357713.3384284}.

\bibitem[Jin et~al.(2026)Jin, Karmarkar, Sidford, and
  Wang]{jin2025reusingsamplesvariancereduction}
Yujia Jin, Ishani Karmarkar, Aaron Sidford, and Jiayi Wang.
\newblock Reusing samples in variance reduction.
\newblock In \emph{Algorithmic Learning Theory}, 2026.

\bibitem[Johnson and Zhang(2013)]{johnson2013accelerating}
Rie Johnson and Tong Zhang.
\newblock Accelerating stochastic gradient descent using predictive variance
  reduction.
\newblock \emph{Advances in neural information processing systems}, 26, 2013.

\bibitem[Karimireddy et~al.(2018)Karimireddy, Stich, and
  Jaggi]{karimireddy2018globallinearconvergencenewtons}
Sai~Praneeth Karimireddy, Sebastian~U. Stich, and Martin Jaggi.
\newblock Global linear convergence of newton's method without strong-convexity
  or lipschitz gradients.
\newblock In \emph{arXiv preprint arXiv:1806.00413}, 2018.

\bibitem[Karmarkar et~al.(2025)Karmarkar, O'Carroll, and
  Sidford]{karmarkar2025solvingzerosumgames}
Ishani Karmarkar, Liam O'Carroll, and Aaron Sidford.
\newblock Solving zero-sum games with fewer matrix-vector products.
\newblock In \emph{\cFOCS{2025}}, 2025.

\bibitem[Karmarkar et~al.(2026)Karmarkar, O'Carroll, and
  Sidford]{karmarkar2026solving}
Ishani Karmarkar, Liam O'Carroll, and Aaron Sidford.
\newblock Solving matrix games with near-optimal matvec complexity.
\newblock \emph{\cSTOC{2026}}, pages arXiv--2601, 2026.

\bibitem[Kornowski and Shamir(2025{\natexlab{a}})]{kornowski2024oracle}
Guy Kornowski and Ohad Shamir.
\newblock The oracle complexity of simplex-based matrix games: Linear
  separability and nash equilibria.
\newblock In \emph{\cCOLT{2025}}, 2025{\natexlab{a}}.

\bibitem[Kornowski and Shamir(2025{\natexlab{b}})]{kornowski2024oracleupdated}
Guy Kornowski and Ohad Shamir.
\newblock The oracle complexity of simplex-based matrix games: Linear
  separability and nash equilibria.
\newblock In \emph{arXiv preprint 2412.06990 [v3]}, 2025{\natexlab{b}}.

\bibitem[Lee et~al.(2015)Lee, Sidford, and Wong]{yintat2015fastercuttingplane}
Yin~Tat Lee, Aaron Sidford, and Sam Chiu-Wai Wong.
\newblock A faster cutting plane method and its implications for combinatorial
  and convex optimization.
\newblock In \emph{2015 IEEE 56th Annual Symposium on Foundations of Computer
  Science}, pages 1049--1065, 2015.
\newblock \doi{10.1109/FOCS.2015.68}.

\bibitem[Lin et~al.(2015)Lin, Mairal, and Harchaoui]{lin2015universal}
Hongzhou Lin, Julien Mairal, and Zaid Harchaoui.
\newblock A universal catalyst for first-order optimization.
\newblock \emph{Advances in neural information processing systems}, 28, 2015.

\bibitem[McCulloch and Pitts(1943)]{mcculloch1943logical}
Warren~S McCulloch and Walter Pitts.
\newblock A logical calculus of the ideas immanent in nervous activity.
\newblock In \emph{The bulletin of mathematical biophysics}, 1943.

\bibitem[Murty and Raghava(2016)]{murty2016kernel}
MN~Murty and Rashmi Raghava.
\newblock Kernel-based svm.
\newblock In \emph{Support vector machines and perceptrons: Learning,
  optimization, classification, and application to social networks}, pages
  57--67. Springer, 2016.

\bibitem[Nelson and Nguy{\^e}n(2013)]{nelson2013osnap}
Jelani Nelson and Huy~L Nguy{\^e}n.
\newblock Osnap: Faster numerical linear algebra algorithms via sparser
  subspace embeddings.
\newblock In \emph{2013 ieee 54th annual symposium on foundations of computer
  science}, pages 117--126. IEEE, 2013.

\bibitem[Nemirovski(1994)]{nemirovski1994parallel}
Arkadi Nemirovski.
\newblock On parallel complexity of nonsmooth convex optimization.
\newblock \emph{Journal of Complexity}, 10\penalty0 (4):\penalty0 451--463,
  1994.

\bibitem[Nemirovski(2004)]{nem04}
Arkadi Nemirovski.
\newblock Prox-method with rate of convergence o(1/t) for variational
  inequalities with lipschitz continuous monotone operators and smooth
  convex-concave saddle point problems.
\newblock \emph{SIAM Journal on Optimization}, 15\penalty0 (1):\penalty0
  229--251, 2004.

\bibitem[Nemirovskij and Yudin(1983)]{nemirovskij1983problem}
Arkadij~Semenovi{\v{c}} Nemirovskij and David~Borisovich Yudin.
\newblock Problem complexity and method efficiency in optimization.
\newblock In \emph{Wiley-Interscience}, 1983.

\bibitem[Nesterov(2005)]{nesterov2005smooth}
Yu~Nesterov.
\newblock Smooth minimization of non-smooth functions.
\newblock In \emph{Mathematical programming}, 2005.

\bibitem[Nesterov(2007)]{Nesterov2007dualextrapolation}
Yurii Nesterov.
\newblock Dual extrapolation and its applications to solving variational
  inequalities and related problems.
\newblock \emph{Mathematical Programming}, 109\penalty0 (2):\penalty0 319--344,
  Mar 2007.

\bibitem[Palaniappan and Bach(2016)]{palaniappan2016stochastic}
Balamurugan Palaniappan and Francis Bach.
\newblock Stochastic variance reduction methods for saddle-point problems.
\newblock In \emph{\cNIPS{2016}}, 2016.

\bibitem[Pan(1987)]{pan1987parallelmatrix}
Victor~Y. Pan.
\newblock Complexity of parallel matrix computations.
\newblock \emph{Theoretical Computer Science}, 54:\penalty0 65--85, 1987.

\bibitem[Pan and Reif(1985)]{pan1985efficientparallel}
Victor~Y. Pan and John~H. Reif.
\newblock Efficient parallel solution of linear systems.
\newblock In \emph{Proceedings of the 17th Annual ACM Symposium on Theory of
  Computing}, pages 143--152. ACM, 1985.

\bibitem[Rakhlin and Sridharan(2013)]{Rakhlin2013online}
A.~Rakhlin and K.~Sridharan.
\newblock Online learning with predictable sequences.
\newblock In \emph{Proceedings of the 26th Annual Conference on Learning
  Theory}, volume~30 of \emph{Proceedings of Machine Learning Research}, pages
  993--1019. PMLR, 2013.

\bibitem[Rosenblatt(1958)]{rosenblatt1958perceptron}
Frank Rosenblatt.
\newblock The perceptron: a probabilistic model for information storage and
  organization in the brain.
\newblock In \emph{Psychological review}, 1958.

\bibitem[Sarlos(2006)]{sarlos2006improved}
Tamas Sarlos.
\newblock Improved approximation algorithms for large matrices via random
  projections.
\newblock In \emph{2006 47th annual IEEE symposium on foundations of computer
  science (FOCS'06)}, pages 143--152. IEEE, 2006.

\bibitem[Sidford and Zhang(2023)]{sidford2023quantumspeedups}
Aaron Sidford and Chenyi Zhang.
\newblock Quantum speedups for stochastic optimization.
\newblock In \emph{Proceedings of the 37th International Conference on Neural
  Information Processing Systems}, NIPS '23, Red Hook, NY, USA, 2023. Curran
  Associates Inc.

\bibitem[Sun and Tran-Dinh(2019)]{sun2019generalized}
Tianxiao Sun and Quoc Tran-Dinh.
\newblock Generalized self-concordant functions: a recipe for newton-type
  methods.
\newblock \emph{Mathematical Programming}, 178\penalty0 (1-2):\penalty0
  145--213, 2019.

\bibitem[Vaidya(1989)]{vaidya1989new}
Pravin~M. Vaidya.
\newblock A new algorithm for minimizing convex functions over convex sets
  (extended abstract).
\newblock In \emph{FOCS}, pages 338--343, 1989.

\bibitem[Van Den~Brand et~al.(2021)Van Den~Brand, Lee, Liu, Saranurak, Sidford,
  Song, and Wang]{van2021minimum}
Jan Van Den~Brand, Yin~Tat Lee, Yang~P Liu, Thatchaphol Saranurak, Aaron
  Sidford, Zhao Song, and Di~Wang.
\newblock Minimum cost flows, mdps, and l1-regression in nearly linear time for
  dense instances.
\newblock In \emph{Proceedings of the 53rd Annual ACM SIGACT Symposium on
  Theory of Computing}, 2021.

\end{thebibliography}

\newpage

\appendix

\section{Reduction from separating hyperplane problem to $\ell_2$-$\ell_1$ games}\label{sec:extra}

\reduction*
\begin{proof}
To reformulate the $\rho$-separating hyperplane problem as an $\ell_2$-$\ell_1$ matrix game, note that~\eqref{eq:linear-separability} is equivalent to
\begin{align*}
\max_{w \in \ball^d} \min_{i \in [n]} l_i \langle w, \phi_i \rangle
=
\max_{w \in \ball^d} \min_{i \in [n]} [\Phi_\cD w]_i 
=
\max_{w \in \mathbb{B}^d} \min_{p \in \Delta^n} p^\top \Phi_\cD w 
=
- \min_{w \in \mathbb{B}^d} \max_{p \in \Delta^n} p^\top (-\Phi_\cD) w .
\end{align*}
By~\eqref{eq:linear-separability}, there exists $w^\star \in \mathbb{B}^d$ such that
\begin{align*}
\min_{p \in \Delta^n} p^\top \Phi_\cD w^\star
=
\min_{i \in [n]} [\Phi_\cD w^\star]_i
=
\min_{i\in[n]} l_i \langle w^\star,\phi_i\rangle
\ge \gamma_\cD.
\end{align*}
Therefore,
\begin{align*}
\max_{w \in \mathbb{B}^d} \min_{p \in \Delta^n} p^\top \Phi_\cD w
\ge \gamma_\cD,
\qquad\text{and consequently}\qquad
\min_{w \in \mathbb{B}^d} \max_{p \in \Delta^n} p^\top (-\Phi_\cD) w
\le -\gamma_\cD.
\end{align*}
Let $\hat{w}$ be a solution to the $\rho$-game of $- \Phi_\cD$. Then, 
\begin{align*}
\max_{p \in \Delta^n} p^\top (-\Phi_\cD)\hat{w}
\leq 
\min_{w \in \mathbb{B}^d} \max_{p \in \Delta^n} p^\top (-\Phi_\cD) w + 
\rho 
\leq -\gamma_\cD + \rho. 
\end{align*}
Consequently, we have $\max_{i \in [n]} [(-\Phi_\cD)\hat{w}]_i \leq -\gamma_\cD + \rho$, which implies that \begin{align*} l_i \langle \hat{w}, \phi_i \rangle = [\Phi_\cD \hat{w}]_i \geq \gamma_\cD - \rho, \text{ for all } i \in [n]. \end{align*} Thus, $\hat{w}$ is a solution to the $(\cD, \rho)$-separating hyperplane problem.
\end{proof}

\section{Linear algebraic properties}

\begin{fact}\label{lemma:condition-number} Suppose that $A \in \pdcone^{n}$ and $x'$ is an $\epsilon$-approximate solution to the linear system $Ax = b$ (Definition~\ref{def:linearsystemsolver}). Then, 
$\normInline{x - x'}_2 \leq \sqrt{\kappa(A)} \epsilon \normInline{x}_2$.
\end{fact}
\begin{proof} We have $\|x'-x\|_{A}
    \le
    \epsilon \|x\|_{A}.$
Since $A \succeq \lambda_{\min}(A)I$  it follows that
\begin{align*}
    \|x'-x\|_2 
    \le
    \frac{\|x'-x\|_{A}}
         {\sqrt{\lambda_{\min}(A)}} \le
    \epsilon
    \frac{\|x\|_A}
         {\sqrt{\lambda_{\min}(A)}}.
\end{align*}
Moreover, $\|x\|_{A}
    \le
    \sqrt{\lambda_{\max}(A)}
    \|x\|_2$. Thus, the lemma holds. 
\end{proof}

\begin{fact}\label{lemma:basic-properties} Suppose that $A, B \in \pdcone^{d}$ with $A \approx_c B$ for some $c \geq 1$. Then,
\begin{itemize}
    \item for any $x \in \R^d$, $\normInline{x}_A \approx_{\sqrt{c}} \normInline{x}_B$, 
    \item $\normInline{A}_F \approx_c \normInline{B}_F$, 
    \item for any $b \in \R^d$, $\normInline{B^{-1} b}_B \approx_{\sqrt{c}} \normInline{A^{-1} b}_A$, and 
    \item for any $b \in \R^d$, $\normInline{A^{-1} b - B^{-1} b}_B \leq (c-1) \normInline{B^{-1} b}_B$,
\end{itemize}
\end{fact}
\begin{proof}
We prove the statements one-by-one.
\begin{itemize}
    \item Since $c^{-1}A \preceq B \preceq cA$ implies $c^{-1}B \preceq A \preceq cB$, for any $x\in\R^d$ we have
    \[
    c^{-1}x^\top Bx \le x^\top Ax \le cx^\top Bx,
    \]
    which yields $\|x\|_A \approx_{\sqrt c} \|x\|_B$.

    \item From $c^{-1}A \preceq B \preceq cA$, the eigenvalues of $B$ are within a factor $c$ of those of $A$ (due to the Courant-Fischer theorem)
    Since
    $\|M\|_F^2 = \sum_i \lambda_i(M)^2$ for $M\succeq 0$, it follows that
    $\|A\|_F \approx_c \|B\|_F$.

    \item Inverting the Loewner order gives
    $c^{-1}A^{-1} \preceq B^{-1} \preceq cA^{-1}$. Therefore, for any $b\in\R^d$,
    \[
    b^\top B^{-1} b \approx_c b^\top A^{-1} b,
    \]
    which is equivalent to
    $\|B^{-1}b\|_B \approx_{\sqrt c} \|A^{-1}b\|_A$.

    \item Writing $b = Bx$, we have
    \[
    A^{-1}b - B^{-1}b = (A^{-1}B - I)x.
    \]
    Since the eigenvalues of $B^{1/2}A^{-1}B^{1/2}$ lie in $[c^{-1},c]$, we have
    $\|A^{-1}B - I\|_B \le c-1$, and
    \[
    \|A^{-1}b - B^{-1}b\|_B \le (c-1)\|x\|_B = (c-1)\|B^{-1}b\|_B.
    \]
\end{itemize}
\end{proof} 
\end{document}